\documentclass[12pt]{amsart}
\usepackage{amsmath,amssymb, amscd}
\usepackage[dvipdfmx]{graphicx}
\usepackage{youngtab}

\def\theequation{\thesection.\arabic{equation}}
\makeatletter
\@addtoreset{equation}{section}
\makeatother
\usepackage{enumerate}
\usepackage{color}
\usepackage[all]{xy}

\def\ii{{\sqrt{-1}}}

\def\tX{{\widetilde{X}}}

\def\hu{{\widehat{u}}}

\def\hphi{{\widehat{\phi}}}
\def\hzero{{\widehat{0}}}
\def\hzeta{\widehat{\zeta}}
\def\hpi{{\widehat{\pi}}}
\def\hE{{\widehat{E}}}

\def\ee{\mathrm e}

\def\fE{{\mathfrak{E}}}
\def\tu{{\tilde{u}}}
\def\tv{{\tilde{v}}}

\def\sn{\mathrm {sn}}
\def\cn{\mathrm {cn}}
\def\dn{\mathrm {dn}}
\def\al{\mathrm {al}}

\def\cL{\mathcal {L}}

\def\wt{\mathrm {wt}_\infty}

\def\nuI{{\nu^{I}}}

\def\tw{{\widetilde{w}}}
\def\tX{{\widetilde{X}}}

\def\CC{{\mathbb C}}
\def\ZZ{{\mathbb Z}}

\def\RR{{\mathbb R}}

\def\cH{{\mathcal{H}}}

\def\cS{{\mathcal{S}}}
\def\cK{{\mathcal{K}}}

\def\cJ{\mathcal{J}}

\def\fJ{\mathfrak{J}}
\def\fX{\mathfrak{X}}
\def\tfX{{\widetilde{\fX}}}

\def\PP{{\mathbb P}}

\def\Sp{{\mathrm{Sp}}}

\def\tv{{{\widetilde{v}}}}

\def\fs{{\mathfrak{s}}}

\def\nB#1{{B_{X_{\hzero}#1}}}

\def\nphi#1{{\phi_{X_{\hzero}#1}}}
\def\nhphi#1{{\hphi_{X_{\hzero}#1}}}
\def\nsigma#1{{\sigma_{X_{\hzero}#1}}}
\def\nnuIo#1{{\nu^{I\circ}_{X_{\hzero}#1}}}

\def\nnuI#1{{\nu^{I}_{X_{\hzero}#1}}}
\def\nnuII#1{{\nu^{II}_{X_{\hzero}#1}}}
\def\nomega#1{\omega_{X_{\hzero}#1}}
\def\nomegap#1{\omega^{{\prime}}_{X_{\hzero}#1}}
\def\nomegapI{\omega^{{\prime}-1}_{X_{\hzero}}}
\def\nomegapp#1{{\omega^{{\prime\prime}}_{X_{\hzero}#1}}}
\def\netap#1{\eta^{{\prime}}_{X_{\hzero}#1}}
\def\netapp#1{{\eta^{{\prime\prime}}_{X_{\hzero}#1}}}
\def\ncK{{\mathcal{K}_{X_{\hzero}}}}

\def\ntw#1{{\widetilde{w}_{X_{\hzero}#1}}}

\def\nw#1{{w}_{X_{\hzero}#1}}
\def\nL{{L}_{X_\hzero}}
\def\nchi{{\chi}_{X_\hzero}}

\def\sphi#1{{\phi_{X_s#1}}}

\def\ssigma#1{{\sigma_{X_s#1}}}
\def\snuIo#1{{\nu^{I\circ}_{X_s#1}}}

\def\snuI#1{{\nu^{I}_{X_s#1}}}
\def\snuII#1{{\nu^{II}_{X_s#1}}}
\def\somega#1{\omega_{X_s#1}}
\def\somegap#1{\omega^{{\prime}}_{X_s#1}}
\def\somegapI{\omega^{{\prime}-1}_{X_s}}
\def\somegapp#1{{\omega^{{\prime\prime}}_{X_s#1}}}
\def\setap#1{{\eta^{{\prime}}_{X_s#1}}}
\def\setapp#1{{\eta^{{\prime\prime}}_{X_s#1}}}
\def\scK#1{{\mathcal{K}_{X_s}}}
\def\stw{{\widetilde{w}_{X_s}}}
\def\stws{{\widetilde{w}_{X_s\mathfrak{s}}}}
\def\sw{{w}_{X_s}}
\def\sL{{L}_{X_s}}
\def\schi{{\chi}_{X_s}}

\newtheorem{thm}{Theorem}[section]
\newtheorem{defn}[thm]{Definition}

\newtheorem{prop}[thm]{Proposition}

\newtheorem{rem}[thm]{Remark}
\newtheorem{lem}[thm]{Lemma}

\def\dfrac#1#2{{\displaystyle\frac{#1}{#2}}}

\begin{document}

\title{The sigma function over a family of
cyclic trigonal curves with a singular fiber}

\author{Yuri Fedorov, Jiyro Komeda, Shigeki Matsutani, Emma Previato,\\
and Kazuhiko Aomoto}

\maketitle

\begin{abstract}
In this paper
we investigate the behavior of the
sigma function over the  family of cyclic trigonal
curves $X_s$ defined  by
the equation $y^3 =x(x-s)(x-b_1)(x-b_2)$ in the affine $(x,y)$ plane,
for $s\in D_\varepsilon:=\{s \in \CC | |s|<\varepsilon\}$.
We compare the sigma function over the punctured disc
$D_\varepsilon^*:=D_\varepsilon\setminus\{0\}$
with the extension
over $s=0$ that specializes to the
sigma function of the normalization  $X_\hzero$ of the singular curve
$X_{s=0}$ by investigating explicitly the behavior of a basis of the first  
algebraic de Rham cohomology group and its period integrals. 
We demonstrate, using modular properties, that sigma, unlike the theta
function, has a limit. In particular, we obtain the limit of the theta
characteristics and
an explicit description of the theta divisor translated by the
Riemann constant. 
\end{abstract}

\noindent
{\bf{
2010 MSC:}}
{\it{Primary}}: 
32G20, 
14H45. 
{\it{Secondary}}: 
14H42, 
14H40, 
14H10. 

\noindent

\bigskip

\tableofcontents

\section{Introduction}\label{sec:intro}


The study of the Jacobi variety or the theta function over
families of curves with singular members has a long history
and is still very active.
Classically, Clebsh and Gordan 
introduced generalized theta functions 
for specific degenerations of
   hyperelliptic curves
\cite[\S 81]{CG}.
Kodaira classified the singular fibers  of elliptic surfaces 
  \cite{Ko}.
Igusa showed that there exists a (generalized) Jacobian fibration over a
family of curves with a finite number of degenerate 
members which have at most nodal singularities \cite{Igusa}.

However, in our work we focus on the analytic limit of a specific section of
(a translate of)
the theta line bundle, which turned out to be
 an important tool in integrable dynamics and
number theory (modular aspects of Riemann period matrices).
Our goal is to look at a family of Jacobians,
which we call a Jacobian fibration for simplicity
even though our general fiber has dimension
three,
while the central fiber has dimension two (therefore the
definition of fibration does not apply)
over a family of curves with a central fiber that has a
nodal singularity, and extend (a translate of) the theta line bundle over the
Jacobian of the normalized central fiber; 
for this, we use a section, known as sigma function. 

Based on classical constructions by Weierstrass and Klein, 
and further study by Baker 
  \cite{Wei54, Klein86, Baker97, Baker98},
Weierstrass' (elliptic) sigma function was generalized to non-singular curves
in Weierstrass canonical form; notably, unlike the theta function,
 sigma obeys an addition rule that can be expressed
in terms of meromorphic functions on the curve 
\cite{BEL:1999, EEL, EG, KMP18, O2009}; moreover, Buchstaber and Leykin 
investigated the behavior of   the sigma functions 
in moduli, under the heat equation and  the Gauss-Manin connection
\cite{BuL,EGOY}.
Using these results, in \cite{BeL,BEN}
the sigma function is analyzed over
a degenerating family of hyperelliptic curves.

In this paper, we investigate the behavior of the sigma function
of a degenerating family of  trigonal curves $X_s$, given by affine equation
$y^3 = x(x-s)(x-b_1)(x-b_2)$ for $s$ in the unit disc
$$
D_\varepsilon=\{ s\in \CC \ |\ |s|<\varepsilon\}.
$$
Our strategy is the following: in \cite{EEMOP07, MP15},
we obtained  
explicit properties  of the sigma function 
$\ssigma{}$ for the non-singular curve $X_s$ over the punctured disc 
$s \in 
D_\varepsilon^* = D_\varepsilon\setminus \{0\}$; 
in \cite{MK,KMP18}, we analyzed
 $\nsigma{}$ for the normalized curve $X_\hzero$ of $X_0=X_{s=0}$ given by
$y^3 = x^2(x-b_1)(x-b_2)$.
Now we consider the degenerating family of curves
$$
\fX:=\{(x,y,s) \ |\ (x,y) \in X_s, s \in D_\varepsilon\},
$$
we will exhibit the first
algebraic de Rham cohomology groups, whose generators are given by
 first and  second-kind differentials, as well as
their period matrices, for $X_s$ when $s \in D_\varepsilon^*$,
and for $X_\hzero$. We show the precise behavior
of the integrals  in
Appendices A and B. Finally, in Section 4 we compare these
objects over the non-singular
fiber $X_s$  to those of $X_\hzero$ at $s=0$.
Using this analysis, we construct the sigma functions
$\ssigma{}$ and $\nsigma{}$ of $X_s$ $(s\in D_s^*)$ and 
of $X_\hzero$. The relationship of the sigma function with 
the algebraic functions of the curve, particularly with the al function, which
we had also previously constructed for the trigonal cyclic case \cite{MP15}
turns 
out to be essential.

Using the fact that a section  of the divisor that gives the
principal polarization with the given modular behavior over the family is
unique, 
we view the sigma functions $\ssigma{}$ and $\nsigma{}$
as the canonical bases of the (translate) 
theta line bundles $\cL_{\cJ_s}$ and $\cL_{\cJ_\hzero}$ 
on the 
Jacobi varieties, $\cJ_s$ and $\cJ_\hzero$.
Thus we obtain an explicit 
extension in the limit $s\longrightarrow 0$ from  $\cL_{\cJ_s}$,
 $s\in D_\varepsilon^*$,
to $\cL_{\cJ_\hzero}$, and produce an explicit  line bundle over
 the Jacobian fibration of the family.
In the family, using the sigma functions,
we provide the connection to
the Jacobian of the desingularization $X_\hzero$
over the central fiber instead of
 a generalized Jacobian considered by Igusa \cite{Igusa}.
Recent progress on the study of the modular structure of the 
sigma functions by Eilbeck, Gibbons, \^Onishi and Yasuda \cite{EGOY},
based on the investigation of Buchstaber and Leykin \cite{BuL},
enables us to find the behavior of the sigma function as the 
period lattice varies. Using their results,  
we explicitly compare the structure of the 
limit of $\cL_{\cJ_s}$ with that of $\cL_{\cJ_\hzero}$
 and the ramified covering 
$\widetilde{D_{\varepsilon}^*}$ of $D_{\varepsilon}^*$
given by the  cyclic group of order three. The definition of sigma involves
theta characteristics; while the parity behavior of theta characteristics over
families is well-known \cite{atiyah, mumford},  we need to compute explicitly
the limit of the Riemann constant, which for our curves is translated by a
multiple of the divisor at infinity on the affine part of the curve; in
particular, we observe the behavior of the Weierstrass semigroup going from
symmetric to non-symmetric under degeneration.

Since this degeneration can be regarded as a higher-genus
version of the elliptic case,   type IV in
Kodaira's classification,
we apply the  technique of this paper
to the case of such  degenerate family of  elliptic curves in
Appendix \ref{AppdxC}. 
We make use of a very simple expression for the elliptic al function,
whereas for the present trigonal case 
defining al requires an elaborate configuration of triple
covers of the Jacobian. We
give  an explicit description of
the behavior of sigma for the degeneration of type IV, which 
had not previously appeared as far as we know.

\bigskip

The contents of the paper are as follows.
Section 2 gives a review of the sigma function $\ssigma{}$
of $X_s$ for $s\in D_\varepsilon^*$.
In Section 3, at $s=0$,
by considering the normalization  $X_\hzero$
of the singular curve $X_0$ and the 
Jacobian $\cJ_\hzero$,
we give  properties of the sigma function $\nsigma{}$.
In Section 4,
we investigate the degeneration explicitly and
present our main theorem.
Appendix A due to Kazuhiro Aomoto
is devoted to  the study of the integrals associated with 
the period matrices over the degeneration.
Using the results in Appendix A, Appendix B gives the explicit 
behavior of the period matrices in the limit $s\to0$, 
the crux of  this paper.
Appendix \ref{AppdxC} gives the behavior of
the Weierstrass sigma function over the degenerating
family of elliptic curves which is classified as type IV by
Kodaira \cite{Ko}.

\bigskip

\textbf{Shorthand.} Throughout the paper, the symbol $d_{\ge n}(t)$ ($d_{>
  n}(t)$, respectively) denotes a formal power series of order $\ge n$ ($>n$,
  resp.) in the variable $t$ (single or multi-variable).

\textbf{Acknowledgment} {The third named author 
thanks Tadashi Ashikaga for 
helpful and crucial comments, which led the Appendix C.
Further he is also grateful to Chris Eilbeck, Victor Enolskii and 
Yoshihiro \^Onishi 
for critical discussions and comments, and 
Takeo Ohsawa and Hajime Kaji
for valuable comments.
The third named author thanks the participants in Numadu-Shizuoka Kenkyukai
and, specially, its organizer Yoshinori Machida.
The second and third named authors
 were supported by
the Grant-in-Aid for Scientific 
Research (C) of Japan Society for the Promotion
of Science, 
Grant No. 18K04830 and Grant No. 16K05187 respectively.
The first and the third named authors thank Victor Enolskii, Julia 
Bernatska, and Tony Shaska for
their hospitality of the conference at Kiev August, 2018.
The authors thank the anonymous reviewer 
for his/her helpful comments.}

\section{The sigma function of $y^3 = x(x-s)(x-b_1)(x-b_2)$, $(s\neq 0)$ }

In this section, we review the properties of the sigma function of
a non-singular cyclic trigonal curve of genus three
following the papers \cite{EEMOP07, MP08, MP15}.

\subsection{Basic properties of $X_s$}
We consider the non-singular cyclic trigonal curve $X_s$ of genus three,
$g=3$, given by
$$
X_s: =
\Bigr\{(x, y) \ \Bigr| \ 
\begin{array}{rl}y^3= x(x - s) (x - b_1) (x - b_2)=:f(x)
\end{array}
\Bigr\} \cup \infty,
$$
and its affine ring,
$$
R_s:=\CC[x,y]/\left(y^3- x(x - s) (x - b_1) (x - b_2)\right).
$$
Here we assume that $b_0:=0$, $b_1$, $b_2$ and $b_3:=s$
are mutually distinct complex
numbers, in particular $s\neq0$.
The curve is given by a ramified cover of $\PP$,
\begin{equation*}
\xymatrix{ X_s \ar[d]^{\pi_2} \ar[r]^{\pi_1}& \PP \\
           \PP& }
\end{equation*}
$$
\pi_1(P) = x, \quad 
\pi_2(P) = y.
$$
The finite branch points of $\pi_1$ are denoted by
\begin{equation}
B_0:=(0,0),\quad B_1:=(b_1,0),\quad B_2:=(b_2, 0),\quad B_3:=B_s=(s,0).
\label{eq:Bs}
\end{equation}
The curve $X_s$ has an automorphism 
 $\hzeta_3: X_s \to X_s$ given by
$\hzeta_3(x,y)=(x, \zeta_3 y)$ for $\zeta_3 :=$ $\ee^{2\pi \ii/3}$.

The point $\infty$ in $X_s$ is a Weierstrass point.
The natural weight of $R_s$ is assigned as
$\wt(x)=-3$, $\wt(y)=-4$, since using the local parameter $t$ at $\infty$,
we have 
$$
x= \frac{1}{t^3}(1+d_{\ge 1}(t)), \quad y = \frac{1}{t^4}(1+d_{\ge 1}(t)).
$$
The Weierstrass non-gap sequence at $\infty$
is given by the following table.

\begin{table}[htb]
Table 1: Weierstrass non-gap sequence of $X_s$\\
  \begin{tabular}{r|ccccccccccccc}
$-\wt$ &0 &1 & 2 & 3 & 4 & 5 & 6 & 7 & 8 & 9 & 10 & $\cdots$  \\
\hline
$\sphi{}$ &
1 & - & - & $x$ & $y$ & - &$x^2$ & $x y$& 
$y^2$ &$x^3$ & $x^2y$ & $\cdots$ 
  \end{tabular}
\end{table}

We denote the corresponding basis of monic monomials 
by $\phi_{X_s}$, i.e.,
$\sphi0=1$, $\sphi1=x$, 
$\sphi2= y$, $\sphi3 = x^2$, $\cdots$.
As a $\CC$-vector space, we have the decomposition of $R_s$,
$$
  R_s = \bigoplus_{i=0} \CC \sphi{i}.
$$ 
Corresponding to the non-gap sequence, we have the numerical semigroup
$H_s:=\{ 3a+4b\}_{a,b\in\ZZ_{\ge0}}$ $:=\langle3,4\rangle$,
where $\ZZ_{\ge0}$ is the set of non-negative integers, i.e.,
$$
   H_s = \{0, 3, 4, 6, 7, \cdots\}, \quad
   L_s = \ZZ_{\ge0} \setminus H_s = \{1,2,5\}.
$$
The numerical semigroup $H_s$ is  related to the
 Young diagram 
$$
\Lambda_{(3,1,1)}=
\yng(3,1,1)
$$
because $(1,2,5)-(0,1,2)=(1,1,3)$.

\subsection{Differentials and Abelian integrals of $X_s$}
The holomorphic one-forms or the differentials of the first kind
on $X_s$ are given by
$$
         \snuI{} :=\begin{pmatrix}\snuI1\\ \snuI2\\ \snuI3\end{pmatrix}
:= {}^t\left(\frac{d x}{3y^2},\frac{x d x}{3y^2},\frac{d x}{3y}\right)
= {}^t\left(\frac{\sphi0d x}{3y^2},\frac{\sphi1d x}{3y^2},
\frac{\sphi2d x}{3y^2}\right).
$$
The divisor of $\nuI$, $\mathrm{Div}(\nuI_i)$, is linearly equivalent
to $(2g-2)\infty=4\infty$.
The differentials of the second kind (holomorphic except at $\infty$) are
$$
\snuII{} :=\begin{pmatrix}\snuII1\\ \snuII2 \\ \snuII3 \end{pmatrix},
\quad
    \snuII1 :=- \frac{(5x^2-3\lambda_{3}x +\lambda_2)d x }{3y}, \quad
$$
$$
    \snuII2 := -\frac{2x d x}{3y}, \quad, \snuII3 := -\frac{x^2d x}{3y^2},
$$
where $\lambda_{3}=-s-b_1-b_2$ and $\lambda_2=s(b_1+b_2)+b_1b_2$.
The set of  $\snuI{}$ and $\snuII{}$ gives a basis of the first 
algebraic de Rham cohomology of $X_s$.

The automorphism $\hzeta_3$ of $X_s$ also acts on the one-forms,
$$
\hzeta_3(\snuI{1}) = \zeta_3\snuI{1},\quad
\hzeta_3(\snuI{2}) = \zeta_3\snuI{2},\quad
\hzeta_3(\snuI{3}) = \zeta_3^2\snuI{3},
$$
$$
\hzeta_3(\snuII{1}) = \zeta_3^2\snuII{1},\quad
\hzeta_3(\snuII{2}) = \zeta_3^2\snuII{2},\quad
\hzeta_3(\snuII{3}) = \zeta_3\snuII{3}.
$$

To define the Abelian integrals, we first consider 
the Abelian covering $\tX_s$ of $X_s$, 
namely the abelianization of the quotient space of path space 
$\mathrm{Path}(X_s)$ divided
by  homotopy equivalence with respect to the fixed point $\infty$. 
There are a natural projection 
$\kappa_X: \tX_s \to X_s$ 
($\kappa(\gamma_{P,\infty})=P$ for a path
$\gamma_{P,\infty}\in \tX_s$ from 
$\infty$ to $P\in X_s$)
and a natural embedding $\iota_X : X_s \to \tX_s$.
The Abelian integral of the cyclic trigonal curve is defined as 
$$
\stw: \tX_s \to \CC^3, \quad  
\left(\stw(P):=\int_{\gamma_{P,\infty}} \snuI{} \right),
$$
$$
	\stw: \cS^k(\tX_s) \to \CC^3, \quad
	\stw(\gamma_{\infty,P_1}, \gamma_{\infty,P_2},\ldots,
 \gamma_{\infty,P_k}) :=\sum_{i=1}^k\stw(\gamma_{\infty,P_i}),
$$
where $\cS^k \tX_s$ is the $k$-th symmetric product of $\tX_s$.
The Abel-Jacobi theorem says that 
$\stw: \cS^3\tX_s \to \CC^3$ is a birational correspondence.

For the point $P_i$ near $\infty$ with a local parameter
$t_i$ $(i=1,2,3)$, the variable
$u_j=$ $\displaystyle{\sum_{i=1}^3 \int_{\gamma_{P_i,\infty}} \snuI{j}}$
$(j=1, 2, 3)$ has the expansion,
\begin{equation}
\begin{split}
u_1&=\frac{1}{5}(t_1^5+t_2^5+t_3^5)(1+d_{\ge 1}(t_1,t_2,t_3)),\quad
u_2=\frac{1}{2}(t_1^2+t_2^2+t_3^2)(1+d_{\ge 1}(t_1,t_2,t_3)),\\
u_3&=(t_1+t_2+t_3)(1+d_{\ge 1}(t_1,t_2,t_3)).
\label{eq:u_i}
\end{split}
\end{equation}
The weight on the ring $R_s$ induces the weight of the components of the
vectors $u$ in $\CC^3$,
\begin{equation}
\wt(u_1) = 5,\quad \wt(u_2) = 2 \quad 
\wt(u_3) = 1.
\label{eq:wt_u_i}
\end{equation}

\subsection{Periods of $X_s$}
The standard homology basis $(\alpha_1, \alpha_2, \alpha_3, 
\beta_1, \beta_2, \beta_3)$ satisfying the relations
$$
\langle \alpha_i,\beta_j\rangle=\delta_{ij},\quad
\langle \alpha_i,\alpha_j\rangle=0,\quad
\langle \beta_i,\beta_j\rangle=0\quad
$$
is illustrated in Figure \ref{fig:H1X3}.

\begin{figure}[ht]
\begin{center}
\includegraphics[width=0.45\textwidth]{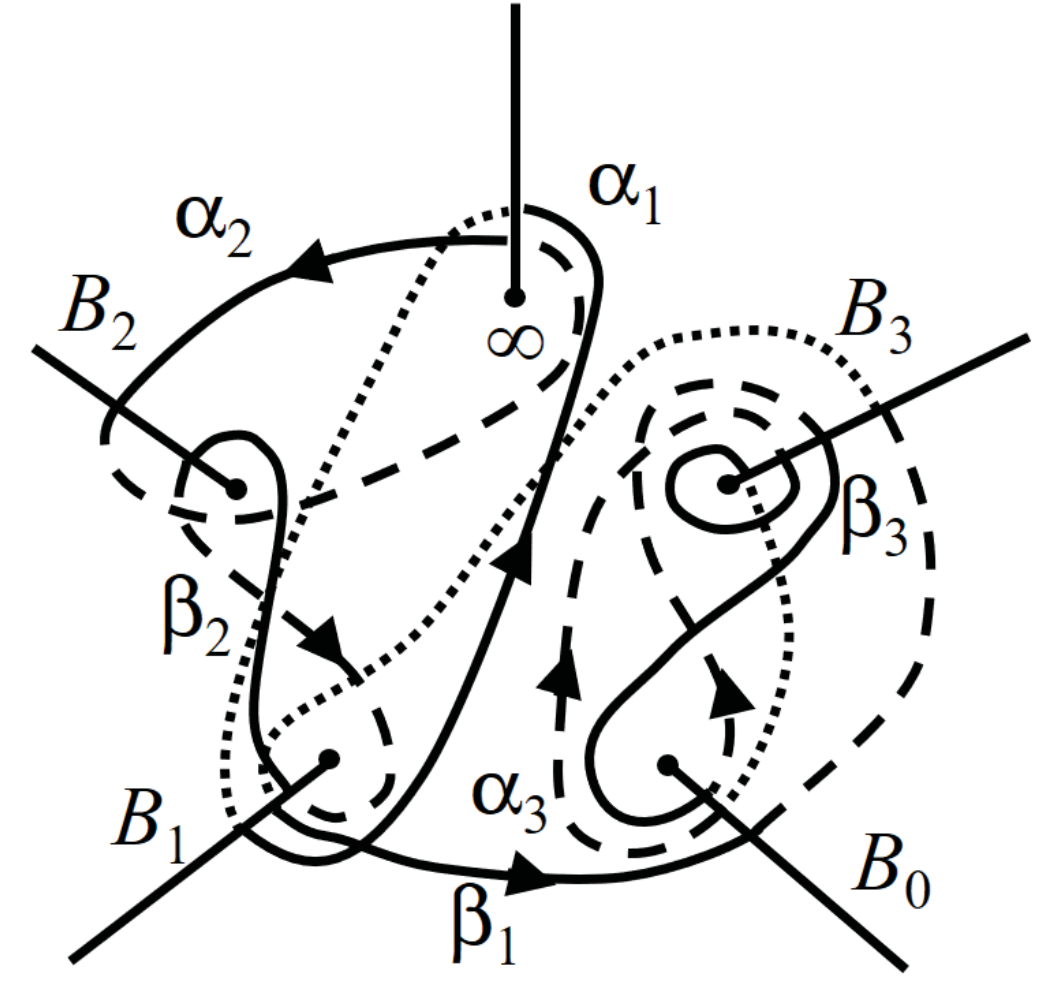}
\includegraphics[width=0.40\textwidth]{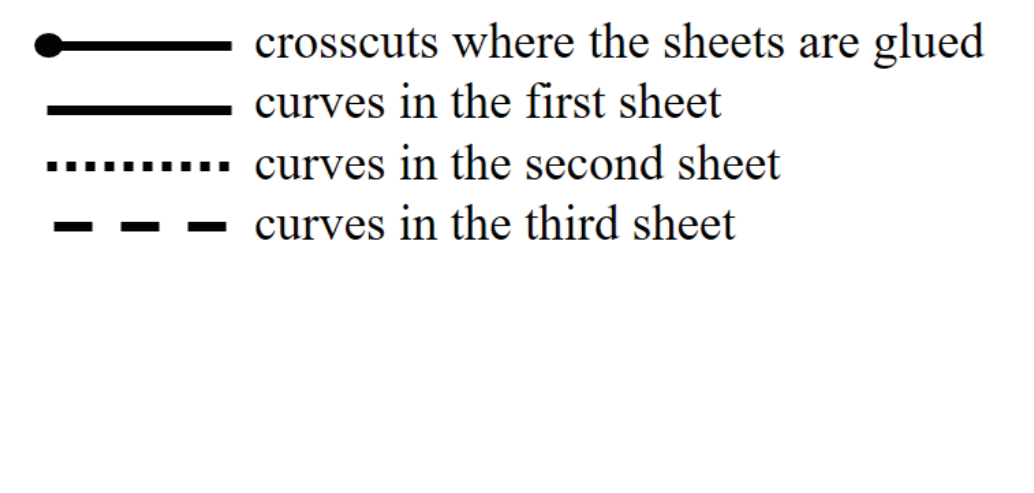}
\end{center}
\caption{
The basis
 of $H_1(X_s,\ZZ)$.}
\label{fig:H1X3}
\end{figure}

The complete Abelian integrals $\somegap{}$ and $\somegapp{}$ 
of the first kind are
$$
 \somegap{ij}
 := \frac{1}{2}
 \left(\int_{\alpha_j}\snuI{i}\right) ,\quad
\somegapp{ij} := \frac{1}{2}
 \left(\int_{\beta_j}\snuI{i}\right), 
$$
and the complete Abelian integrals $\setap{}$ and $\setapp{}$ 
of the second kind are
$$
\setap{ij}:= \frac{1}{2}
 \left(\int_{\alpha_j}\snuII{i}\right) , 
\quad
\setapp{ij}:= \frac{1}{2}
 \left(\int_{\beta_j}\snuII{i}\right) .
$$
Since we have the
integral along the contour
$\gamma_a$ from $\infty$ to the branch point $B_a=(b_a, 0)$ 
$(a=0, 1, 2, 3)$ \cite{MP15}, i.e.,
$$
\somega{a}:=\int_\infty^{B_a} \snuI{} =\int_{\gamma_a}\snuI{},
$$
the period matrices
$\somegap{}=(\somegap1,\somegap2,\somegap3)$ and
$\somegapp{}=(\somegapp1,\somegapp2,\somegapp3)$
are described in terms of $\somega{a}$ as follows.

\begin{lem}\label{lm:omegas}
$$
{}^t(\somegap1, \somegap2, \somegap3, \somegapp1, \somegapp2, \somegapp3)
= {}^t((\somega0, \somega1, \somega2,\somega3) W_{X_s}),
$$
where $W_{X_s}$ is a rank three matrix given by
\begin{gather*}
\begin{split}
W_{X_s}&:=\frac{1}{2}\begin{pmatrix}
0 & 0 & \hzeta_3-\hzeta_3^2 & 1-\hzeta_3^2 &0& 1-\hzeta_3^2  \\
\hzeta_3-1&0&0&\hzeta_3-1& \hzeta_3^2-1& 0\\
0& 1- \hzeta_3^2&0&0& 1-\hzeta_3^2 & 0\\
0&0&\hzeta_3^2-\hzeta_3& \hzeta_3^2-\hzeta_3&0 & \hzeta_3^2-1\\
\end{pmatrix}\\
&=\frac{1-\hzeta_3^2}{2}
\begin{pmatrix}
0       &0    &-\hzeta_3^2 & 1    & 0     &1     \\
\hzeta_3&0    &0           &\hzeta_3&  -1 &0     \\
0       &1    &0           &0     &1      &0     \\
0       &0  &\hzeta_3^2    &\hzeta_3^2 &0 & -1 
\end{pmatrix}\\
\end{split}
\end{gather*}
\end{lem}
The transpose is  introduced because the action of $\hzeta_3$ on 
the integrals is left-to-right.
Noting that the integral over a path which is homotopic to a point, e.g.,
\begin{gather}
(1-\hzeta_3^2) \somega2 +\hzeta_3^2(1-\hzeta_3^2) \somega1
	+\hzeta_3(1-\hzeta_3^2) \somega0 +(1-\hzeta_3^2) \somega3 = 0,
\label{eq:homotopic0}
\end{gather}
is equal to zero, $\somega{a}$ $(a=0,1,2,3)$ are described in terms of
 $\somegap{}=(\somegap1,\somegap2,\somegap3)$:

\begin{lem}\label{lm:somega_somegap}
$$
{}^t(\somega0, \somega1, \somega2,\somega3)=
{}^t((\somegap1, \somegap2, \somegap3) V_{X_s}),
$$
where
$$
V_{X_s}:=\frac{2}{3}\begin{pmatrix}
\hzeta_3^2 - 1 &\hzeta_3^2 - 1& 0& \hzeta_3^2 - 1\\
 \hzeta_3-\hzeta_3^2 & 0& 1 - \hzeta_3& \hzeta_3-\hzeta_3^2\\
 \hzeta_3^2-1& 0& 0&
\hzeta_3-1 \\
\end{pmatrix}
=
\frac{2(\hzeta_3^2 - 1)}{3}\begin{pmatrix}
1&1& 0&1\\
\hzeta_3^2 & 0& \hzeta_3&\hzeta_3^2\\
1& 0& 0&1 \\
\end{pmatrix}
.
$$
\end{lem}
\begin{proof}
By considering the relation
$$
{}^t(\somegap1, \somegap2, \somegap3,0)= {}^t
((\somega0, \somega1, \somega2,\somega3)V)
$$
with (\ref{eq:homotopic0}) and the matrix,
$$
V=\frac{1-\hzeta_3^2}{2}
\begin{pmatrix}
0       &0    &-\hzeta_3^2 & \hzeta_3 \\
\hzeta_3&0    &0           &\hzeta_3^2\\
0       &1    &0           &1    \\
0       &0  &\hzeta_3^2    &1 
\end{pmatrix},
$$
the inverse matrix of $V$ is reduced to $V_{X_s}$ 
when (\ref{eq:homotopic0}) holds.
\end{proof}
Lemma \ref{lm:somega_somegap}
can be regarded as a generalization of 
Lemma \ref{lm:y=0} which holds for  the genus one case.

Corresponding to (\ref{eqC:omegapp}) and (\ref{eqC:etapp}) in the genus 
one case, we have the following lemma:
\begin{lem}\label{eq:somegapp_p}
$$
\begin{pmatrix}
\somegapp{11} & \somegapp{12}& \somegapp{13}\\
\somegapp{21} & \somegapp{22}& \somegapp{23}\\
\somegapp{31} & \somegapp{32}& \somegapp{33}\\
\end{pmatrix}
=
\begin{pmatrix}
-\zeta_3^2\somegap{12}  &-\zeta_3^2\somegap{11}+ \somegap{12}
& -\zeta_3\somegap{13}\\
-\zeta_3^2\somegap{22}  &-\zeta_3^2\somegap{21}+ \somegap{22}
& -\zeta_3\somegap{23}\\
-\zeta_3  \somegap{32}&-\zeta_3\somegap{31}+  \somegap{32}
&-\zeta_3^2 \somegap{33}\\
\end{pmatrix},
$$
$$
\begin{pmatrix}
\setapp{11} & \setapp{12}& \setapp{13}\\
\setapp{21} & \setapp{22}& \setapp{23}\\
\setapp{31} & \setapp{32}& \setapp{33}\\
\end{pmatrix}
=
\begin{pmatrix}
-\zeta_3  \setap{12}  &-\zeta_3\setap{11}+ \setap{12}
& -\zeta_3^2\setap{13}\\
-\zeta_3  \setap{22}  &-\zeta_3\setap{21}+ \setap{22}
& -\zeta_3^2\setap{23}\\
-\zeta_3^2\setap{32}&-\zeta_3^2\setap{31}+  \setap{32}
&-\zeta_3 \setap{33}\\
\end{pmatrix}.
$$
\end{lem}

\begin{proof}
Lemmas \ref{lm:omegas} and \ref{lm:somega_somegap} directly give
$
{}^t(\somegapp1, \somegapp2, \somegapp3)
=
{}^t((\somegap1, \somegap2,$ $ \somegap3)U_{X_s})
$ and 
$
{}^t(\setapp1, \setapp2, \setapp3)=
{}^t((\setap1, \setap2, \setap3)U_{X_s})
$
where
$$
U_{X_s}:=\frac{2}{3}\begin{pmatrix}
0& -\hzeta_3^2& 0\\
-\hzeta_3^2& 1& 0\\
0& 0& -\hzeta_3
\end{pmatrix}.
$$
\end{proof}

\subsection{Jacobian and Abel-Jacobi map of $X_s$}
The lattice $\Gamma_s$
in $\CC^3$ is defined as
$$
\Gamma_s := \langle 2\somegap{}, 2\somegapp{}\rangle_{\ZZ}
=\left\{2\somegap{} 
\begin{pmatrix}a_1\\ a_2\\ a_3\end{pmatrix}+
 2\somegapp{}\begin{pmatrix}a_4\\ a_5\\ a_6\end{pmatrix}\ \Bigr|
\ a_i \in \ZZ\right\} \subset \CC^3,
$$
and the Jacobi variety $\cJ_{X_s}$ is given by the canonical projection,
$$
\kappa_\cJ:\CC^3\to \cJ_{X_s} = \CC^3 / \Gamma_s.
$$

Using the Abelian integral $\stw{}$, 
we define the Abel-Jacobi map $\sw{}$,
$$
\sw{}:\cS^k X_s\to \cJ_{X_s} \quad
\sw{}(P):= \kappa_\cJ \circ\stw{} \circ \iota_X(P).
$$
When $k=3$, $\sw{}$ is  birational due to the Abel-Jacobi theorem.
The normalized versions of these objects are given by
$$
\snuIo{} :=\somegapI \snuI{}, \quad \tw_{X_s}^\circ:= \somegapI \stw{}, \quad
w_{X_s}^\circ:= \somegapI \sw{},
$$
for the normalized period $(I_3, \tau_{X_s}:=\somegapI\somegapp{})$ and 
the normalized Jacobian,
$$
\kappa_\cJ^\circ: \CC^3 \to \cJ_{X_s}^\circ:=\CC^3/\Gamma_{s}^\circ,
$$
where $I_3$ is the unit matrix
and $\Gamma_{s}^\circ:=\langle I_3, \tau_{X_s}\rangle_\ZZ$.

\begin{lem}\label{lm:tauXs}
$$
\tau_{X_s}=\begin{pmatrix}0 & 0 & 0\\ 0 & 1 & 0\\ 0 & 0 & 0\end{pmatrix}
+\frac{\zeta_3}{|\somegap{}|} T_{1}
+\frac{\zeta_3^2}{|\somegap{}|} T_{2},
$$
where
{\small{
$$
T_{1}:=\left(\begin{matrix}
(\somegap{13}\somegap{22}-\somegap{12}\somegap{23})\somegap{32}&
(\somegap{13}\somegap{22}-\somegap{12}\somegap{23})\somegap{31}\\
(\somegap{11}\somegap{23}-\somegap{13}\somegap{21})\somegap{32}&
(\somegap{11}\somegap{23}-\somegap{13}\somegap{21})\somegap{31}\\
(\somegap{12}\somegap{21}-\somegap{11}\somegap{22})\somegap{32}&
(\somegap{12}\somegap{21}-\somegap{11}\somegap{22})\somegap{31}
\end{matrix}\right.
$$
$$
\qquad\qquad
\left.
\begin{matrix}
-(\somegap{13}\somegap{22}-\somegap{12}\somegap{23})\somegap{33}\\
-(\somegap{11}\somegap{23}-\somegap{13}\somegap{21})\somegap{33}\\
-(\somegap{11}\somegap{23}-\somegap{13}\somegap{21})\somegap{32}
+(\somegap{13}\somegap{22}-\somegap{12}\somegap{23})\somegap{31}
\end{matrix}\right),
$$
$$
T_{2}:=\left(\begin{matrix}
-(\somegap{13}\somegap{22}-\somegap{12}\somegap{23})\somegap{32}\\
(\somegap{12}\somegap{21}-\somegap{11}\somegap{22})\somegap{33}
+(\somegap{13}\somegap{22}-\somegap{12}\somegap{23})\somegap{31}\\
-(\somegap{12}\somegap{21}-\somegap{11}\somegap{22})\somegap{32}\\
\end{matrix}\right.
$$
$$
\qquad\qquad
\left.\begin{matrix}
(\somegap{11}\somegap{23}-\somegap{13}\somegap{21})\somegap{32}
+(\somegap{12}\somegap{21}-\somegap{11}\somegap{22})\somegap{33}\\
(\somegap{13}\somegap{21}-\somegap{11}\somegap{22})\somegap{31}\\
-(\somegap{12}\somegap{21}-\somegap{11}\somegap{22})\somegap{31}\\
\end{matrix}\right.
$$
$$
\qquad\qquad
\left.
\begin{matrix}
(\somegap{13}\somegap{22}-\somegap{12}\somegap{23})\somegap{33}\\
(\somegap{11}\somegap{23}-\somegap{13}\somegap{21})\somegap{33}\\
(\somegap{12}\somegap{21}-\somegap{11}\somegap{22})\somegap{33}
\end{matrix}\right).
$$
}}
\end{lem}

Since $\hzeta_3$ is an automorphism of the curve $X_s$ 
which admits Galois action, 
$\hzeta_3^c(2\somegap{a})$ $(a=1,2,3,$ $ c=0,1,2)$
in Lemma \ref{lm:somega_somegap} is a point in the lattice $\Gamma_s$,
and thus Lemma \ref{lm:somega_somegap} is reduced to the following lemma:
\begin{lem}\label{lm:hcab}{\rm{(\cite[Prop.2.2]{MP15})}}
The set $(\hzeta_3^c\somega{a})$ $(a=0,1,2,3,c=0,1,2)$ is a subset of $\Gamma_s$
and thus there are integers $h^{(c)\prime}_{a,b}$ and 
$h^{(c)\prime\prime}_{a,b}$ satisfying the relation
$$
\displaystyle{
3\hzeta_3^c\somega{a}=\sum_{b=1}^3
\left(
h^{(c)\prime}_{a,b}(2\somegap{b})+h^{(c)\prime\prime}_{a,b}(2\somegapp{b})
\right)} \mbox{ modulo }\Gamma_s.
$$
\end{lem}
This corresponds to Lemmas \ref{lm:y=0} and \ref{lm:C.3} of the genus-one case.
The factor $3$ means that $\hzeta_3^c\somega{a}$ is $3$-division point in 
the lattice $\Gamma_s$.

The Legendre relation, which determines the symplectic structure of the 
Jacobian $\cJ_s$, is given as
$$
{}^t\somegap{} \cdot\setapp{} - {}^t\somegapp{}
 \cdot\setap{} = \frac{\pi}{2}I_3.
$$

Corresponding to (\ref{eqC:etapp_omegap}), we have the following relation:
\begin{lem}\label{eq:etapp_omegap}
The left hand side of the Legendre relation is given by
$$
{}^t\somegap{} \cdot\setapp{} - {}^t\somegapp{}
 \cdot\setap{}=L_0+\zeta_3L_1+\zeta_3^2L_2,
$$
where
{\small{
$$
L_0:=\left(\begin{matrix}
0 & 
\setap{12}\somegap{11}+\setap{22}\somegap{21}+\setap{32}\somegap{31} \\
-\setap{11}\somegap{12}-\setap{21}\somegap{22}-\setap{31}\somegap{31}
& 0 \\
0 &
\setap{12}\somegap{13}+\setap{22}\somegap{23}+\setap{32}\somegap{33} \\
\end{matrix}\right.
$$
$$
\left.
\begin{matrix}
0\\
-\setap{13}\somegap{12}-\setap{23}\somegap{22}-\setap{33}\somegap{32}\\
 0\\
\end{matrix}\right),
$$
$$
L_1:=
\left(
\begin{matrix}
-\setap{12}\somegap{11}-\setap{22}\somegap{21}+\setap{31}\somegap{32} &
-\setap{11}\somegap{11}-\setap{21}\somegap{21}+\setap{32}\somegap{32} \\
-\setap{12}\somegap{12}-\setap{22}\somegap{22}+\setap{31}\somegap{31} &
-\setap{11}\somegap{12}-\setap{21}\somegap{22}+\setap{32}\somegap{31} \\
(\setap{11}-\setap{12})\somegap{13} 
+(\setap{21}-\setap{22})\somegap{23} &
-(\setap{11}-\setap{12})\somegap{13} 
-(\setap{21}-\setap{22})\somegap{23} 
\end{matrix}\right.
$$
$$
\left.
\begin{matrix}
-\setap{33}(\somegap{31}-\somegap{32})\\
-\setap{33}(\somegap{31}-\somegap{32})\\
\setap{13}\somegap{13}-\setap{23}\somegap{23}-\setap{33}\somegap{33} 
\end{matrix}\right),
$$
$$
L_2:=
\left(
\begin{matrix}
\setap{11}\somegap{12}+\setap{21}\somegap{22}-\setap{32}\somegap{31} &
\setap{12}\somegap{12}+\setap{22}\somegap{22}-\setap{31}\somegap{31} \\
\setap{11}\somegap{11}+\setap{21}\somegap{21}-\setap{32}\somegap{32} &
\setap{12}\somegap{11}+\setap{22}\somegap{21}-\setap{31}\somegap{32} \\
(\setap{31}-\setap{32})\somegap{33}  &
(\setap{32}-\setap{31})\somegap{33}
\end{matrix}\right.
$$
$$
\left.
\begin{matrix}
-\setap{13}(\somegap{11}-\somegap{12})
-\setap{23}(\somegap{21}-\somegap{22})\\
\setap{13}(\somegap{11}-\somegap{12})
+\setap{23}(\somegap{21}-\somegap{22})\\
-\setap{13}\somegap{13}-\setap{23}\somegap{23}+\setap{33}\somegap{33} 
\end{matrix}\right).
$$
}}
\end{lem}
 For later use when taking limits in Lemmas \ref{lm:etap} and \ref{lm:eta},
we note that $\somegap{13}$ appears in the monomials
$\somegap{13}\setap{11}$,
$\somegap{13}\setap{12}$ and 
$\somegap{13}\setap{13}$ in the Lemma above.

It is known that 
the Riemann constant $\xi$ defined by
\begin{equation}
\xi_{X_sj}:=\frac{1}{2}\tau_{X_sjj}+\sum_{i}^3
\int_{\alpha_i}\tw_{X_si}^o(\iota_X P)\snuIo{j}(P)+
\tw_{X_sj}^o(\iota_X Q_j),
\label{eq:RconstXs}
\end{equation}
where $Q_j$ is the beginning point of $\beta_j$,
corresponds to  a $\theta$ characteristic 
$$
\delta_{X_s} :=\left[\begin{matrix} \delta_{X_s}'' \\ \delta_{X_s}' \end{matrix}\right],\quad
  \delta_{X_s}' \in (\ZZ/2)^3,\quad
  \delta_{X_s}'' \in (\ZZ/2)^3,
$$
for the curve $X_s$ \cite{Lew,EEMOP07}. The vector
$\xi_{X_s}$ is unique modulo $\Gamma_s$.
Here it is also noted that since $Q_j$ is $\infty$ due to
Lemma \ref{lm:omegas}, $\tw_{X_si}^o(\iota_X Q_j)=0$.

\subsection{The sigma function of $X_s$}

Using the above structure,
the sigma function of $X_s$ is defined by \cite{EEMOP07}
$$
\ssigma{}(u):= c_s
\ee^{-\frac{1}{2}u^t \somegapI {}^t\setap{} u}
\theta_{X_s}\left[\begin{matrix} \delta_{X_s}'' \\ \delta_{X_s}' \end{matrix}\right]
\left(\frac{1}{2}\somegapI u; \somegapI {\somegapp{}}\right),
\ u\in\mathbb{C}^3.
$$
The ingredients of the formula are as follows:
$c_s$ is the constant factor
\begin{equation}
c_s:=\left(\frac{(2\pi)^3}{|\somegap{}|}\right)^{1/2} \Delta_s^{-1/8};
\label{eq:c_s}
\end{equation}
for the discriminant \cite{EGOY},
\begin{equation}
\Delta_s=
-729s^4 b_2^4 b_1^4 \left((s+b_2)^3 b_1^3+
3sb_2(s+\frac{1}{4} b_2)(s+4 b_2)b_1^2
+3s^2b_2^2(s+b_2)b_1+s^3b_2^3\right)^2,
\label{eq:discr}
\end{equation}
 and 
$\theta_{X_s}$ is the Riemann theta function associated with 
$\cJ_{X_s}^\circ$,
$$
\theta_{X_s}\left[\begin{matrix} a \\ b \end{matrix}\right]
(z; \tau_{X_s})
= \sum_{n \in \ZZ^3} \exp\left(\pi\ii((n+a)^t 
\tau_{X_s}(n+a) - (n+a)^t (z+b))\right).
$$

\begin{rem}
{\rm{
The discriminant $\Delta$ is computed using the recent results
in \cite{EGOY},
which play a crucial role in  this 
paper, cf. Remark \ref{rmk:sigmas0}.
}}
\end{rem}

For the translation formula, 
we introduce several pieces of notation.
For $u$, $v\in\CC^3$, and $\ell$
($=2\somegap{}\ell'+2\somegapp{}\ell''$) $\in\Gamma_s$,
we let
\begin{align}
  \sL(u,v)    &:=2\ {}^t{u}(\setap{}v'+\setapp{}v''),\nonumber \\
  \schi(\ell)&:=\exp[\pi\sqrt{-1}\big(2({}^t {\ell'}\delta_{X_s}''-{}^t
  {\ell''}\delta_{X_s}') +{}^t {\ell'}\ell''\big)] \ (\in \{1,\,-1\}) .
\label{eq:translXs}
\end{align}

Using the expansion (\ref{eq:u_i}),
we summarize the properties of the sigma function $\ssigma{}$
of $X_s$.
\begin{prop}{\rm{\cite{EEMOP07, N, EGOY}}}
The sigma function $\ssigma{}$ satisfies the followings:
\begin{enumerate}

\item it is an entire function over $\CC^3$,

\item its zeros $\kappa_\cJ^{-1}\Theta_s$ are given by
$\displaystyle{\Theta_s = \sw(X_s^2)}$,

\item its translation property is given by
\begin{align*}
\ssigma{}(u + \ell) = 
\ssigma{}(u) \exp(\sL(u+\frac{1}{2}\ell, \ell)) \schi(\ell),
\end{align*}
for 
$\ell \in \Gamma_{s}$,

\item it is modular invariant for the action of $\Sp(3,\ZZ)$, and

\item the leading term of its expansion is given by 
the Schur polynomial;
$\ssigma{}(u)=
s_{\Lambda_{(3,1,1)}}(u)+$ higher order terms with respect to the 
weight in (\ref{eq:wt_u_i}), where
$s_{\Lambda_{(3,1,1)}}$
is the Schur polynomial of  the Young diagram 
$\Lambda_{(3,1,1)}$,
$$
s_{\Lambda_{(3,1,1)}}(\tu)
 =\tu_1 -\tu_2^2 \tu_3 = t_1 t_2 t_3 
  (t_1^2 + t_2^2 + t_3^2 + t_1 t_2 + t_2 t_3 + t_3 t_1)
$$
for $\tu := {}^t(\tu_1,\tu_2,\tu_3)$,
where $\displaystyle{\tu_1:=\frac{1}{5}(t_1^5+ t_2^5+ t_3^5)}$,
      $\displaystyle{\tu_2:=\frac{1}{2}(t_1^2+ t_2^2+ t_3^2)}$ and 
      $\displaystyle{\tu_3:=t_1+ t_2+ t_3}$.
\end{enumerate}
\end{prop}

\subsection{The al function of $X_s$}\label{sec:al34}
We introduce a meromorphic function on an unramified
covering of the Jacobian $\cJ_s$, which
is a generalization of Jacobi's $\sn$, $\cn$, $\dn$-functions, 
i.e., the al-function 
\cite{MP15}; it is also regarded as
 a generalization of the hyperelliptic al-function
\cite{Baker98,Wei54}. 
(Since the
 simplest non-rational curve with an order-three automorphism
 is realized by an elliptic 
curve $E$: $y(y-s)=x^3$, we compute in Appendix \ref{AppdxC}
 the
al function of $E$, and note similar properties to 
the ones derived in \cite{MP15}.)
In the same way as the fundamental domains
 of $\sn$, $\cn$, $\dn$-functions are
double covering spaces of the Jacobian of genus one, or the fundamental
domain of the $\wp$ function, so is the
fundamental domain of the $\al_a$-function
also given as a certain triple covering of $\cJ_s$, denoted by $\cJ_s^{(a,c)}$;
$\pi_{\al}^{(a,c)}: \cJ_s^{(a,c)} \to \cJ_s$ for a branch point $B_a$
($a = 0, 1, 2, 3$) and $c = 0, 1, 2$.
For a fixed $B_a$,
there is an order-three cyclic action on $\cJ_s^{(a,c)}$ with respect to $c$
so that the origin of
$\cJ_s^{(a,c)}$ is translated by $\hzeta_3'$,
\begin{equation}
\xymatrix{ \cJ_s^{(a,0)}\ar[dr]\ar[r]^-{\hzeta_3'}& 
           \cJ_s^{(a,1)}\ar[d]\ar[r]^-{\hzeta_3'}&
           \cJ_s^{(a,2)}\ar@/_20pt/[ll]_-{\hzeta_3'} \ar[dl] \\
          & \cJ_s&. }
\label{eq:cJ2.6}
\end{equation}
Corresponding to $\cJ_s^{(a,c)}$, 
we have the triple covering $\widehat{X}_s^{(a)}$ of $X_s$
with respect to the branch point $B_a$ ($a=0,1,2,3$);
$\pi_{al,a}:\widehat{X}_s^{(a)}\to X_s$.
There is a natural projection $\tX_s \to \widehat{X}_s^{(a)}$
and there exist three different $\infty$'s on the three sheets
in $\widehat{X}_s^{(a)}$. The above action 
$\hzeta_3'$ on $\cJ_s^{(a,c)}$  corresponds to the choice for which 
$\infty$ is assigned to the fixed point in $\tX_s$ modulo the triple covering.
Further, for each $a=0,1,2,3$ and $c=0,1,2$,
we introduce  a certain vector $\varphi_{a;c} \in \CC^3$
defined in \cite[Definition 8.2]{MP15} 
which is related to the periods of the Jacobian $\cJ_s^{(a,c)}$:
$$
\varphi_{a;c}:=\frac{2}{3}
\sum_{b=1}^3
\left(
h^{(c)\prime}_{a,b}\setap{b}
+h^{(c)\prime\prime}_{a,b}\setapp{b}
\right),
$$
where $h$'s are defined in Lemma \ref{lm:hcab};
this formula corresponds to $\varphi_r$ in 
Definition \ref{def:al_E} of the elliptic curve $E$ case.

\begin{prop}\label{prop:al}
For $\gamma_{\infty,P_i}\in \tX_s$ of
$P_i = (x_i, y_i)$ $(i = 1, 2, 3)$ of $X_s$, and
$$
u=\stw{}(\gamma_{\infty,P_1},\gamma_{\infty,P_2},\gamma_{\infty,P_3}),
$$
the al-function defined by
$$
\al_a^{(c)}(u):=\frac{\ee^{- {}^t u \varphi_{a;c}}
       \ssigma{}(u + \hzeta_3^c \somega{a})}
{\ssigma{}(u) \ssigma{33}(\hzeta_3^c\somega{a})},
\quad (a=0, 1, 2, 3, c= 0, 1, 2)
$$
is a function of $u \in \cJ_s^{(a,c)}$ and 
is equal to
$$
\al_a^{(c)}(u)=-
\zeta_3^{c+\epsilon_a(\gamma_{\infty,P_1},
       \gamma_{\infty,P_2},\gamma_{\infty,P_3})}
\frac{A_a(P_1, P_2, P_3)}{ \sqrt[3]{(b_a - x_1) (b_a - x_2) (b_a - x_3)}},
$$
where 
$\displaystyle{\ssigma{33}(u):=\frac{\partial^2}{\partial u_3^2} \ssigma{}(u)}$,
$$
A_a(P_1, P_2, P_3):=
\displaystyle{ \frac{
     \left| \begin{matrix}
     1 & x_1 & y_1 & {x_1}^2 \\
     1 & x_2 & y_2 & {x_2}^2 \\
     1 & x_3 & y_3 & {x_3}^2 \\
     1 & b_a & 0   & b_a^2 \\
     \end{matrix} \right|} {
     \left| \begin{matrix}
     1 & x_1 & y_1  \\
     1 & x_2 & y_2  \\
     1 & x_3 & y_3  \\
     \end{matrix} \right|}},
\quad
\ssigma{33}(\somega{a})=\frac{\sqrt{2}}{\sqrt[3]{f'(b_a)}},
$$
and $\epsilon_a$ is a map $\epsilon_a: \tX_s^3 \to \ZZ_3$
so that the right-hand side is a function on
the triple covering $\widehat{X}_s^{(a)}$ of $X_s$. 
Here the cubic root is determined by the choice of $c$-th
Jacobian $\cJ_s^{(a,c)}$ and the corresponding sheet of the triple
covering $\widehat{X}_s^{(a)} \to X_s$. 
\end{prop}

\begin{rem}
{\rm{
We note that $\al_a$-function is a function on a covering space 
$\cS^3 \widehat{X}_s^{(a)}$ of
$\cS^3 X_s$ and thus could be viewed as a function on $\cS^3\tX_s$
because there is a natural projection $\tX_s \to \widehat{X}_s^{(a)}$;
a point $\gamma$ of $\tX_s$ corresponds to a path in $X_s$ from $\infty
\in X_s$ to a point $P \in X_s$ and thus we denote it by
$\gamma=\gamma_{\infty, P}$.
If we fix the points $\gamma_2$ and $\gamma_3$ for 
$(\gamma_1, \gamma_2, \gamma_3) 
\in \cS^3\tX_s$ and 
regard the $\al_a$-function as a function of $\gamma_1=\gamma_{\infty, P}$,
the $\al_a$-function is a function on $\widehat{X}_s^{(a)}$
and locally agrees with the local parameter at 
the branch point $\pi_{al;a}^{-1}B_a$ up to $\zeta_3^c$.
The $\al_a$-function could be characterized as being a 
meromorphic function over $\cS^3\tX_s$
so that it agrees with the local parameter at $B_c$ if and only if $c=a$.
}}
\end{rem}
\bigskip

\section{The sigma function of $y^3 = x^2(x-b_1)(x-b_2)$}

In this section we introduce the sigma function of 
a normalization of the singular curve $X_0$ \cite{MK,KMP18}.

\subsection{Basic properties of the normalization $X_{\hzero}$ of $X_0$}

We consider  the singular curve $X_s$ at  $s=0$
given by $\displaystyle{y^3 =x^2 k(x)}$ where $k(x) = (x-b_1)(x-b_2)$,
which is known as the Borwein curve \cite{BB} and studied in 
\cite{EG,MK,KMP18}.
Its affine ring is given by
$$
R_0=\mathbb{C}[x,y]/(y^3-x^2k(x)).
$$
The normalization of 
$R_0$ yields
 the ring \cite{MK,KMP18}
$$
R_{\hzero}=\CC[x,y,z]/(y^2-x z, z y-x k(x), z^2 - k(x) y),
$$
and a curve $\hpi:X_{\hzero}\to X_0$ of genus two
given by three equations in affine space and completed by a smooth point at
infinity:
$$
X_{\hzero}=\{(x,y,z)\ |\ y^2=x z, z y=x k(x), z^2 = k(x) y\}
\cup \{\infty\}.
$$
We note that both $z=y^2/x$ and $t$ $(t^3=x)$
are local parameters at $\nB0=(x=0,y=0,z=0)$
and the local ring is given by $\CC[[z]]$ or $\CC[[t]]$.

The automorphism $\hzeta_3$ on $X_0$, by virtue of the relation $z y=x k(x)$,
induces the action on $X_\hzero$ and $R_\hzero$,
$$
\hzeta_3 (x,y,z) = (x, \zeta_3 y, \zeta_3^2 z).
$$

We also have the natural projection $\hpi_a: X_\hzero \to \PP$,
$$
\hpi_1(P) = x, \quad 
\hpi_2(P) = y, \quad \hpi_3(P) = z.
$$

Each branch point of $\hpi_1$ is given by
$$
\nB0=(x=0,y=0,z=0), \nB1=(b_1, 0, 0), \nB2=(b_2, 0, 0),
$$
and is simply denoted by $B_i$ if there is no confusion.
The smooth point at
infinity of $X_{\hzero}$ is a non-Weierstrass point whose
Weierstrass non-gap sequence is generated by $(x,y,z)$ 
as in Table 2, and Weierstrass semigroup  generated by $\{3,4,5\}$
because $y^3=x^2 k(x)$ and $z^3=x k(x)^2$. 

\begin{table}[h]
Table 2: Weierstrass non-gap sequence of $X_{\hzero}$\\
  \begin{tabular}{l|cccccccccccccc}
$-\wt$&0 &1 & 2 & 3 & 4 & 5 & 6 & 7 & 8 & 9 & 10 & 11 & $\cdots$\\
\noalign{\smallskip}
\noalign{\hrule height0.3pt}
\noalign{\smallskip}
$\nphi{i}$ & 
 1& - & - & $x$ & $y$ & $z$ & $x^2$& $x y$ & $x z$ & $y z$ 
& $x^2 y$
& $x^2 z$  & $\cdots$\\
  \end{tabular}
\end{table}
We have a natural decomposition as $\CC$-vector space:
$$
R_{\hzero}=\bigoplus_{i=0} \CC \nphi{i}.
$$
We define a weight,
 wt, using the order of pole at $\infty$
$$
\wt(x)=-3, \quad \wt(y)=-4, \quad \wt(z)=-5.
$$
The Weierstrass non-gap sequence is determined by
  the numerical semigroup 
$H_\hzero:=$ $\{ 3a+4b+$ $5c\}_{a,b,c\in\ZZ_{\ge0}}
=\langle3,4,5\rangle$, i.e.,
$$
   H_\hzero = \{0, 3, 4,5, 6, 7, \cdots\}, \quad
   L_\hzero = \ZZ_{\ge0} \setminus H_\hzero = \{1,2\},
$$
which is related to the Young diagram ($(1,1)=(1,2)-(0,1)$)
$$
\Lambda_{(1,1)}= \yng(1,1).
$$
The Young diagram $\Lambda_{(1,1)}$ is not self-dual because it differs from
its transpose $\yng(2)$. A semigroup whose associated Young diagram is not 
self-dual is non-symmetric \cite{KMP18}. Note that
the numerical semigroup $H_s$ of 
$X_s$ $(s\neq 0)$ is symmetric whereas $H_\hzero$ is a non-symmetric semigroup.

Of course the affine model of the normalization is not unique, as
the following lemma shows.
\begin{lem}\label{lem:GAc}
There is a group action 
$(\hzeta_3^*)^a:z\to \zeta_3^a z$ $(a=0,1,2)$ on the prime ideals of 
$\CC[x,y,z]$, i.e.,
$$
(y^2-\zeta_3^a x z, \zeta_3^a z y-x k(x), 
\zeta_3^{2a} z^2 - k(x) y),
$$
which induces the three different normalizations 
$R_\hzero$, $R_\hzero^*$, and  
$R_\hzero^{**}$of $R_0$ with the $\hzeta_3$ action
$$
\hzeta_3 (x,y,z) = (x, \zeta_3 y, \zeta_3^{2+a} z)
$$
respectively.
\end{lem}

The normalization is unique up to the isomorphism given by 
 the (biholomorphic) action $\hzeta_3^*$. 
Corresponding to the relations
among $R_\hzero$, $R_\hzero^*$, and   $R_\hzero^{**}$,
we have the relations among the normalized curves
$X_\hzero$, $X_\hzero^*$, and $X_\hzero^{**}$,
\begin{equation}
\xymatrix{ X_\hzero \ar[dr]\ar[r]^-{\hzeta_3^*}& 
           X_\hzero^*\ar[d]\ar[r]^-{\hzeta_3^*}&
           X_\hzero^{**} \ar@/_20pt/[ll]_-{\hzeta_3^*} \ar[dl] \\
          & X_0&. }
\label{eq:zeta*}
\end{equation}
The isomorphism $\hzeta_3^*$ also acts on the local parameter 
$z$ at $B_0=(x=0,y=0,z=0)$.

\begin{rem} \label{rmk:XhzeroP^2}
{\rm{
The action $\hzeta_3^*$ can be induced on
 $\hpi_3(X_\hzero^{**})=\PP$
because it affects only  the third component $z$.
The relation (\ref{eq:zeta*}) among the three genus two curves 
$X_\hzero$, $X_\hzero^{*}$ and $X_\hzero^{**}$ can be regarded as the 
relation among the genus two curve
$X_\hzero$, and two rational curves 
$\hpi_3(X_\hzero^{*})=\PP$ and $\hpi_3(X_\hzero^{**})=\PP$,
which is Remark \ref{rmk:Esingfiber} in Appendix \ref{AppdxC}.
}}
\end{rem}

\begin{rem}\label{lm:Hyp}
{\rm{
Every genus-two curve is hyperelliptic, and the
birational map \cite{EG} that sends $(x,y,z)\in X_{\hzero}$
to $\displaystyle{\eta:=\frac{x^2-b_1 b_2}{x}}$ gives the double-cover of
$\mathbb{P}^1$ 
model:
 $$
\eta^2 =\xi^6+2(b_1+b_2) \xi^3+(b_1-b_2)^2.
$$
This result was obtained using the Maple software {\it{algcurves}}-package
 based on  van Hoeij's algorithm \cite{vH}
(in \cite{vH}, this model is called 
 Weierstrass normal form, but note that there are
two (non-Weierstrass) points at infinity).
Using $X_\hzero$ instead, 
our method works for the more general case of
a cyclic trigonal curve of $(3,p,q)$ type \cite{KMP18}.
}}
\end{rem}

\subsection{Differentials and Abelian integrals of
  $X_{\hzero}$}

In order to describe the holomorphic one-forms of $X_{\hzero}$,
we define a subspace of $R_\hzero$,
$$
\hat R_\hzero:=\{ h \in R \ | \ \exists \ell, \ 
\mbox{such that}\ (h)-(B_0+B_1+B_2) +\ell \infty >0\},
$$
which is decomposed into $\hat R_\hzero=\oplus_{i=0} \CC  \nhphi{i}$ as a 
$\CC$-vector space,
with basis displayed in Table 3.

\begin{table}[h]
Table 3: Weierstrass non-gap sub-sequence of $X_\hzero$.
  \begin{tabular}{l|cccccccccccccc}
$-\wt$&0 &1 & 2 & 3 & 4 & 5 & 6 & 7 & 8 & 9 & 10 & 11& $\cdots$\\
\noalign{\smallskip}
\noalign{\hrule height0.3pt}
\noalign{\smallskip}
 $\nphi{i}$ & 
 1& - & - & $x$ & $y$ & $z$ & $x^2$& $x y$ & $x z$ & $y z$ 
& $x^2 y$
& $x^2 z$ & $\cdots$\\
$\nhphi{i}$ & 
 -& - & - & - & $y$ & $z$ & -& $x y$ & $x z$ & $y z$ 
& $x^2 y$ 
& $x^2 z$& $\cdots$\\
\noalign{\smallskip}
\noalign{\hrule height0.3pt}
\noalign{\smallskip}
  \end{tabular}
\end{table}

We have the
differentials of the first kind (holomorphic one-forms)
 given by
$$
\nnuI1=\frac{\nhphi0 d x}{3y z}=\frac{d x}{3z}, \ 
\nnuI2=\frac{\nhphi1 d x}{3y z}=\frac{d x}{3y}, \quad
H^0(X_{\hzero},\Omega^1)=\CC \nnuI1+\CC \nnuI2.
$$
The ratio $(\nhphi0:\nhphi1)$ 
describes
  a canonical embedding of $X_{\hzero}$,
and the canonical divisor is given as
$$
\cK_{X_{\hzero}} = 2g\infty-2B_0 \sim (2g-2)\infty+(B_1+ B_2),
$$
which is not linearly equivalent to 
$(2g-2)\infty$ $(g=2)$; this  corresponds to the fact $H_\hzero$
is a non-symmetric semigroup \cite{KMP16}.
The differentials of the second kind 
are given by \cite{MK}:

$$
\nnuII1 = \frac{-2x d x}{3y}, \quad 
\nnuII2 = \frac{-x d x}{3z}
$$
(we note that in \cite{MK}, the numerator of $\nnuII1$ should read
$(-2x+\lambda_1^{(1)}) dx$; in this case, $\lambda_1^{(1)}=0$
thus the result in \cite{MK} is in agreement with these formulas).
These $\nnuI{}$ and $\nnuII{}$ form the basis of the first 
algebraic de Rham cohomology group of $X_\hzero$.
The automorphism $\hzeta_3$ also acts on the one-forms,
$$
\hzeta_3(\nnuI{1}) = \zeta_3\nnuI{1},\quad
\hzeta_3(\nnuI{2}) = \zeta_3^2\nnuI{2},\quad
\hzeta_3(\nnuII{1}) = \zeta_3^2\nnuII{1},\quad
\hzeta_3(\nnuII{2}) = \zeta_3\nnuII{2}.\quad
$$

As in $\tX_s$,
let $\tX_{\hzero}$ be the Abelian covering of
$X_{\hzero}$, $\kappa_X:\tX_{\hzero} \to X_{\hzero}$, 
with the fixed point $\infty$; 
we fix the natural
embedding $\iota_X: X_{\hzero} \to \tX_{\hzero}$.
The Abelian integral is defined by
$$
\ntw{}: \tX_{\hzero}\to \CC^2, \quad
\ntw{}(\gamma_{\infty, P})=\int_{\gamma_{\infty, P}}
 \nnuI{}.
$$
For the local parameters $t_1$ and $t_2$ 
at $\infty$ of $X_\hzero$ and $v=\ntw{}(\iota_X (t_1,t_2))$, 
\begin{gather}
\begin{split}
v_1 &=\frac{1}{2}(t_1^2+t_2^2)(1 + d_{\ge 1}(t_1, t_2)),\\
v_2 &=\frac{1}{1}(t_1+t_2)(1 + d_{\ge 1}(t_1, t_2)).\\
\end{split}
\label{eq:v_t}
\end{gather}
We also define the weight of $v$'s:
\begin{equation}
\wt(v_1) = 2,\quad \wt(v_2)=1.
\label{eq:wt_v}
\end{equation}

\subsection{Periods of $X_{{\hzero}}$ }
For the case of $X_{\hzero}$, we also consider the basis of the 
homology,
 $\alpha$'s and $\beta$'s, illustrated in Figure 
\ref{fig:H1X2}.

\begin{figure}[ht]
\begin{center}
\includegraphics[width=0.40\textwidth]{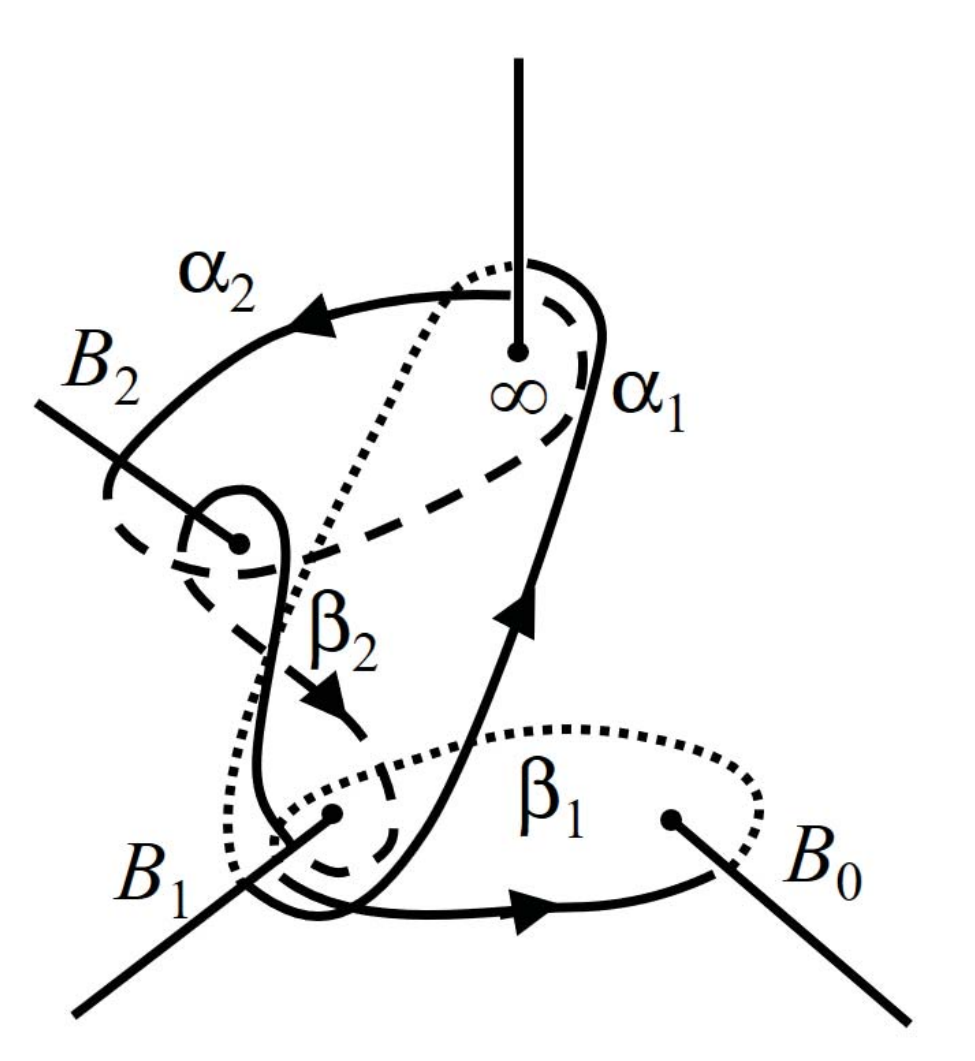}
\includegraphics[width=0.40\textwidth]{FigX3sopenA.pdf}
\end{center}
\caption{
The basis of $H_1(X_{\hzero},\ZZ)$:
We note the fact that 
$B_0$ is the singular point.
}
\label{fig:H1X2}
\end{figure}

The half periods
 $\nomegap{}=(\nomegap{ij})$ and $\nomegapp{}=(\nomegapp{ij})$
are given by
$$
\nomegap{ij}:=\frac{1}{2}\left(\int_{\alpha_j} \nnuI{i}\right), \quad
\nomegapp{ij}:=\frac{1}{2}\left(\int_{\beta_j} \nnuI{i}\right).
$$
For the 
integral along the contour
$\gamma_a$ from $\infty$ to the branch point $B_a=(b_a, 0)$ 
\cite{MP15},
$$
\nomega{a}:=\int_\infty^{B_a} \nnuI{} =\int_{\gamma_a}\nnuI{},
$$
the period matrices,
$\nomegap{}=(\nomegap1,\nomegap2)$ and
$\nomegapp{}=(\nomegapp1,\nomegapp2)$ are described in terms of $\nomega{a}$:
\begin{lem}\label{lm:omegan}
$$
{}^t(\nomegap1, \nomegap2, \nomegapp1, \nomegapp2)
={}^t((\nomega0, \nomega1, \nomega2) W_{X_{\hzero}}),
$$
where
$$
W_{X_{\hzero}}:=\frac{1}{2}\begin{pmatrix}
0 & 0 & 1-\hzeta_3 & 0\\
\hzeta_3-1 & 0&\hzeta_3-1  & \hzeta_3^2-1 \\
0 & 1-\hzeta_3^2 & 0 & 1-\hzeta_3^2\end{pmatrix}
=\frac{1-\hzeta_3^2}{2}\begin{pmatrix}
0 & 0 & -\hzeta_3 & 0\\
\hzeta_3 & 0&\hzeta_3  & -1 \\
0 & 1 & 0 & 1\end{pmatrix}
.
$$
\end{lem}
Similarly, the Abelian integral of the differentials of 
the second kind $\netap{}=(\netap{ij})$ and $\netapp{}=(\netapp{ij})$
are given by
$$
\netap{}:=\frac{1}{2}\left(\int_{\alpha_j} \nnuII{i}\right), \quad
\netapp{}:=\frac{1}{2}\left(\int_{\beta_j} \nnuII{i}\right).
$$

Since we also have the identity
$$
(\hzeta_3-\hzeta_3^2) \nomega0 +
(\hzeta_3^2-\hzeta_3) \nomega1 +
(1-\hzeta_3^2) \nomega2 =0,
$$
we have the relations between $\nomegap{}$ and $\nomegapp{}$
($\netap{}$ and $\netapp{}$):
\begin{lem}\label{eq:nomegapp_p}
$$
\begin{pmatrix}
\nomegapp{11} & \nomegapp{12}\\
\nomegapp{21} & \nomegapp{22}\\
\end{pmatrix}
=
\begin{pmatrix}
-\zeta_3^2\nomegap{12}  &-\zeta_3^2\nomegap{11}+ \nomegap{12}\\
-\zeta_3^2\nomegap{22}  &-\zeta_3^2\nomegap{21}+ \nomegap{22}
\end{pmatrix},
$$
$$
\begin{pmatrix}
\netapp{11} & \netapp{12}\\
\netapp{21} & \netapp{22}
\end{pmatrix}
=
\begin{pmatrix}
-\zeta_3 \netap{12}  &-\zeta_3\netap{11}+ \netap{12}\\
-\zeta_3  \netap{22}  &-\zeta_3\netap{21}+ \netap{22}\\
\end{pmatrix}.
$$
\end{lem}

\subsection{Jacobian and Abel-Jacobi map of $X_{\hzero}$ }

Using the lattice defined by
 $\Gamma_{\hzero}:=$ $\langle 2\nomegap{}, 2\nomegapp{}\rangle_{\ZZ}$,
the Jacobi variety is obtained
by the canonical projection
$\kappa_\cJ :\CC^2 \to \cJ_{\hzero}:=\CC^2/\Gamma_{\hzero}$.
The Legendre relation is given as
$\displaystyle{
{}^t\nomegap{} \netapp{} - {}^t\nomegapp{} \netap{} = \frac{\pi}{2}I_2
}$.

Using the Abelian integrals, we define the Abel-Jacobi map $\nw{}$,
$$
\nw{}:\cS^k X_{\hzero}\to \cJ_{X_{\hzero}}, \quad
\nw{}(P):= \kappa_\cJ\circ \ntw{} \circ\iota_X(P);
$$
now $\nw{}$ is a birational map in the $k=2$ case.
We also have the standard-normalization versions,
$$
\nnuIo{} :=\nomegapI \nnuI{}, \quad \ntw{}^\circ:= \nomegapI \ntw{}, \quad
\nw{}^\circ:= \nomegapI \nw{},
$$
for the normalized period $(I_2, \tau_{X_\hzero}:=\nomegapI\nomegapp{})$ and 
the normalized Jacobian,
$$
\kappa_\cJ^\circ: \CC^2 \to 
\cJ_{\hzero}^\circ:=\CC^2/\Gamma_{\hzero}^\circ,
\quad \Gamma_{\hzero}^\circ:=\langle I_2, \tau_{X_\hzero}\rangle.
$$

Direct computations give the following lemma.
\begin{lem}\label{lm:tauXhz}
$$
\tau_{X_\hzero}=
\begin{pmatrix}0 & 0 \\ 0 & 1\end{pmatrix}
+\frac{\zeta_3}{|\nomegap{}|} 
\begin{pmatrix}
\nomegap{22}\nomegap{12} & 
\nomegap{12}\nomegap{21} \\
-\nomegap{11}\nomegap{22} & 
-\nomegap{11}\nomegap{21} \end{pmatrix}
$$
$$
+\frac{\zeta_3^2}{|\nomegap{}|} 
\begin{pmatrix}
\nomegap{22}\nomegap{12} & 
-\nomegap{11}\nomegap{22} \\
\nomegap{12}\nomegap{21} & 
\nomegap{21}\nomegap{11} \end{pmatrix}.
$$
\end{lem}

As in (\ref{eq:RconstXs}),
the Riemann constant $\xi_{X_\hzero}$ of $X_\hzero$ is defined by \cite{Lew}
\begin{equation}
\xi_{X_\hzero j}:=\frac{1}{2}\tau_{X_\hzero jj}+\sum_{i}^2
\int_{\alpha_i}\tw_{X_\hzero i}^o(\iota_X P)\nnuIo{j}(P)+
\tw_{X_\hzero i}^o(\iota_X Q_j),
\label{eq:RconstXhzero}
\end{equation}
where $Q_j$ is the beginning point of $\beta_j$ and 
$\tw_{X_\hzero i}^o(\iota_X Q_j)$ in this convention is equal to zero.

\subsection{The shifted Abel-Jacobi map and the shifted Riemann constant 
of $X_{\hzero}$}\label{sec:3.5}

The fact that $\ncK \neq (2g-2)\infty=2\infty$ makes
the construction of $\nsigma{}$ non-standard,
but by using the shifted Abelian integral
$\ntw{\fs}$ and the shifted Riemann constant $\xi_{X_\hzero\fs}$ 
as in \cite{KMP16},
we can bypass the problem. We review the results of \cite{KMP16}.

We, first, introduce the 
unnormalized shifted Abelian integral and 
the unnormalized shifted Abel-Jacobi map.
For $\gamma_1, \gamma_2, \dots, \gamma_k$,
the shifted Abelian integral is defined by
\cite{KMP16},
$$
\ntw{\fs}(\gamma_1, \gamma_2, \ldots, \gamma_k)=\sum_i^k \ntw{}(\gamma_i)
+\ntw{}(\iota_{X} B_0)
$$
and for $P_1, P_2, \cdots, P_k \in \cS^k X_{\hzero}$,
the shifted Abel-Jacobi map is given by
$$
\nw{\fs}(P_1, P_2, \ldots, P_k)=\sum_i^k \nw{}(P_i)
+\nw{}(B_0).
$$
Using $\ntw{\fs}$ and $\nw{\fs}$, 
we let the normalized versions be
$\ntw{\fs}^\circ:=\nomegapI{}\ntw{\fs}$, 
$\ntw{}^\circ:=\nomegapI{}\ntw{}$.
From \cite{KMP16}
we have the following facts:

\begin{enumerate}
\item Since $\nw{}^\circ(\ncK) + 2\xi_{X_\hzero}=0$ modulo
 $\Gamma_{\hzero}$ and
 $\nw{}^\circ(\ncK)=-\nw{}^\circ(2B_0)$,
we have 
$-2\nw{}^\circ(B_0) +2\xi_{X_\hzero}=0$ modulo $\Gamma_{\hzero}$.
Thus the Riemann constant $\xi_{X_\hzero}$ is not a half period,
however
\begin{equation}
\xi_{{X_\hzero}\fs}:=\xi_{X_\hzero}-\ntw{}^\circ(\iota_X B_0)
\label{eq:Rconsts}
\end{equation}
becomes a half period.

\item
The relation between the theta divisor 
$\Theta:=$div$(\theta)$ and the standard theta divisor
$\nw{}(X_{\hzero})$ is alternatively expressed by
$$
\Theta = \nw{}(X_{\hzero}) +\xi_{X_\hzero} = \nw{\fs}(X_{\hzero}) 
+\xi_{{X_\hzero}\fs} \quad \mbox{modulo}\quad
 \Gamma_{\hzero}.
$$

\item
The $\theta$-characteristics of the half period 
$\left[\begin{matrix}\delta_{X_\hzero}''\\ \delta_{X_\hzero}'\end{matrix}\right]\in (\ZZ/2)^{2}$
corresponding to the vector
 $\xi_{{X_\hzero}\fs}$,
represent
the shifted Riemann constant $\xi_{{X_\hzero}\fs}$ \cite{KMP16}.

\end{enumerate}

It will turn out that the shifted Riemann constant naturally appears
in the degeneration limit, cf. Proposition \ref{prop:RCs}.

\subsection{The sigma function of $X_{\hzero}$}
For 
$\left[\begin{matrix}\delta_{X_\hzero}''\\ \delta_{X_\hzero}'\end{matrix}\right]\in (\ZZ/2)^{2}$,
which corresponds to the vector
 $\xi_{{X_\hzero}\fs}$,
we define the sigma function as an entire function over $\CC^2$:
$$
   \nsigma{}(v)
   =c_\hzero\mathrm{e}^{-\tfrac{1}{2}\ ^t 
v\netap{}\nomegapI{}  v} 
\theta_{X_{\hzero}}\left[\begin{matrix} \delta_{X_\hzero}'' \\ \delta_{X_\hzero}'\end{matrix}\right]
\left(\frac{1}{2}\somegapI{} v; \tau_{X_\hzero}\right),
$$
where
$c_\hzero(\neq 0)$ is a constant complex number,  and
$$
\theta_{X_{\hzero}}\left[\begin{matrix} a \\ b \end{matrix}\right]
(z; \tau_{X_\hzero})
:= \sum_{n \in \ZZ^2} \exp\left(\pi\ii((n+a)^t 
\tau_{X_\hzero}(n+a) - (n+a)^t (z+b))\right).
$$

As in (\ref{eq:translXs}), 
we introduce several pieces of notation;
for $u$, $v\in\CC^2$, and $\ell$
($=2\nomegap{}\ell'+$ $2\nomegapp{}\ell''$) 
$\in\Gamma_{\hzero}$,
we let
\begin{align}
  \nL(u,v)    &:=2\ {}^t{u}(\netap{}v'+\netapp{}v''),\nonumber \\
  \nchi(\ell)&:=\exp[\pi\sqrt{-1}\big(2({}^t {\ell'}\delta_{X_\hzero}''-{}^t
  {\ell''}\delta_{X_\hzero}') +{}^t {\ell'}\ell''\big)] \ (\in \{1,\,-1\}) .
\label{eq:translXhzero}
\end{align}
Noting
(\ref{eq:v_t}) and (\ref{eq:wt_v}),
 we summarize the properties of the sigma function $\nsigma{}$ \cite{KMP18,MK}:

\begin{prop}\label{prop:sigmaX2}
The sigma function $\nsigma{}$ satisfies the followings:
\begin{enumerate}

\item it is an entire function over $\CC^2$,

\item its divisor  is given by
$\displaystyle{
\{\mathrm{div}\,\nsigma{}\}
:=\kappa_\cJ^{-1} \Theta = \kappa_\cJ^{-1}\nw{\fs}(X_{\hzero})}$,

\item its translation property is given by 
\begin{align*}
\nsigma{}(v + \ell) = 
\nsigma{}(v) \exp(\nL(v+\frac{1}{2}\ell, \ell)) \nchi(\ell),
\end{align*}
for $\ell \in \Gamma_{\hzero}$,

\item 
it is a modular invariant for $\Sp(2,\ZZ)$, and 

\item the expansion of $\nsigma{}(v+\nomega{0})$ is given by 
$\nsigma{}(v+\nomega{0})=
s_{\Lambda_{(1,1)}}(v)+$higher order terms with respect to
the weight of $v$ in (\ref{eq:wt_v}),
where $s_{\Lambda_{(1,1)}}$ is the
 Schur polynomial of the Young diagram  $\Lambda_{(1,1)}$,
$$
s_{\Lambda_{(1,1)}}(\tv)=
\tv_1-\tv_2^2=t_1t_2,
$$
for $\tv:=^t(\tv_1,\tv_2)$
where $\displaystyle{
\tv_1 :=\frac{1}{2}(t_1^2+t_2^2)}$ and
$\displaystyle{
\tv_2 :=(t_1+t_2)}$.

\end{enumerate}
\end{prop}

\subsection{The action of $\hzeta_3^*$ on the sigma function}
\label{action_hzetap}

The operator $\hzeta_3^*$ acts on the set 
$\{X_\hzero, X_\hzero^*, X_\hzero^{**}\}$ 
via the relation $\hzeta_3^*(x,y,z) = (x,y,\zeta_3 z)$.
This provides an induced action on the differentials of each 
$X_\hzero^*$ and $X_\hzero^{**}$ using $X_\hzero$.
The action on the differentials of the first kind is given by
$$
\hzeta_3^*(\nnuI{1}) = \zeta_3^2 \nnuI{1}, \quad
\hzeta_3^*(\nnuI{2}) = \nnuI{2}, \quad
$$
whereas that on the differentials of the second kind is given by
$$
\hzeta_3^*(\nnuII{1}) = \zeta_3 \nnuII{1}, \quad
\hzeta_3^*(\nnuII{2}) = \nnuII{2}.
$$
These relations determine
the action of 
$\hzeta_3^*$ 
on the matrices $\nomega{}$, $\nomegap{}$,
$\nomegapp{}$, $\netap{}$ and $\netapp{}$ by the matrices 
$\displaystyle{M_{\hzeta^*1}:=
\begin{pmatrix} \zeta_3^2 & 0\\ 0 & 1\end{pmatrix}}$
and 
$\displaystyle{M_{\hzeta^*2}:=
\begin{pmatrix} \zeta_3 &0 \\0  & 1 \end{pmatrix}}$,
 e.g,
$$
\hzeta_3^*\nomega{}=
 \begin{pmatrix} \zeta_3^2 & 0\\ 0 & 1\end{pmatrix}
\begin{pmatrix} \nomega{10} &\nomega{11}&\nomega{12} \\  
 \nomega{20} &\nomega{21}&\nomega{22} \\  
\end{pmatrix},
$$
and 
$$
\hzeta_3^*\nomegap{}=M_{\hzeta^*1}\nomegap{},\quad
\hzeta_3^*\nomegapp{}=M_{\hzeta^*1}\nomegapp{},\quad
\hzeta_3^*\netap{}=M_{\hzeta^*2}\netap{},\quad
\hzeta_3^*\netapp{}=M_{\hzeta^*2}\netapp{}.\quad
$$
Therefore the action does not have any effect on the Legendre relation
and symplectic structure.
Further it acts on the Abelian integral by $M_{\hzeta^*1}$, so that
the action on its image is given by
$
\hzeta^*_3 \CC^2 = M_{\hzeta^*1}\CC^2.
$
Thus the action leaves
$$
\nomegapI{} \nomegapp{}, \quad
^t 
u\netap{}\nomegapI{}  u, \quad \nomegapI{}  u
$$
invariant. The sigma function $\sigma_{X_\hzero^*}$ of $X_\hzero^*$ 
($\sigma_{X_\hzero^{**}}$ of $X_\hzero^{**}$) on $\CC^2$ is given by
\begin{equation}
\sigma_{X_\hzero^{*}}(v)=\sigma_{X_\hzero}(M_{\hzeta^*1}^{-1}v),\quad
\sigma_{X_\hzero^{**}}(v)=\sigma_{X_\hzero}((M_{\hzeta^*1}^{-1})^2v).
\label{eq:sigmaXhzero*}
\end{equation}
The action on the sigma function is denoted by
 $\sigma_{X_\hzero^{*}}=\hzeta_3^* \sigma_{X_\hzero}$ and 
$\sigma_{X_\hzero^{**}}=(\hzeta_3^*)^2 \sigma_{X_\hzero}$.
Corresponding to  these sigma functions and 
(\ref{eq:zeta*}), we have the Jacobians,
\begin{equation}
\xymatrix{ \cJ_\hzero \ar[dr]\ar[r]^-{\hzeta_3^*}& 
           \cJ_\hzero^*\ar[d]\ar[r]^-{\hzeta_3^*}&
           \cJ_\hzero^{**} \ar@/_20pt/[ll]_-{\hzeta_3^*} \ar[dl] \\
          & \cJ_0& },
\label{eq:zeta*J}
\end{equation}
and their structures are determined by these sigma functions
respectively.

\section{The sigma function for the degenerating family of curves $X_s$}

Once we have constructed  the sigma functions
of $X_s$ and $X_{\hzero}$, the desingularization $X_{\hzero}\to X_0$,
we can investigate the behavior of sigma function under the limit
$\displaystyle{
    \lim_{s\to 0}}  y^3 = x(x-s)$ $(x-b_1)(x-b_2)$.

\subsection{Preliminaries}
In order to describe the degenerating family of curves
$X_s$, we introduce some objects.
For a real parameter $\varepsilon > 0$, $\min(|b_1|,|b_2|)>\varepsilon$,
we define
the $\varepsilon$ (punctured) disk
$$
D_\varepsilon = \{s \in \CC \ | \ |s| < \varepsilon\},
\quad
D_\varepsilon^* =D_\varepsilon \setminus \{0\},
$$
and consider the degenerating family of curves $X_s$,
$$
\fX :=\{(x,y,s) \ | \ (x,y) \in X_s, s \in D_\varepsilon \} 
$$
with the projection 
$\pi_\fX : \fX \to D_\varepsilon$.
We also consider
the trivial bundle
$$
\PP_{D_\varepsilon}=\PP \times D_\varepsilon
$$
with 
$\pi_{\PP}: \PP \times D_\varepsilon \to D_\varepsilon$
so that we define the bundle map $\pi_1: \fX \to \PP_{D_\varepsilon}$ 
which is induced from $\pi_1: X_s \to \PP$ 
($\pi_1(P)=x$).
Similarly we define
$$
\tfX :=\{(x,y,s) \ | \ (x,y) \in \tX_s, s \in D_\varepsilon \} ,
$$
and  ``symmetric products"
$$
\cS^k \fX:=\{(x,y,s) \ | \ (x,y) \in \cS^k X_s, s \in D_\varepsilon \}, \quad 
\cS^k \tfX:=\{(x,y,s) \ | \ (x,y) \in \cS^k \tX_s, s \in D_\varepsilon \}.
$$
We also have the family of Jacobians
$$
\fJ:=\{(v,s) \ |\ v \in \cJ_s, s \in D_\varepsilon^* \}.
$$
Among them, we also have the bundle maps as morphisms
induced from each fiber.

Further we consider a smooth section $P: D_\varepsilon \to \fX$
($P_s=(x_s, y_s)$ for a point $s \in D_\varepsilon$)
which satisfies the commutative diagram
\begin{equation*}
\xymatrix{ 
 X_s \ar[dr]^{\pi_1}\ar[r]^-{s\to0}& X_{0} \ar[d]^-{\pi_1} \\
  & \PP,
}\qquad
\end{equation*}
i.e., $\pi_1 (P_s) = \pi_1(P_{s'})=x \in \PP$ for $s, s' \in D_\varepsilon$.
We call it $x$-constant section of $D_\varepsilon$.

\begin{rem}{\rm{
It is noted that $x$-constant section means the flow of the differential
$D_x$ satisfies $D_x(x)=0$ in the moduli space $\fX$. 
It is remarked that the flow is unique in the sense of
quasi-isomorphism \cite{Man,EGOY}.
}}
\end{rem}

\bigskip

By using the $x$-constant section,
we will compare the sigma functions
 over $X_{s\to0}$ of $s\in D_\varepsilon^*$
and $X_\hzero$ as follows.

\subsection{Integrals for $X_s$ for $s\in D_\varepsilon^*$}
In this subsection, we consider $X_s:=\pi_\fX^{-1}(s)$ 
by fixing the parameter $s \in D_\varepsilon^*$.

Using the notation (\ref{eq:Bs}), we evaluate the integrals,
$$
\omega_{X_s ij}=\int^{B_i}_\infty \snuI{j}, \quad
\mbox{and}\quad \eta_{X_s ij}:=\int^{B_i}_\infty \snuII{j},
\quad(i=0,1,2,3, j=1,2,3).
$$
Noting $|s|<|b_1|, |b_2|$, we introduce the 
regions,
$$
V_s:=\{(x,y) \in X_s\ | 
\ |x|\le |s|\}, \quad
V_b:=\{(x,y) \in X_s\ | 
\ |x|\le \min\{|b_1|, |b_2|\}\},
$$
and $V_s^c:=X_s \setminus V_s$.
In the following expansions, we use the cubic root
 since the ambiguity of the cubic root is naturally fixed.

For the region $V_s^c$, 
we have the expansion due to the absolute convergence,
$$
 \frac{1}{\sqrt[3]{x-s}}=\frac{1}{x^{1/3}}
\left(\sum_{\ell=0}  \frac{(3\ell+1)!!!}{\ell!}
\left(\frac{1}{3}\frac{s}{x}\right)^\ell\right)
=:\frac{1}{x^{1/3}}\sum_{\ell=0}c^{(1)}_\ell 
\left(\frac{s}{x}\right)^\ell ,
$$
$$
 \frac{1}{\sqrt[3]{(x-s)^2}}=\frac{1}{x^{2/3}}
\left(\sum_{\ell=0}  \frac{(3\ell+2)!!!}{\ell!}
\left(\frac{2}{3}\frac{s}{x}\right)^\ell\right)=:
\frac{1}{x^{2/3}}\sum_{\ell=0}c^{(2)}_\ell 
\left(\frac{s}{x}\right)^\ell, \quad
$$
where $n!!! = n(n-3)(n-6)\cdots(\ell+3)\ell$ for $\ell \in \{0,1,2\}$ such that
$\ell\equiv n$ modulo 3.

\begin{lem}\label{lm:Vsc1}
In $V_s^c$,
the differentials are expressed as follows:
$$
\snuI1 = \frac{ d x}{3y^2}=\frac{d x}{3
            \sqrt[3]{x^4(x-b_1)^2(x-b_2)^2}}
\left(1+ 
\sum_{\ell=1}c^{(2)}_\ell 
\left(\frac{s}{x}\right)^\ell\right),
$$
$$
\snuI2=\frac{x d x}{3y^2}=
\frac{d x}{3\sqrt[3]{x(x-b_1)^2(x-b_2)^2}}\left(1+ \sum_{\ell=1}c^{(2)}_\ell 
\left(\frac{s}{x}\right)^\ell\right),
$$
$$
\snuI2=\frac{d x}{3y}=\frac{d x}{3
            \sqrt[3]{x^2(x-b_1)(x-b_2)}}\left(1+\sum_{\ell=1}c^{(1)}_\ell 
\left(\frac{s}{x}\right)^\ell\right),
$$
$$
\snuII1 = -\frac{(5x^2-3\lambda_3x+\lambda_2) d x}{3y}
=-\frac{(5x^2-3\lambda_3x+\lambda_2)d x}{3
            \sqrt[3]{x^2(x-b_1)(x-b_2)}}\left(1+\sum_{\ell=1}c^{(1)}_\ell 
\left(\frac{s}{x}\right)^\ell\right),
$$
$$
\snuII2 = \frac{-2x d x}{3y}
=\frac{-2x d x}{3
            \sqrt[3]{x^2(x-b_1)(x-b_2)}}\left(1+ \sum_{\ell=1}c^{(1)}_\ell 
\left(\frac{s}{x}\right)^\ell\right),
$$
$$
\snuII3 = \frac{-x^2 d x}{3y^2}
=\frac{-x d x}{3
            \sqrt[3]{x(x-b_1)^2(x-b_2)^2}}\left(1+ \sum_{\ell=1}c^{(2)}_\ell 
\left(\frac{s}{x}\right)^\ell\right).
$$
\end{lem}

Let us consider the region $V_b$ which includes $V_s$
because we have assumed that \break
$\min\{|b_1|,|b_2|\}$ $>\varepsilon> |s|$.
In this
 region $V_b$, 
$\displaystyle{
h_a(x):=\frac{1}{\sqrt[3]{(x-b_1)^a(x-b_2)^a}}
}$ is expanded as
\begin{equation}
h_a(x)
=\frac{\zeta_3^a}{\sqrt[3]{(b_1 b_2)^a}}
\left(\sum_{\ell=0} \beta^{(a)}_\ell x^\ell\right).
\label{eq:ha_exp}
\end{equation}
Further by letting $s=s_r\ee^{\ii\varphi_s}$ ($s_r \in \RR$),
the crosscuts in Figure \ref{fig:H1X3}
are illustrated in 
Figure \ref{fig:Branchline} (a).

\begin{figure}[ht]
\begin{minipage}{0.45\hsize}
\begin{center}
\includegraphics[width=0.80\textwidth]{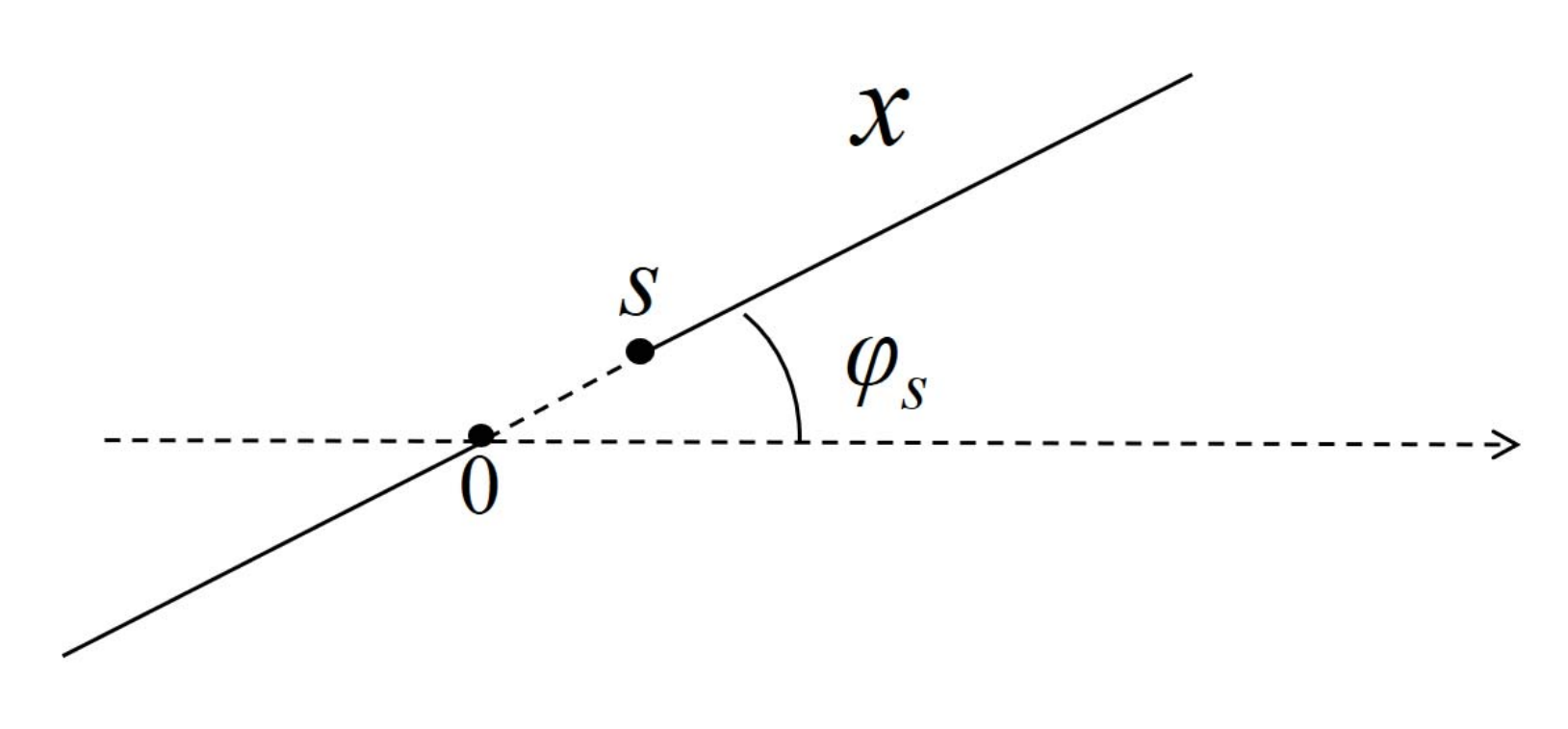}
\newline
(a)
\end{center}
\end{minipage}
\begin{minipage}{0.45\hsize}
\begin{center}
\includegraphics[width=0.80\textwidth]{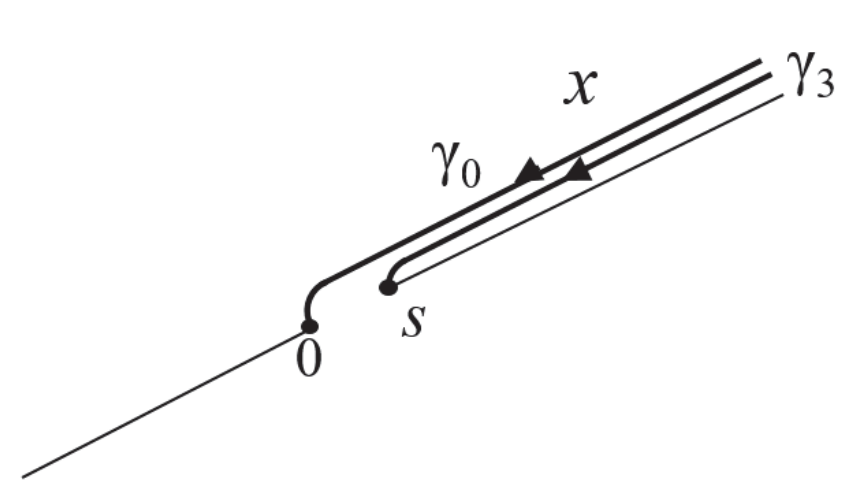}
\newline
(b)
\end{center}
\end{minipage}
\caption{Crosscuts and contours:
(a) shows the crosscuts and (b) shows the contours of the 
integrals, $\gamma_a=\gamma_{\infty,B_a}$ $(a=0,3)$.
}
\label{fig:Branchline}
\end{figure}

We have the integrals along the contours from $\infty$ to $B_a$  denoted by
$\gamma_a:=\gamma_{\infty,B_a}$; 
$\gamma_0$ and $\gamma_3$ are
 displayed in Figure \ref{fig:Branchline} (b) 
as in Appendices A and B.

The integrals $\displaystyle{\somega{i,a}=\int_{\gamma_a} \snuI{i}}$
$(i=1,2,3, a = 0, 1, 2, 3)$ are evaluated as follows.
\begin{lem}\label{lem:omegaij}
The 
matrix of integrals
 has the expression,
$$
\begin{pmatrix}
\somega{1,0}& \somega{1,1}& \somega{1,2}& \somega{1,3}\\
\somega{2,0}& \somega{2,1}& \somega{2,2}& \somega{2,3}\\
\somega{3,0}& \somega{3,1}& \somega{3,2}& \somega{3,3}
\end{pmatrix}
=
\begin{pmatrix}
A^{(0)}_{10} & A^{(0)}_{11} & A^{(0)}_{12} &  s^{-1/3}A^{(0)}_{13}\\
A^{(0)}_{20} & A^{(0)}_{21} & A^{(0)}_{22} &          A^{(0)}_{23}\\
A^{(0)}_{30} & A^{(0)}_{31} & A^{(0)}_{32} &          A^{(0)}_{33}\\
\end{pmatrix},
$$
where $A^{(0)}_{ij}:=A^{(0)}_{ij}(s^{1/3})$
and $A^{(0)}_{ij}(t)$ are holomorphic functions with respect to $t=s^{1/3}$
such that 
$$
A^{(0)}_{13}(s=0) \neq 0.
$$
\end{lem}

\begin{proof}
When we compute $\somega{a,b}$ $(a=1,2,3, b=1,2)$, 
we apply  Lemma \ref{lm:Vsc1}
since we can regard
 $\gamma_a$ $(a=1,2)$ as a subset of $ V_s^c$.
On the other hand the cases of $a=0$ and $a=3$ 
 are in 
Lemmas \ref{lmm:B2} and \ref{lmm:B3} 
in Appendix B
following Aomoto's investigation in Appendix A.
\end{proof}

\subsection{Periods for the degeneration}

The following Lemmas 
\ref{lm:omega} and \ref{lm:omega-1} are corollaries of Lemma \ref{lem:omegaij}
by assuming $s\in D_\varepsilon^*$.

\begin{lem}\label{lm:omega}
The matrices of the half-periods have the expression,
$$
\omega'_{X_s}=
\begin{pmatrix}
         A^{(1)}_{11} & A^{(1)}_{12} &   s^{-1/3}A^{(1)}_{13}\\
         A^{(1)}_{21} & A^{(1)}_{22} &  A^{(1)}_{23}\\
         A^{(1)}_{31} & A^{(1)}_{32} &  A^{(1)}_{33}\\
\end{pmatrix}, \quad
\omega''_{X_s}=
\begin{pmatrix}
         s^{-1/3}A^{(2)}_{11} & A^{(2)}_{12} &  s^{-1/3}A^{(2)}_{13}\\
         A^{(2)}_{21} & A^{(2)}_{22} &  A^{(2)}_{23}\\
         A^{(2)}_{31} & A^{(2)}_{32} &  A^{(2)}_{33}\\
\end{pmatrix},
$$
where $A^{(a)}_{ij}:=A^{(a)}_{ij}(s^{1/3})$
and $A^{(a)}_{ij}(t)$ are holomorphic functions with respect to $t=s^{1/3}$
such that $A^{(a)}_{13}(0)$ $(a=1,2)$ and $A^{(2)}_{11}(0)$ 
are not zero.
\end{lem}
\begin{proof}
In the computation, we note that the contours of $\gamma_a$ $(a=0,3)$
shown in Figure \ref{fig:H1X3} are given as Figure \ref{fig:Branchline}.
Due to Lemmas \ref{eq:somegapp_p},
\ref{lm:Vsc1}, \ref{lem:omegaij}, and the fact that
 $\gamma_a$ $(a=1,2)$ is a subset of $ V_s^c$, we have the above results.
Especially from (\ref{eq:A_1}), we have
$$
A^{(1)}_{13}=\frac{\zeta_3^2}{\sqrt[3]{(b_1 b_2)^2}}
\sum_{\ell=0} \beta^{(2)}_\ell s^\ell
\frac{\Gamma(\ell+\frac{1}{3})\Gamma(\frac{1}{3})}
{\Gamma(\ell +\frac{2}{3})}
$$
using the notations in (\ref{eq:A_1}).
\end{proof}

\begin{lem}\label{lm:omega-1}
The determinant $|\omega^{\prime}_{X_s}|= s^{-1/3}A(s^{1/3})$,
$$
\omega^{\prime-1}_{X_s}=
\begin{pmatrix}
s^{1/3}A^{(3)}_{11} &        A^{(3)}_{12} &  A^{(3)}_{13}\\
s^{1/3}A^{(3)}_{21} &        A^{(3)}_{22} &  A^{(3)}_{23}\\
s^{1/3}A^{(3)}_{31} & s^{1/3} A^{(3)}_{32} &  s^{1/3}A^{(3)}_{33}\\
\end{pmatrix},
 \quad
\omega^{\prime-1}_{X_s}\omega''_{X_s}=
\begin{pmatrix}
A^{(4)}_{11} &        A^{(4)}_{12} &        A^{(4)}_{13}\\
A^{(4)}_{21} &        A^{(4)}_{22} &s^{1/3} A^{(4)}_{23}\\
A^{(4)}_{31} & s^{1/3}A^{(4)}_{32} &        A^{(4)}_{33}\\
\end{pmatrix},
$$
where for $a=3, 4$, $A^{(a)}_{ij}:=A^{(a)}_{ij}(s^{1/3})$
and $A^{(a)}_{ij}(t)$ and $A(t)$ 
are holomorphic functions with respect to $t=s^{1/3}$
such that $A(0)\neq 0$.
\end{lem}

\begin{proof}
The first relation is directly derived from Lemma \ref{lm:omega}.
Thus we obtain the formula for 
$\omega^{\prime-1}_{X_s}$.
Using it, we have
$$
\omega^{\prime-1}_{X_s}\omega''_{X_s}=
\begin{pmatrix}
A^{(4)}_{11} &        A^{(4)}_{12} &        A^{(4)}_{13}\\
A^{(4)}_{21} &        A^{(4)}_{22} & A^{(4)\prime}_{23}\\
A^{(4)}_{31} & s^{1/3}A^{(4)}_{32} &        A^{(4)}_{33}\\
\end{pmatrix},
$$
and the fact that $\omega^{\prime-1}_{X_s}\omega''_{X_s}$
is a symmetric matrix leads the third result.
\end{proof}

From Lemma \ref{eq:etapp_omegap}, we have the following lemma:
\begin{lem}\label{lm:etap}
$$
\setap{}=
\begin{pmatrix}
s^{1/3}A^{(5)}_{11} &s^{1/3}A^{(5)}_{12} &s^{1/3} A^{(5)}_{13}\\
A^{(5)}_{21} &  A^{(5)}_{22} &  A^{(5)}_{23}\\
A^{(5)}_{31} &  A^{(5)}_{32} &  A^{(5)}_{33}\\
\end{pmatrix},
 \quad
\setapp{}=\begin{pmatrix}
s^{1/3}A^{(6)}_{11} &s^{1/3}A^{(6)}_{12} &s^{1/3} A^{(6)}_{13}\\
A^{(6)}_{21} &  A^{(6)}_{22} &  A^{(6)}_{23}\\
A^{(6)}_{31} &  A^{(6)}_{32} &  A^{(6)}_{33}\\
\end{pmatrix},
$$
where for $a=5, 6$, $A^{(a)}_{ij}:=A^{(a)}_{ij}(s^{1/3})$
and $A^{(a)}_{ij}(t)$ and $A(t)$ 
are holomorphic functions with respect to $t=s^{1/3}$
such that $A(0)\neq 0$.
\end{lem}

\begin{proof}
The terms coupled with $\somegap{13}$ in 
Lemma \ref{eq:etapp_omegap} should be finite due to
the Legendre relation and they  consist of $\setap{1j}$ for
$j=1,2,3$. Thus
$\setap{}$ behaves above and 
Lemma \ref{eq:somegapp_p} determines the $\setapp{}$.
\end{proof}

\subsection{The sigma function in terms of the al function}

From (\ref{eq:c_s}) and (\ref{eq:discr}),
 the discriminant $\Delta_s$ vanishes at $s\to 0$
whereas $c_s$ diverges.
\begin{lem}\label{lem:dis0}
At $s=0$, 
we have
$$
\Delta_s=(-729b_1^{10}b_2^{10})s^4(1 +d_{\ge1}(s)), \quad
c_s= \left(\frac{(2\pi)^3}{A(0)}\right)^{1/2}
\frac{s^{-1/3}}{(-729b_1^{10}b_2^{10})^{1/8}} (1 +d_{\ge1}(s)),
$$
where 
$d_{\ge n}(s)$ means formal series of $s$ 
whose degree is greater than or equal to $n$ in 
$\CC[b_1,b_2][[s]]$.
\end{lem}

The recent investigation \cite{BuL, EGOY} shows the result:
\begin{prop}\label{prop:sigma_Schur}
For every $s\in D_s$, the sigma function is expanded as 
$$
\ssigma{}(u) = u_1 -u_2^2 u_3 + 
\mathrm{\mbox{higher-order terms with respect to the weight of $u$}}
$$
at the origin of $\CC^3$.
\end{prop}

\begin{rem}\label{rmk:sigmas0}
{\rm{
Proposition \ref{prop:sigma_Schur}
 entails that the sigma function can be defined even for 
$X_0$. 
One of the reasons is that the Schur polynomial can be regarded
as a sigma function of the monomial curve 
$Y^3=X^4$ \cite{BuL, EGOY}.
Proposition \ref{prop:sigma_Schur} and
Lemma \ref{lem:dis0} imply that 
for the limit of $s\to 0$, the $c_s$ diverges
 whereas the theta function also vanishes to the order
of $s^{1/3}$ and $\sigma$ is defined.
Thus this property of the sigma function in Proposition 
\ref{prop:sigma_Schur} is very crucial.
However the proposition only shows the local behavior of
$\ssigma{}$ at a point, or in the
image of $\stw$ for $\kappa_X^{-1}(\infty,\infty,\infty)$ in $\cS^3 \tX_s$;
this does not capture the degeneration phenomenon, and
moreover,  the argument of $\nsigma{}$ is translated (cf. Subsection
\ref{sec:3.5}). We investigate instead the relation between
$\nsigma{}$ and $\ssigma{}$ directly using the result in 
Proposition \ref{prop:sigma_Schur}
 and the $\al_a^{(c)}$-function in Proposition \ref{prop:al}.
}}
\end{rem}

We  now, evaluate the sigma function, 
$\ssigma{}$ at the Abel-map
image of the branch point
$\displaystyle{\omega_s=\int_{\gamma_3} \snuI{}}$
for $\gamma_3$ in Figure \ref{fig:Branchline}
using the $\al_3^{(0)}$-function.
\begin{lem}\label{sigma(u+w)}
For $P_i=(x_i, y_i) \in X_s$ 
$$
x_i = \frac{1}{t_i^3}, \quad y_i=\frac{1}{t_i^4}(1+ d_{\ge 1}(t_i)),
$$
and $u^{(i)} =\stw{}(P_i)$ $(i=1,2)$,
we have the following relation,
$$
\ssigma{}( u^{(1)} + u^{(2)}
 + \omega_s)=-\frac{\sqrt{2}}{\sqrt[3]{b_1b_2}}s^{-1/3}t_1t_2(1+ 
d_{\ge 1}(s,t_1, t_2, t_3)).
$$
\end{lem}

\begin{proof}
For $P_3 \in X_s$ and $u^{(3)} =\stw{}(P_3)$,
let us consider 
$$
u = u^{(1)} + u^{(2)} + u^{(3)} + \omega_s,
$$
and the limit 
$u^{(3)} \to 0$ or $P_3 \to \infty$.
For the limit, 
\begin{equation}
\displaystyle{
\ssigma{}(u + \omega_s)=
\frac{\sqrt{2} \ee^{ {}^t u \varphi_{s;0}  }}{
 \sqrt[3]{s(s-b_1)(s-b_2)}}   \al_3^{(0)}(u)\ssigma{}(u)}
\label{eq:sigma(u+ws)}
\end{equation}
is reduced to the above relation because 
 a direct computation shows
$$
\al_3^{(0)}(u)= \frac{1}{t_1^2+ t_1t_2+t_2^2}\frac{1}{t_3} 
(1+ d_{\ge1}(s,t_1, t_2, t_3))
$$
and due to Proposition \ref{prop:sigma_Schur},
$$
\ssigma{}(u)=t_1t_2(t_1^2+ t_1t_2+t_2^2)t_3 (1+ d_{\ge1}(s,t_1, t_2, t_3)).
$$
Note that the $\varphi_{s;0}$ of $\ee^{ {}^t u \varphi_{s;0}}$ does not have
any effect on this evaluation.
We can arrange for the ambiguous third root of unity  to 1 by 
considering proper contours in the integral. 
\end{proof}

\bigskip

\bigskip

\subsection{The limit in $D_\varepsilon^*$ and $X_\hzero$}

Now we are ready to 
describe the behavior of $X_s$ for $s \to 0$ in 
$\pi_\fX^{-1} D_s^*$ and compare it with that of $X_\hzero$ by
using $x$-constant sections of $D_\varepsilon$.
As we consider the $x$-constant section $P: D_\varepsilon \to \fX$
which 
is denoted by $P_s=(x_s, y_s)$ for a point $s \in D_\varepsilon$,
we consider the element $P$ in $X_\hzero$ 
via the commutative diagram,
\begin{equation*}
\xymatrix{ 
 X_s \ar[dr]\ar[r]^-{s\to0}& X_{0} \ar[d]& X_\hzero \ar[l]^-{\hpi}\ar[dl] \\
  & \PP,&
}\qquad
\end{equation*}
i.e.,  $\pi_1 (P_s) =\pi_1(P_0)=\hpi_1(P)
=x \in \PP$ for $s \in D_\varepsilon$ and  
$\pi_2(P_0)=\hpi_2(P)$ of $P\in X_\hzero$.
We will fix this convention.

The following show 
that under the limit of $s \to 0$,
 some quantities of 
$\pi_\fX^{-1} D_\varepsilon^*\subset \fX$
corresponds to those of $X_{\hzero}$ using the $x$-constant sections,
whereas others don't correspond to anything.

We have the following lemmas.
\begin{lem}\label{lm:Vsc2}
For the $x$-constant section $P_s$ whose $\pi_1(P_s)=x$ 
 belongs to the region $V_s^c$
and $\pi_1(P_0)=\hpi_1(\hat P)$ of $\hat P \in X_\hzero$,
 we have the relations under the limit $s\to 0$,
$$
   \lim_{s\to0}\snuI{i+1}(P_s) = \nnuI{i}(\hat P), \quad
    \lim_{s\to0}\snuI{i+1}(P_s) = \nnuII{i}(\hat P),
\quad (i=1,2).
$$

\end{lem}

\begin{lem}\label{omegalim0}
The $A^{(1)}$'s in Lemma \ref{lm:omega} and $A^{(3)}$'s
 in Lemma \ref{lm:omega-1} satisfy the following relations under
the max-norm of the matrices,
$$
\lim_{s\to0}
\begin{pmatrix}
A^{(1)}_{21} &  A^{(1)}_{22}\\
A^{(1)}_{31} &  A^{(1)}_{32}\\
\end{pmatrix}=\omega_{X_{\hzero}}',
\quad
\lim_{s\to0}
\begin{pmatrix}
A^{(1)}_{21} &  A^{(1)}_{22}\\
A^{(1)}_{31} &  A^{(1)}_{32}\\
\end{pmatrix}
=\omega_{X_{\hzero}}'',
$$
$$
\lim_{s\to0}s^{1/3}|\omega_{X_{s}}^{\prime}|=
A^{(1)}_{13}|\omega_{X_{\hzero}}^{\prime}|, \quad
\lim_{s\to0}
\begin{pmatrix}
A^{(3)}_{12} &  A^{(3)}_{13}\\
A^{(3)}_{22} &  A^{(3)}_{23}\\
\end{pmatrix}=\omega_{X_{\hzero}}^{\prime-1}.
$$
\end{lem}

\begin{proof}
In the computation, we note that the contours 
in Figures \ref{fig:H1X3} 
and \ref{fig:Branchline}
are consistent with 
those of $X_{\hzero}$ in in Figure \ref{fig:H1X2}.
\end{proof}
This  means that
$$
\lim_{s\to0}
\somegap{i+1, j}= \nomegap{i,j},\qquad
\lim_{s\to0}
\somegapp{i+1, j}= \nomegapp{i,j}.
$$

Similarly
we have the relations for the  integrals of the second kind.
\begin{lem}\label{lm:eta}
\begin{enumerate}
\item $\displaystyle{
\lim_{s\to0}
\setap{i+1, j}= \netap{i,j}}$ and 
$\displaystyle{\lim_{s\to0}
\setapp{i+1, j}= \netapp{i,j}}$ for 
$i,j = 1, 2$, and

\item $\displaystyle{
\lim_{s\to0}
\setap{1, j}=
\lim_{s\to0}\setapp{1, j}=0}$ for $j = 1, 2, 3$.
\end{enumerate}
\end{lem}

From Lemmas \ref{lm:tauXs} and \ref{lm:tauXhz}, it is easy to 
obtain the limit of 
$\tau_{X_s}$:
\begin{lem}\label{lm:tau_s}
$$
\lim_{s\to 0} \tau_{X_s i,j}=\tau_{X_\hzero i,j}, \quad
(i, j = 1,2).
$$
\end{lem}

We show that the limit of the Riemann constant $\xi_{X_s}$ 
in (\ref{eq:RconstXs}) turns out to
be the shifted Riemann constant $\xi_{X_\hzero\fs}$ in 
(\ref{eq:Rconsts}); 
we analytically compute the limit of the  theta characteristics
$\delta_{X_s}$ is  $\delta_{X_\hzero}$; since
 the divisors at
infinity correspond, we demonstrate        
 naturalness of the shifted Riemann 
constant.

\begin{prop}\label{prop:RCs}
$$
\lim_{s\to 0}
\xi_{X_s, j}=\xi_{X_\hzero,j}-\tw_{X_\hzero,j}^o(\iota_X B_0)
=\xi_{X_\hzero\fs,j}, 
\quad (j=1,2).
$$
$$
\delta'_{X_s, j}\longrightarrow\delta'_{X_\hzero,j}, 
\quad \delta''_{X_s, j}\longrightarrow\delta''_{X_\hzero,j},
\quad (j=1,2).
$$
\end{prop}

\begin{proof}
For brevity, we omit $\iota_X$ in $\tw_{X_s,i}^o$ here.
Let us consider the limit of $s\to0$
of (\ref{eq:RconstXs}),
\begin{equation}
\xi_{X_sj}=\frac{1}{2}\tau_{X_sjj}+\sum_{i=1}^2
\int_{\alpha_i}\tw_{X_s,i}^o(P)\snuIo{j}(P)+
\int_{\alpha_3}\tw_{X_s,3}^o(P)\snuIo{j}(P)+
\tw_{X_s,j}^o(Q_j).
\label{eq:xiXs0}
\end{equation}
noting (\ref{eq:RconstXhzero}).
Let the third term in the right hand side be 
denoted by $I$.
For sufficiently small $s\in \Delta_\varepsilon^*$,
let $\tw_{X_s,i}^o(P)=\tw_{X_s,i}^o(B_0)+w_i(z)$ using the local
parameter $z$ of $x=z^3$ at $B_0$. 
Though $w_3(z)$ is a linear combination of 
$(\tw_{X_s,i}(P)-\tw_{X_s,i}(B_0))$'s, the dominant term 
of $w_3(z)$ at $B_0$
is  $\displaystyle{\int_0^P \snuI{1}}$. Since 
$\displaystyle{\snuI{1}=\frac{dz}{z^2}(1+d_{\ge1}(z))}$ 
and other entries are holomorphic
at $B_0$, the partial integration at $s=0$ gives
$$
I=\oint_{\alpha_3}(w_3(z) d w_j(z))=
-\oint_{\alpha_3}(dw_3(z)( w_j(z)+\tw_{X_s,j}^o(B_0)))
=- \tw_{X_s,j}^o(B_0).
$$
Since $\displaystyle{\oint_{\alpha_3}dw_3(z)=1}$,
the third equality implies 
$\displaystyle{I':=\oint_{\alpha_3}dw_3(z) w_j(z)\to 0}$ for
$s\to 0$, which we now show.

Let $\displaystyle{w_j(z)=\sum_{i=1}^\infty a^{(j)}_i z^i}$ noting
$w_j(0)=0$; it has non-vanishing radius of convergence. 
On the other hand, at $s=0$,
$\displaystyle{
dw_3(z)=\frac{s^{1/3}}{A_{13}^{(1)}}\snuI{1}(1+d_{\ge1}(s^{1/3}))}$.
As in Appendix A, we assume $s$ to be positive without loss of generality.
Using the notations in (\ref{eq:A_1}),
\begin{gather*}
\begin{split}
I'&=\frac{(\zeta_3^2-\zeta_3)s^{1/3}}
{A_{13}^{(1)}}
\int^s_0 \frac{h_2(x)w_j(x^{1/3})}{\sqrt[3]{|x^2(x-s)^2|}}dx
+ d_{\ge2}(s^{1/3})
\\
&=\frac{(\zeta_3^2-\zeta_3)}
{A_{13}^{(1)}}
\frac{\zeta_3^2}{\sqrt[3]{(b_1 b_2)^2}}
\sum_{\ell=0} \sum_{n=1}\beta^{(2)}_\ell a^{(j)}_n s^{\ell+n/3}
\int^1_0 t^{\ell+\frac{n+1}{3}-1} (1-t)^{\frac{1}{3}-1}dt
+d_{\ge2}(s^{1/3}).
\end{split}
\end{gather*}
Thus it is obvious that $\displaystyle{\lim_{s\to0} I'=0}$.
The other terms in (\ref{eq:xiXs0}) becomes $\xi_{X_\hzero,j}$ and thus
we show the proposition.
\end{proof}

Now let us consider the limit of $s\to 0$ of $\cJ_s$.
\begin{rem}
{\rm{
For $s\to 0$, the rank of the matrix
$\omega_{X_3}^{\prime-1}$ is reduced to two. This implies
that for $s\to 0$, we have a 2-dimensional space, hence every vector is a
combination of two vectors and
$$
\CC^3 \ni \omega^{\prime -1}_{X_s}\begin{pmatrix}
u_1 \\ u_2 \\ u_3 \end{pmatrix}
\longrightarrow 
\begin{pmatrix}
z_0 +s^{1/3}A^{(3)}_{11}u_1\\ 
z_1+s^{1/3}A^{(3)}_{21}u_1 \\ 
s^{1/3}z_2+s^{1/3}A^{(3)}_{31}u_1 \end{pmatrix}\in \CC^3,
$$
for certain $z$'s. 
However Proposition \ref{prop:sigma_Schur} and
Remark \ref{rmk:sigmas0} mean that by rescaling $u_1$ in  each component of 
 the image of $\omega^{\prime -1}_{X_s}$, 
$u_1$ 
survives 
in the expansion of the sigma function
 and contributes to the expression in Proposition \ref{prop:sigma_Schur}.
}}
\end{rem}

For $s\in D_\varepsilon^*$, let us consider the image of 
$$
    \stws{} : \cS^2 \tX_s \to \CC^3, \quad
    \stws{}(\gamma_1, \gamma_2) := \stw{}(\gamma_1, \gamma_2, \iota_X(B_s)),
$$
and consider the projections $p_{2,3}:\CC^3 \to \CC^2$ and 
$p_{2,3}^\perp:\CC^3 \to \CC$ such that 
$p_{2,3}(u_1, u_2, u_3)=(u_2, u_3)$ and
$p_{2,3}^\perp(u_1, u_2, u_3)=(u_1)$.

\begin{prop} \label{prop:stws}
The map $(p_{2,3} \stws{})$ is well-defined even for 
the limit $s\to 0$.
The image of $\lim_{s\to 0}(p_{2,3} \stws{})$ is  $\CC^2$.
\end{prop}

\begin{proof}
It is obvious that $p_{2,3} \stws{}$ agrees with $\ntw{\mathfrak{s}}$
for the limit $s\to 0$
due to Lemmas \ref{lm:Vsc2} and \ref{omegalim0}.
The Abel-Jacobi theorem for $X_\hzero$
shows that the image of $\ntw{\mathfrak{s}}$ is $\CC^2$.
\end{proof}

\bigskip

We state our main theorem.
\begin{thm}\label{thm:5.12}
For every $v={}^t(v_1, v_2) \in \CC^2$, there
are $x$-constant sections $\gamma_1$ and $\gamma_2$ 
of $\cH^0(D_\varepsilon^*,\cS^2 \tfX)$,
such that 
$$
  v= \lim_{s\to 0} p_{2,3} \stws{}(\gamma_1, \gamma_2),
$$
and $u=\stws{}(\gamma_1, \gamma_2)(=\stw{}(\gamma_1, \gamma_2)+\omega_s)$,
we have the relation,
\begin{equation}
\sigma_{X_{\hzero}}(v)=
\lim_{s\to 0}\left(
\frac{\sqrt[3]{sb_1b_2}}{\sqrt{2}}
\ssigma{}(u)\right).
\label{eq:thm}
\end{equation}
\end{thm}

\begin{proof}
The first equation is obtained by the Abel theorem for $\tX_{\hzero}$ and 
Proposition \ref{prop:stws}.
Noting  $u=\stw{}(\gamma_1, \gamma_2)+\omega_s$,
Proposition \ref{prop:sigmaX2} (v), and Lemma \ref{sigma(u+w)},
 the expansions 
of both sides at the origin in $\CC^2$  agree.
On the other hand, the translational formula of $\ssigma{}$ of $X_s|_{s=0}$
on $(u_1, u_2 + \ell_2, u_3+\ell_3)$ for 
${}^t(\ell_2, \ell_3)\in p_{2,3}\Gamma_s|_{s=0}$
agrees with that of $\nsigma{}$ on
$(v_1 + \ell_2, v_2+\ell_3)$ for ${}^t(\ell_2, \ell_3)\in 
\Gamma_\hzero=p_{2,3}\Gamma_s|_{s=0}$
because of  Lemmas 
\ref{omegalim0},
 \ref{lm:eta} and \ref{lm:tau_s} and
Propositions \ref{prop:sigmaX2} (3) and \ref{prop:RCs}.
\end{proof}

\begin{rem}\label{rmk:4.16}
{\rm{
The factor $s^{1/3}$ of the right-hand side in (\ref{eq:thm})
 means that the function is well defined 
on a triple cover
$\widetilde{D_\varepsilon^*}$ of $D_\varepsilon^*$.
This corresponds to the group action $\hzeta_3^*$ in
Lemma \ref{lem:GAc}.
Corresponding to the group action $\hzeta_3^*$, i.e.,
$X_{\hzero}^*=\hzeta_3^* X_{\hzero}$ and
$X_{\hzero}^{**}=\hzeta_3^{*2} X_{\hzero}$,
the left-hand side of $\nsigma{}$ of $X_{\hzero}$ in 
(\ref{eq:thm})
is replaced with the sigma functions 
  $\sigma_{X_{\hzero}^{*}}$ and $\sigma_{X_{\hzero}^{**}}$
 of $X_{\hzero}^*$ and $X_{\hzero}^{**}$ rather than 
$X_\hzero$.

On the other hand,
since the factor $s^{1/3}$ in the right hand side in
(\ref{eq:thm}) comes from $f'(b_s)$ in Proposition 
\ref{prop:al} and the proof of Lemma \ref{sigma(u+w)},
this action $\hzeta_3^*$ can be regarded as the action 
$\hzeta_3'$ on the triple covering of the domain of the $\al$ function 
(\ref{eq:cJ2.6}) in the limit of the degenerating family of curves
because $z$ in $X_\hzero$ and $\al_3^{(0)}$ in $X_0$
are the local parameters of the branch points $(B_{X_\hzero,0})$ and
$\pi_{al,3}^{-1}(B_{s})$. The symmetries 
 $\hzeta_3^*$ and $\hzeta_3'$ act on these
local parameters. After $B_s$ and $B_0$ collapse,
Theorem \ref{thm:5.12} shows  the identification
of $\hzeta_3'$ with $\hzeta_3^*$.

Accordingly, by introducing 
$\stws{}^{(c)}(\gamma_1, \gamma_2):= \stw{}(\gamma_1, \gamma_2)+
 \hzeta_3^c\omega_s$
for $c=0,1,2$ and using the expansions of $\al_3^{(c)}$,
we have other versions of the sigma function as the left hand side in 
(\ref{eq:thm}) under the action of $\hzeta_3'$.
We encounter three sigma functions
\begin{equation}
\sigma_{X_{\hat 0}}(v), \quad
\sigma_{X_{\hat 0}^{*}}(v), \quad
\sigma_{X_{\hat 0}^{**}}(v) \quad
\label{eq:R4.18}
\end{equation}
of three curves of genus two 
as in Subsection \ref{action_hzetap}.
}}
\end{rem}

\begin{rem}\label{rmk:4.16a}
{\rm{
In Appendix \ref{AppdxC},
we investigate the behavior of the
Weierstrass sigma function of 
the degenerating family of the elliptic curves $y(y-s)=x^3$ for $s\to 0$,
which behaves similarly to the sigma function
 for  $X_s$.
 We  make some  remarks on this comparison.

We note that
 ${}^tu\varphi_{s;0}$ in the proof of Lemma \ref{sigma(u+w)}
is decomposed to ${}^tu\varphi_{s;0}=
p_{2,3}^\perp({}^tu\varphi_{s;0})+$ $p_{2,3}({}^tu\varphi_{s;0})$
by letting 
$p_{2,3}^\perp({}^tu\varphi_{s;0}):=u_1\varphi_{s;0,1}$
and $p_{2,3}({}^tu\varphi_{s;0}):=
u_2\varphi_{s;0,2}+u_3\varphi_{s;0,3}$
and  in the genus one case, $\ee^{\varphi_0 u}$ behaves
like the sigma function of $\PP$, cf. 
Proposition \ref{prop:AppndixC} and Remark \ref{rmk:Esingfiber}.
Further the quasi-periodicity of the sigma function $\sigma(u+\omega_s)$
in the left hand side of (\ref{eq:sigma(u+ws)}) for the image of $p_{2,3}$
is determined by $\ee^{p_{2,3}({}^tu\varphi_{s;0})}$, $\al$-function
and $\sigma(u)$.
Though due to Lemma \ref{lm:eta},
$\eta_{1j}'$ and $\eta_{1j}''$ $(j=1, 2, 3)$ in $\varphi_{s;0,1}$
become zero for the limit $s\to 0$, we can rewrite 
the relation (\ref{eq:thm}) in Theorem \ref{thm:5.12} at $s=0$ as
\begin{equation}
\ssigma{}(u)=
\frac{\sqrt{2}\ee^{p_{2,3}^\perp({}^tu\varphi_{s;0})}}
{\sqrt[3]{sb_1b_2}}
\sigma_{X_{\hzero}}(p_{2,3}(u))(1+d_{>0}(s^{1/3}))
\label{eq:ssigma_exp}
\end{equation}
for $u=\stw{}(\gamma_1, \gamma_2)+\omega_s$ and $x$-constant sections 
$\gamma_1$, $\gamma_2$ $\in \cH^0(D_\varepsilon^*,\cS^2 \tfX)$,
where $d_{>0}(s^{1/3})$ means a formal power series in $s^{1/3}$ with 
certain coefficients.
For the parameter $t= \somegapI u$ ($u= \somegap{}t$), 
$\ee^{p_{2,3}^\perp({}^t t {}^t\somegap{}\varphi_{s;0})}$
is the sigma function of $\PP$
for the limit $s\to 0$ as in Remark \ref{rmk:Esingfiber}.
The expression (\ref{eq:ssigma_exp})
 is consistent with  the generalized 
theta function
\cite{CG, Igusa, FM}.

Further
from (\ref{eq:sigmaXhzero*}),
(\ref{eq:R4.18}) is regarded as the sigma function 
$\sigma_{X_{\hat 0}}$ with the actions 
$(\hzeta_3^*)^0$, 
$\hzeta_3^*$ and $(\hzeta_3^*)^2$. On the other hand,
on the normalization of the curve  the action
$\hzeta_3^*$ is represented 
by the action $\hzeta_3^*:
z \mapsto \zeta_3 z$ on the rational curve $\PP$ as in Lemma \ref{lem:GAc}
and Remark \ref{rmk:XhzeroP^2}.
Thus we can interpret (\ref{eq:R4.18}) as the sigma function of
the curve $X_{\hat 0}$ of genus two and two actions
$\hzeta_3^*$ on two rational curves $\PP$'s; their union can be viewed
as the singular fiber in this degenerating family $\fX$.
This interpretation corresponds to (\ref{eq:Clast}) 
in Remark \ref{rmk:Esingfiber}.
}}
\end{rem}

\begin{rem}
{\rm{
It should be emphasized that {\it{
Theorem \ref{thm:5.12} provides the algebraic and analytic properties of
the sigma function of the normalized curve of singular fiber in 
the degenerating family $\fX$ as the limit of $\ssigma{}$}}.
This is an essentially new step, since
the algebraic and analytic properties of the sigma function,
(e.g., addition theorem,  the Galois action,
dependence on moduli
parameters via the heat equation)
 have received
much attention \cite{EEMOP07, EEMOP08, EGOY, MP08, MP15}
for the case of (planar)
$(n,s)$ curves, including
$X_s$. 
Theorem \ref{thm:5.12} determines how these algebraic 
and analytic properties behave for $\nsigma{}$ of $X_\hzero$,
which cannot be of $(n,s)$ type (its Weierstrass semigroup is non-symmetric).
Due to Theorem \ref{thm:5.12}, it is obvious that 
some of the properties of $\ssigma{}$ survive.
For example,  the power expansion of the sigma function of 
$X_s$ is explicitly given by Nakayashiki  and
\^Onishi \cite{N,O2018};
Theorem \ref{thm:5.12} and (\ref{eq:sigma(u+ws)}) 
provide the power expansion of $\sigma_{X_{\hzero}}(v)$ explicitly.
As another example, 
Theorem \ref{thm:5.12} and (\ref{eq:sigma(u+ws)}) 
lead 
the addition theorem and $n$-division
formulae for $\ssigma{}$ \cite{O2009}
to those of $\sigma_{X_{\hzero}}$, a non-planar
curve.

Further Theorem \ref{thm:5.12} and above remarks show 
that  the Galois actions $\hzeta_3$ and $\hzeta_3'$
on our cyclic curves $X_s$
play an essential role in their degenerations.
Theorem \ref{thm:5.12} can be generalized to
the sigma functions of cyclic trigonal curves  $y^3=f(x)$ whose Weierstrass 
semigourp is generated by $(3,p)$ since we also constructed
the sigma functions of the space curves of $(3,p,q)$-type
whose Weierstrass 
semigroup is generated by $(3,p,q)$ \cite{KMP18};
our investigations in this paper 
correspond to the simplest case $(p,q)=(4,5)$.

Lastly,  in \cite{O2011},
 \^Onishi investigated the more general
Galois action on the sigma function for the trigonal curves
of the Weierstrass canonical form 
$y^3 + (\mu_2x + \mu_5)y^2 +
 (\mu_1x + \mu_4x^2 + $ $\mu_7x + \mu_{10})y + x^5 
+ \mu_3x^4 + \mu_6x^3 +
\mu_9x^2 + \mu_{12}x +\mu_{15}=0$ of genus four
rather than a
cyclic trigonal curve $X_s$;
his result is easily modified to the curve of genus three. 
 Using the results in \cite{O2011}, 
the above actions $\hzeta_3$, $\hzeta_3'$ and $\hzeta_3^*$
can be extended to more general Galois actions.
Thus our investigations could be 
generalized  for a
degenerating family of plane curves in Weierstrass canonical form
as mentioned in \cite{KM, KMP16}. 
}}
\end{rem}

\begin{rem}{\rm{
The behavior of these sigma function
 could be explicitly compared with the generalized theta function
for the same limit \cite{FM}.
}}
\end{rem}

\setcounter{section}{0} %
\renewcommand{\thesection}{\Alph{section}} 
\setcounter{equation}{0} %
\renewcommand{\theequation}{\Alph{section}.\arabic{equation}}
\setcounter{figure}{0} %
\renewcommand{\thefigure}{\Alph{section}.\arabic{figure}}
\setcounter{table}{0} %
\renewcommand{\thetable}{\Alph{section}.\arabic{table}}

\section{Appendix: Integrals by Kazuhiko Aomoto}
\label{sec:AppA}

In this appendix, we assume that $s$ is a positive sufficiently small 
real number
satisfying $0<s<\delta<\min\{|b_1|, |b_2|\}$.
Further, without loss of generality,
 the imaginary parts of $b_1$ and $b_2$ are positive, i.e.,
$Im(b_1)>s$ and $Im(b_2)>s$. 
Let 
$$
	h_2(x):=\frac{1}{\sqrt[3]{(x-b_1)^2(x-b_2)^2}}
$$
such that $\displaystyle{h_2(0)=\frac{\zeta_3^2}{\sqrt[3]{(b_1)^2(b_2)^2}}}$
and we choose  the branch so that $\sqrt[3]{x^2(x-s)^2}$ should be
positive for $[x, \infty)$. Using $h_2(x)$, we consider two integrals
for $[0, \infty)$ and $(s, \infty)$,
$$
I_1(s):=\int^\infty_0 \frac{h_2(x)}{\sqrt[3]{x^2(x-s)^2}}dx,\quad
I_2(s):=\int^\infty_s \frac{h_2(x)}{\sqrt[3]{x^2(x-s)^2}}dx.
$$
More precisely, $I_1$ is defined as the upper half part 
of the contour integral with a positive $\varepsilon\to0$,
$$
I_1(s) = \lim_{\varepsilon \to 0} \frac{1}{\xi_3-1}
\int_{\gamma_\varepsilon} \frac{h_2(x)}{\sqrt[3]{x^2(x-s)^2}}dx,
$$
which is illustrated in Figure \ref{fig:gammaepsilon}.
\begin{figure}[ht]
\begin{center}
\includegraphics[width=0.70\textwidth]{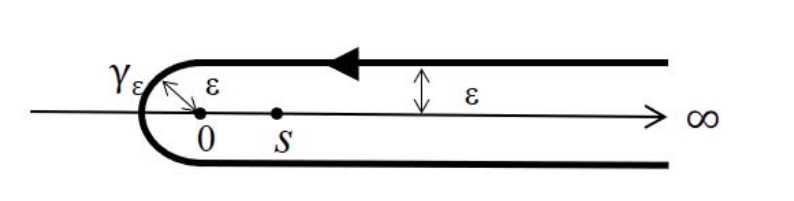}
\end{center}
\caption{The contour: 
The contour $\gamma_\varepsilon$ of the integral is illustrated.}
\label{fig:gammaepsilon}
\end{figure}

Then we have the following proposition. 
\begin{prop}
When $\min\{|b_1|,|b_2|\}>s>0$, 
$Im(b_1)>s$, and $Im(b_2)>s$, 
$$
I_1(s) = f(s), \quad I_2(s) = s^{-1/3} g(s^{1/3}),
$$
where $f(t)$ and $g(t)$ 
are regular functions with respect to $t$ in the region 
$$
U_\epsilon:=\{t \in \RR\ | 0\le t <\epsilon\},
$$
for a certain $\epsilon>0$ and $g(t)|_{t=0}\neq 0$.
\end{prop}

We prove this proposition as follows.
First we note that 
the assumptions $Im(b_1)>0$ and $Im(b_2)>0$ mean
 that $b_a$ $(a=1,2)$ does not belong to the $x$-axis.

It is obvious that $h_2(x)$ is holomorphic at $x=0$ and, as in
(\ref{eq:ha_exp}) its
expansion given by
$$
h_2(x) = \frac{\zeta_3^2}{\sqrt[3]{b_1^2b_2^2}}
 \left(\sum_{i=0}^\infty \beta^{(2)}_i x^i\right), \quad
 \beta^{(2)}_0 = 1,
$$
converges in  $V_b:=\{x \in \CC \ | \ |x|<\min\{|b_1|,|b_2|\} \}$.
There is  a  certain neighborhood $V$ of $\RR_{\ge0}:=\{x\in \RR\ |\ x\ge 0\}$
in the complex plane $\CC$ such that
$$
\Phi(x,s):= \frac{h_2(x)}{\sqrt[3]{x^2(x-s)^2}}
$$
is a holomorphic function in  $V \setminus \RR_{\ge0}$.
Using  a positive parameter $\rho$ 
($\min(Im (b_1), Im (b_2))>\rho>$
$|s|>0$),
we redefine the contour $\gamma \subset V$ 
$\gamma: (-\infty, \infty)\to \CC$, given  by 
$$
\gamma(t) = \left\{ \begin{matrix} 
  -t - \dfrac{1}{2} + \rho \ii & t \in \left(-\infty, -\dfrac{1}{2}\right],\\
 \rho \ee^{\pi(t+1)\ii} & t \in \left(-\dfrac{1}{2}, \dfrac{1}{2}\right],\\
  t - \dfrac{1}{2} + \rho \ii & t \in \left(\dfrac{1}{2}, \infty\right)
\end{matrix} \right.
$$
illustrated in Figure \ref{fig:Contour}. 
$\gamma$ is disjoint to the points $b_1$ and $b_2$ in $\CC$.
Since $\gamma$ is homotopic to $\gamma_\varepsilon$, we show the proposition 
for the integrals of $\gamma$.
\begin{figure}[ht]
\begin{center}
\includegraphics[width=0.70\textwidth]{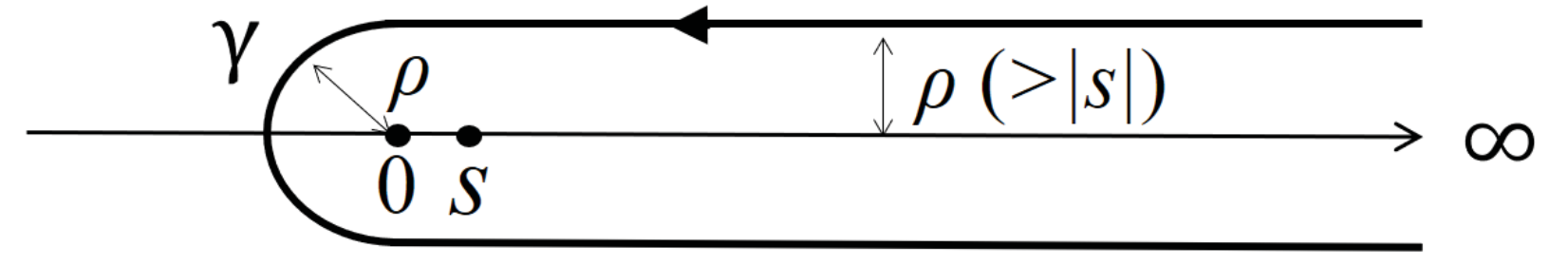}
\end{center}
\caption{The contour: 
The contour $\gamma$ of the integral is illustrated.}
\label{fig:Contour}
\end{figure}

In view of
the multi-valued property of $\Phi(x,s)$,
the Cauchy integral theorem gives
\begin{equation}
\begin{split}
\int_\gamma \Phi(x,s) dx &= 
(\zeta_3-1) \int_s^\infty
\frac{h_2(x)}{ \sqrt[3]{x^2(x-s)^2}}dx
+(-\zeta_3+1)
\int_0^s 
\frac{h_2(x)}{
|\sqrt[3]{x^2(x-s)^2}|}dx\\
&=:(\zeta_3-1) I_\gamma(s).
\label{eq:A.1}
\end{split}
\end{equation}

Since $\gamma$ is disjoint with $[0,s]$ and $\{b_1, b_2\}$
for $|s|<\rho$, $\Phi(x,s)$ is a regular function with respect to 
$(x,s) \in \gamma \times [0, \delta)$. 

For the region, we have the expansion due to the absolute convergence,
$$
 \frac{1}{\sqrt[3]{(x-s)^2}}=\frac{1}{x^{2/3}}
\left(\sum_{\ell=0}  \frac{(3\ell+2)!!!}{\ell!}
\left(\frac{2}{3}\frac{s}{x}\right)^\ell\right)
=:\frac{1}{x^{2/3}}\sum_{\ell=0}c^{(2)}_\ell 
\left(\frac{s}{x}\right)^\ell,
$$
where $(3\ell+2)!!! = (3\ell+2)(3\ell-1)\cdots5\cdot2$.
$$
\Phi(x,s)=\Phi(x,0)
\left(1+ 
\sum_{\ell=1}c^{(2)}_\ell 
\left(\frac{s}{x}\right)^\ell\right),
$$
and 
$$
\Phi(x,0) = \frac{1}{\sqrt[3]{x^4 (x-b_1)^2 (x-b_2)^2}}.
$$
\begin{lem}
$\displaystyle{\int_\gamma \Phi(x,s) dx}$ is a regular function with respect 
to $s$, i.e.,
the power expansion
$$
I_\gamma(s)=\int_\gamma \Phi(x,s) dx = \sum_{i=0} v_i s^i,
$$
has a finite radius of convergence $\delta'$, where
$$
v_0:=\int_\gamma \frac{1}{
\sqrt[3]{x^4(x-b_1)^2(x-b_2)^2}} dx.
$$
\end{lem}

Since $\gamma_\varepsilon$ is homotopic to $\gamma$, 
we conclude that $I_1(s) = I_\gamma(s)/(\zeta_3-1)$.

On the other hand, from the definition, we obtain
\begin{equation}
\begin{split}
I_1(s) &=\int^s_0 \frac{h_2(x)}{3\sqrt[3]{|x^2(x-s)^2|}}dx+I_2(s).
\end{split}
\label{eq:A.2}
\end{equation}
It means that $I_2(s)$ is represented as
the difference between the integrals over 
the generalized Lefschetz cycle
and
 the detoured
cycle,
$$
I_2(s) = A_2 - A_1,
$$
$$
A_1 = \int^s_0 \frac{h_2(x)}{\sqrt[3]{|x^2(x-s)^2|}}dx,\quad
A_2 = \frac{1}{\zeta_3-1}\ \int_\gamma \Phi(x,s) dx.
$$
By exchanging $x=st$, we have
\begin{equation}
\begin{split}
A_1&= \int^s_0 \frac{h_2(x) dx}{
\sqrt[3]{|x^2 (x-s)^2|}}
=\frac{\zeta_3^2}{\sqrt[3]{(b_1 b_2)^2}}
\sum_{\ell=0} \beta^{(2)}_\ell
\int^s_0 \frac{x^\ell dx}{
\sqrt[3]{|x^2 (x-s)^2|}}\\
&=s^{-1/3}\frac{\zeta_3^2}{\sqrt[3]{(b_1 b_2)^2}}
\sum_{\ell=0} \beta^{(2)}_\ell
\int^1_0 \frac{(st)^\ell dt}{
\sqrt[3]{|t^2 (t-1)^2|}}\\
&=s^{-1/3}\frac{\zeta_3^2}{\sqrt[3]{(b_1 b_2)^2}}
\sum_{\ell=0} \beta^{(2)}_\ell s^\ell
\int^1_0 t^{\ell+\frac{1}{3}-1} (1-t)^{\frac{1}{3}-1}dt\\
&=s^{-1/3}\frac{\zeta_3^2}{\sqrt[3]{(b_1 b_2)^2}}
\sum_{\ell=0} \beta^{(2)}_\ell s^\ell
\frac{\Gamma(\ell+\frac{1}{3})\Gamma(\frac{1}{3})}
{\Gamma(\ell +\frac{2}{3})}.
\label{eq:A_1}
\end{split}
\end{equation}
We note that the series converges absolutely because,
there is $0<c<1$ such that the ratio
\begin{equation}
\left|\beta^{(2)}_{\ell+1} 
\frac{\Gamma(\ell+\frac{4}{3})\Gamma(\frac{1}{3})}
{\Gamma(\ell +\frac{5}{3})}
\frac{1}{\beta^{(2)}_{\ell}}
\frac{\Gamma(\ell +\frac{2}{3})}
{\Gamma(\ell+\frac{1}{3})\Gamma(\frac{1}{3})}\right|
=\left|\frac{\beta^{(2)}_{\ell+1}}{\beta^{(2)}_{\ell}}
\frac{\ell+\frac{4}{3}}{\ell +\frac{5}{3}}\right| 
\label{eq:A_1conv}
\end{equation}
is smaller than $c$ for a certain $L$ and $\ell> L$.
Further it is obvious that 
the leading term is not equal to zero, i.e.,
$$
\frac{\zeta_3^2}{\sqrt[3]{(b_1 b_2)^2}}
 \beta^{(2)}_0
\frac{\Gamma(\frac{1}{3})\Gamma(\frac{1}{3})}
{\Gamma(\frac{2}{3})} \neq 0.
$$
On the other hand, the second term $A_2$ is given by
$$
\sum_{\nu=0}^\infty u_\nu s^\nu,
\quad
u_0:=\frac{1}{\zeta_3-1}\int_\gamma \frac{1}{
\sqrt[3]{x^4(x-b_1)^2(x-b_2)^2}} dx.
$$
Then we obtain
$$
I_2(s) = \sum_{\nu=0}^\infty u_\nu s^\nu-
s^{-1/3}\frac{\zeta_3^2}{\sqrt[3]{(b_1 b_2)^2}}
\sum_{\ell=0} \beta^{(2)}_\ell s^\ell
\frac{\Gamma(\ell+\frac{1}{3})\Gamma(\frac{1}{3})}
{\Gamma(\ell +\frac{2}{3})}.
$$
It means the complete proof of the proposition.

\begin{rem}{\rm{
The case when $s < 0$ can be treated similarly,
by replacing   $I_1$ and $I_2$
 with the corresponding integrals from $-\infty$ to $s$.
}}
\end{rem}

\section{Appendix: Period Integrals}\label{sec:AppB}

In this appendix, we apply the results in Appendix \ref{sec:AppA} by Aomoto
 to our integrals,
$\somega{a,b}$. 
We assume that 
$s$ is a complex number such that $|s|\ll \min\{|b_1|, |b_2|\}$.
 We employ the contours in Figure \ref{fig:Branchline}.
However 
 the phase factor $\log(s/|s|)$ is not crucial
since it correspond to rescaling the variables in the curve equation.
Noting (\ref{eq:ha_exp}),
we evaluate the integrals associated with the period matrices. 
\begin{lem}
There are  regular functions $f_a(t)$ for $t \in V_\epsilon$
such that
$$
\int^s_0 \snuI{1}= s^{-1/3} f_1(s^{1/3}), \quad
\int^s_0 \snuI{2}=  f_2(s^{1/3}), \quad
\int^s_0 \snuI{3}=  f_3(s^{1/3}), \quad
$$
and $f_1(t)|_{t=0}$ does not vanish.
\end{lem}
\begin{proof}
From (\ref{eq:A_1}), we have the first relation. 
Similarly, we also have
\begin{equation*}
\begin{split}
\int^s_0 \snuI{2}&= \int^s_0 \frac{h_2(x)x dx}{
\sqrt[3]{x^2 (x-s)^2}}
=\frac{\zeta_3^2}{\sqrt[3]{(b_1 b_2)^2}}
\sum_{\ell=0} \beta^{(2)}_\ell
\int^s_0 \frac{x^{\ell+1} dx}{
\sqrt[3]{x^2 (x-s)^2}}\\
&=s^{-1/3}\frac{\zeta_3^2}{\sqrt[3]{(b_1 b_2)^2}}
\sum_{\ell=0} \beta^{(2)}_\ell
\int^1_0 \frac{(st)^{\ell+1} dt}{
\sqrt[3]{t^2 (t-1)^2}}\\
&=s^{2/3}\frac{1}{\sqrt[3]{(b_1 b_2)^2}}
\sum_{\ell=0} \beta^{(2)}_\ell s^\ell
\int^1_0 t^{\ell+\frac{4}{3}-1} (1-t)^{\frac{1}{3}-1}dt\\
&=s^{2/3}\frac{1}{\sqrt[3]{(b_1 b_2)^2}}
\sum_{\ell=0} \beta^{(2)}_\ell s^\ell
\frac{\Gamma(\ell+\frac{4}{3})\Gamma(\frac{1}{3})}
{\Gamma(\ell +\frac{5}{3})},
\end{split}
\end{equation*}
\begin{equation*}
\begin{split}
\int^s_0 \snuI{3}&= \int^s_0 \frac{h_1(x)x dx}{
\sqrt[3]{x (x-s)}}
=\frac{\zeta_3}{\sqrt[3]{(b_1 b_2)}}
\sum_{\ell=0} \beta^{(1)}_\ell
\int^s_0 \frac{x^{\ell+1} dx}{
\sqrt[3]{x^2 (x-s)^2}}\\
&=s^{1/3}\frac{\zeta_3}{\sqrt[3]{(b_1 b_2)}}
\sum_{\ell=0} \beta^{(1)}_\ell
\int^1_0 \frac{(st)^{\ell} dt}{
\sqrt[3]{t (t-1)}}\\
&=s^{2/3}\frac{1}{\sqrt[3]{(b_1 b_2)}}
\sum_{\ell=0} \beta^{(1)}_\ell s^\ell
\int^1_0 t^{\ell+\frac{2}{3}-1} (1-t)^{\frac{2}{3}-1}dt\\
&=s^{1/3}\frac{1}{\sqrt[3]{(b_1 b_2)}}
\sum_{\ell=0} \beta^{(1)}_\ell s^\ell
\frac{\Gamma(\ell+\frac{2}{3})\Gamma(\frac{2}{3})}
{\Gamma(\ell +\frac{4}{3})}.
\end{split}
\end{equation*}
Except the prefactor $s^{\ell}$, the series absolutely converge 
as in 
(\ref{eq:A_1conv}).
\end{proof}

Thus, we consider the integrals along the contours in 
Figure \ref{fig:Branchline} (b). We have the following lemmas:
\begin{lem}\label{lmm:B2}
There are  regular functions $f'_a(t)$ for $t \in V_\epsilon$
such that
$$
\int_{\gamma_3} \snuI{1}= s^{-1/3} f'_1(s^{1/3}), \quad
\int_{\gamma_3}\snuI{2}=  f'_2(s^{1/3}), \quad
\int_{\gamma_3} \snuI{3}=  f'_3(s^{1/3}), \quad
$$
and $f'_1(t)|_{t=0}$ does not vanish.
\end{lem}

\begin{lem}\label{lmm:B3}
There are  regular functions $g_a(t)$ for $t \in V_\epsilon$
such that
$$
\int_{\gamma_0} \snuI{1}=  g_1(s^{1/3}), \quad
\int_{\gamma_0}\snuI{2}=  g_2(s^{1/3}), \quad
\int_{\gamma_0} \snuI{3}=  g_3(s^{1/3}). \quad
$$
\end{lem}


\section{Appendix: The Weierstrass sigma function over the Kodaira IV degeneration }
\label{AppdxC}

In this Appendix, we investigate
the behavior of the Weierstrass sigma function of
the degenerating family of the elliptic curves
$y(y-s) = x^3$ $s\to0$, which corresponds to type IV in the Kodaira 
classification of the degeneration of elliptic curves \cite{Ko}.
 The main result in Proposition \ref{prop:AppndixC} in this appendix
can be compared to Theorem \ref{thm:5.12}.
Within this investigation, we introduce the al-function,
which is the elliptic function version of 
the $\al$ function in Subsection \ref{sec:al34} and \cite{MP15}.
The elliptic curve has
the nongap sequence at $\infty$  determined by the numerical semigroup 
$\langle3, 2\rangle$.

\subsection{Addition formula for the sigma function of 
$E_s: y(y-s)=x^3$}

In \cite{EMO},
Eilbeck, Matsutani and \^Onishi
showed the following relation satisfied by the elliptic sigma function 
\begin{equation}
\frac{\sigma(u-v) \sigma(u-\zeta_3 v) \sigma(u-\zeta_3^2 v)}
{\sigma(u)^3 \sigma(v)^3 }= (y(u)-y(v))
\label{eq:C_sigmas}
\end{equation}
for the curve
$$
    y^2 + \mu_3 y = x^3 + \mu_6,
$$
where $\zeta_3 = \ee^{2\pi\ii/3}$.
We note that
the formula holds for the sigma function 
of this particular curve because the curve
has the symmetry of the trigonal cyclic action.
The curve corresponds to 
the Weierstrass standard form,
$$
(\wp')^2 = 4 \wp^3 -g_3=4(\wp-e_1)(\wp-e_2)(\wp-e_3),
$$
where $g_3 =-4(\mu_3^2 + \mu_6)$ and $\wp(u)' = 2y +\mu_3$.
Then we have
$$
\sigma(\zeta_3 u) = \zeta_3 \sigma(u), \quad
\wp(\zeta_3 u ) = \zeta_3 \wp(u), \quad
\wp'(\zeta_3 u) = \wp'(u)
$$
for $u \in \CC$.

In this appendix, we specify the curve $E_s$,
$$
y(y-s) = x^3, \quad \left(y-\frac{s}{2}\right)^2 = 
\left(x+\sqrt[3]{\frac{s^2}{4}}\right)
\left(x+\zeta_3\sqrt[3]{\frac{s^2}{4}}\right)
\left(x+\zeta_3^2\sqrt[3]{\frac{s^2}{4}}\right)
$$
and its limit $s \to 0$, 
which corresponds to the type IV of the
degeneration in Kodaira's classification. 
Here $\mu_3= -s$, $\mu_6=0$, 
$\displaystyle{e_j=-\zeta_3^{1-j} \sqrt[3]{\frac{s^2}{4}}}$ 
$(j=1,2,3)$ and
$g_3 = -4 s^2$.

In other words, we consider the degenerating family of 
$E_s$ for $D_\varepsilon :=\{s \in \CC \ | \ |s| < \varepsilon\}$
and $D_\varepsilon^* =D_\varepsilon^* \setminus \{0\}$,
$$
\fE:=\{(x,y,s) \ | \ (x,y) \in E_s, s\in D_\varepsilon\}
$$
and $\pi_{\fE}:\fE \to D_\varepsilon$.

\subsection{Elliptic integrals on $E_s$}\label{subsec:C.2}

We denote the integral from the point at infinity 
to $(x,y)=(0, s),$ 
 by $\omega_s$, and  that to $(0,0)$ by $\omega_0$,
$$
    \omega_s = \int^{(0,s)}_\infty du,\quad
   \omega_0 = \int^{(0,0)}_\infty du,\quad
du=\frac{dx}{2y-s}=\frac{dy}{3x^2}.
$$
Further, 
the standard half-period integrals 
are given by \cite{O1998},
\begin{equation}
\omega' =\omega_1= \int^{e_1}_\infty du,\quad
\omega_2= \zeta_3^2\omega',\quad
\omega''=\omega_3=\zeta_3\omega',
\label{eqC:omegapp}
\end{equation}
and 
\begin{equation}
\eta_i = \int_{\infty}^{e_i} x du, \quad \eta' = \eta_1, \quad
\eta''=\eta_3=\zeta_3^2\eta'.
\label{eqC:etapp}
\end{equation}
They satisfy the following
$$
\omega_1+\omega_2+\omega_3=0, \quad
\eta_1+\eta_2+\eta_3 = 0, \quad
\eta'\omega''-\eta''\omega'=\frac{\pi\ii}{2},
$$
and \cite{O1998}
\begin{equation}
 \eta'\omega'=\frac{\pi\ii}{2(\zeta_3-\zeta_3^2)}
=\frac{\pi}{2\sqrt{3}}\in \RR.
\label{eqC:etapp_omegap}
\end{equation}
Using these, we define the lattice $\Gamma_s:= \ZZ 2\omega'+ \ZZ 2\omega''$
\cite{EMO}.
Further the affine coordinates $(x,y)$ of the curve $E_s$ are 
written in terms of the $\wp$-function
using the results in \cite{EMO}:
$$
\wp(u)=x(u),\quad \wp'(u) = 2y(u)-s, \quad
y(u)=\frac{1}{2}(\wp'(u)+s).
$$
From \cite{EMO}, we have the expansions of 
$\sigma$ and $y$ for every $s\neq 0$:
\begin{lem}\label{lm:exp_sigmaE}
$$
\sigma(u) = u - \frac{1}{120}s^2 u^7 +d_{\ge11}(u),\qquad
y(u) = -\frac{1}{u^3}+\frac{1}{2}s^2 u^3+d_{\ge9}(u).
$$
\end{lem}

\begin{lem}\label{lm:y=0} 
$$
x(\zeta_3^r\omega_s)=\wp(\zeta_3^r\omega_s) = 0, (r=0,1,2),
\quad y(\omega_0) = 0, \quad 
y(\omega_s)=s.
$$
$$
\omega_0 = \frac{1-\zeta_3}{3}2\omega',\quad
\omega_s = \frac{1-\zeta_3^2}{3}2\omega'
=\frac{2+\zeta_3}{3}2\omega'=\frac{2}{3}(2 \omega'+\omega'').
$$
\end{lem}

\begin{proof}
The first three equations are obvious from \cite{EMO}.
We use the covering $\pi_2:E_s\to \PP$ 
($(x,y)\mapsto y$).
We note that the $\omega_s$ is 
the contour integral from $\infty$ to $(0,s)$
and the point $(0,s)$ is a branch point. 
Thus when we  consider the 
contour on another sheet of $\pi_2^{-1}$ as the 
return path from
$(0,s)$ to $\infty$, we obtain a period and thus
$(1-\zeta_3) \omega_s$ must be a point in the lattice $\Gamma_s$.
There exist $n$ and $m$ such that 
$$
(1-\zeta_3) \omega_s = 2 \omega' n + 2\omega'' m
=2(n +\zeta_3m)\omega'.
$$
Thus we have
$$
\omega_s=\frac{1}{3}(1-\zeta_3^2)2(n +\zeta_3m)\omega'.
$$

We fix  $\omega_s$ modulo $\Gamma_s$ and  
there are two possibilities
$$
\omega_s = \pm \frac{1}{3}(1-\zeta_3^2) 2 
\omega' \mbox{ modulo  }\Gamma_s.
$$
We find  $\omega_s = \frac{1}{3}(1-\zeta_3^2) 
2 \omega'$ numerically
 using Maple; in  the other case we have $\omega_0$.
\end{proof}

\begin{lem}
$$
\omega'=\frac{3}{2}\frac{1}{\zeta_3-\zeta_3^2}
\frac{\Gamma\left(\frac{1}{3}\right)^2}
{s^{1/3}\Gamma\left(\frac{2}{3}\right)},\quad
\eta'=\frac{\pi\ii}{3\sqrt{3}}
\frac{s^{1/3}\Gamma\left(\frac{2}{3}\right)}
{\Gamma\left(\frac{1}{3}\right)^2}.
$$
\end{lem}

\begin{proof}
$\displaystyle{\omega_s-\omega_0=\int^s_0 \frac{dy}{\sqrt[3]{y(y-s)}}}$ is 
given by $\displaystyle{\frac{\Gamma\left(\frac{1}{3}\right)^2}
{s^{1/3}\Gamma\left(\frac{2}{3}\right)}}$ by putting $h(x)=1$ in 
(\ref{eq:A_1}).
\end{proof}

Then the action of the  cyclic 3-group on $\omega_s$ and $\omega_0$
is given by the translation, which is illustrated 
in Figure \ref{fig:FD_altri_eh0}:
\begin{lem}\label{lm:C.3}
$$
\zeta_3 \omega_s \equiv \omega_s - 2 \omega', \quad
\zeta_3^2 \omega_s \equiv \omega_s - 2(\omega'+ \omega'')
\mbox{ modulo  }\Gamma_s.
$$
\end{lem}

\begin{figure}[ht]
\begin{center}
\includegraphics[width=0.70\textwidth]{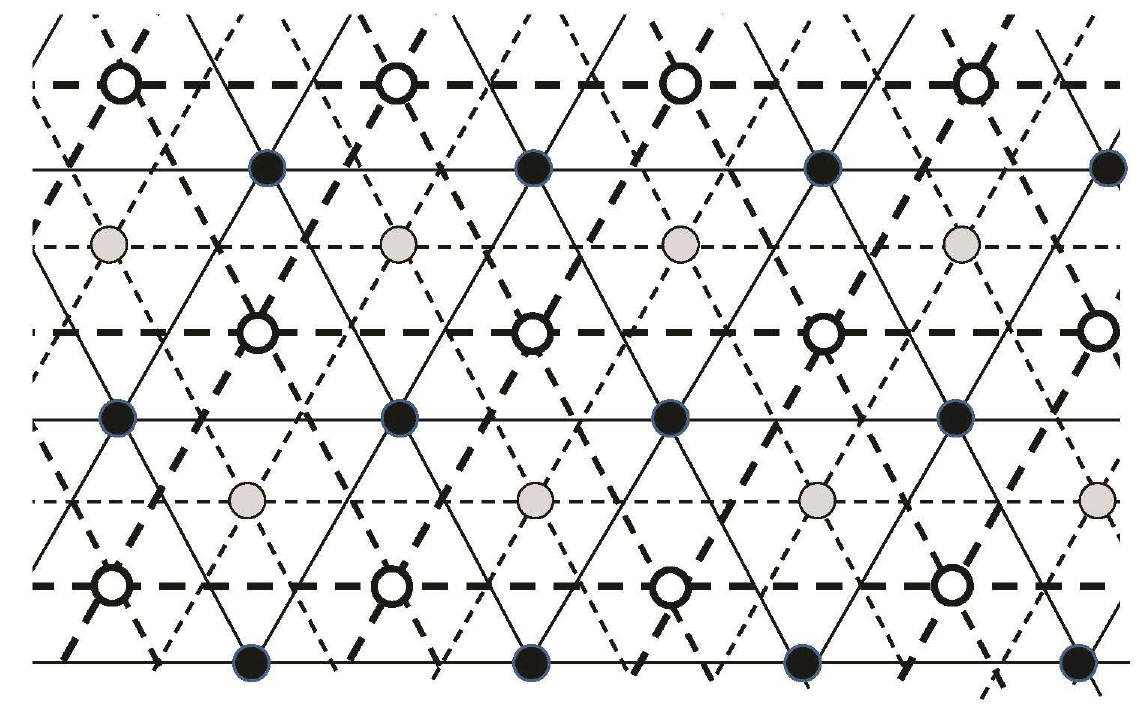}
\end{center}
\caption{
The points of $\omega_0$ and $\omega_s$ in the lattice
$\Gamma_s$:
The black dots show the points in $\Gamma_s$ whereas
the white and the gray dots mean $\omega_0$ and $\omega_s$
modulo $\Gamma_s$
respectively. They have the translational symmetry.}  
\label{fig:FD_altri_eh0}
\end{figure}

The translation formula of the sigma function is given by \cite{WW}
\begin{equation}
\sigma(u+ 2 \omega' n + 2\omega'' m) =
(-1)^{n+m+nm}\ee^{(2n\eta'+2m\eta'')
(u+n\omega'+m\omega'')}\sigma(u).
\label{eq:trans_ehe}
\end{equation}
and thus we have
\begin{equation}
\sigma(u-\zeta_3 \omega_s) =\sigma(u-\omega_s+2\omega')
= (-1)\ee^{2\eta'(u+\omega_s+\omega')}\sigma(u-\omega_s),
\label{eq:trans_ehe2}
\end{equation}
$$
\sigma(u-\zeta_3^2 \omega_s) =\sigma(u-\omega_s+2(\omega'+ \omega''))
= (-1)\ee^{2(\eta'+\eta'')
(u+\omega_s+\omega'+\omega'')}\sigma(u-\omega_s).
$$

\subsection{The al-function of $E_s$}

It is well-known that
the Jacobian $\cJ$, the fundamental domain of the $\wp$-function,
 is given by $\cJ=\CC/\Gamma_s$ but 
the fundamental domains of Jacobi's $\sn$, $\cn$, $\dn$-functions
differ from $\cJ$. In this appendix, we introduce a meromorphic function
$\al$, which is the elliptic function version of 
the $\al$ function in Subsection \ref{sec:al34}. Its domain also
differs from $\cJ$. By identifying its domain, we give a
crucial relation in Proposition \ref{prop:alEy}.

(\ref{eq:C_sigmas})  and (\ref{eq:trans_ehe2}) 
give the following lemma:
\begin{lem}\label{lmm:C.4}
$$
\frac{\sigma(u-\omega_s) \sigma(u-\zeta_3 \omega_s) 
\sigma(u-\zeta_3^2 \omega_s) }
{\sigma(u)^3 \sigma(\omega_s)^3}= y(u)- s,
$$
$$
-\frac{\sigma(u+\omega_s) \sigma(u+\zeta_3 \omega_s) 
\sigma(u+\zeta_3^2 \omega_s) }
{\sigma(u)^3 \sigma(\omega_s)^3}= y(u),
$$
$$
\frac{\ee^{2(2+\zeta_3)\eta')(u+\omega_s)
+\pi\sqrt{3}}\sigma(u-\omega_s)^3} 
{\sigma(u)^3 \sigma(\omega_s)^3}= y(u)- s.
$$
\end{lem}

\begin{proof} 
The first and the second equalities are directly obtained 
from (\ref{eq:C_sigmas}),
and we have the third one by the computation,
\begin{gather*}
\begin{split}
\sigma(u-\omega_s) \sigma(u-\zeta_3 \omega_s) 
\sigma(u-\zeta_3^2 \omega_s)
&= \ee^{2(2\eta'+\eta'')(u+\omega_s)
+2(2\eta'\omega'+\eta''\omega'')}\sigma(u+\omega_s)^3\\
&=
\ee^{2(2+\zeta_3^2)\eta')(u+\omega_s)
+6\eta'\omega'}\sigma(u-\omega_s)^3.
\end{split}
\end{gather*}
\end{proof}

Noting the relation $
\sigma(u+ \zeta_3^{\ell}\omega_s)
=\sigma(\zeta_3^{\ell}(\zeta_3^{-\ell}u+ \omega_s))
=\zeta_3^{\ell}\sigma(\zeta_3^{-\ell}u+ \omega_s)$,
we  define the al-function:

\begin{defn}\label{def:al_E}
$$
\al_r(u) := \frac{\ee^{-\varphi_r u}\sigma(u-\zeta_3^{-r}\omega_s)}
{\sigma(u)\sigma(\omega_s)}
=\frac{\ee^{-\varphi_r u}\zeta_3^{-r}
\sigma(\zeta_3^{r}u-\omega_s)}
{\sigma(u)\sigma(\omega_s)},
$$
where
$$
\varphi_0:=-\frac{2}{3}(2\eta'+\eta''), \quad
\varphi_1:=\frac{2}{3}(\eta'+2\eta''), \quad
\varphi_2:=\frac{2}{3}(\eta'-\eta''). \quad
$$
\end{defn}

Then $\al_r(u)$ is a meromorphic function of $\CC$ 
with double periods
and has the properties:
\begin{prop}\label{prop:al_rE1}
\begin{enumerate}
\item $\displaystyle{
\prod_{r = 0}^2 \al_r(u) = y - s}$,

\item for every $\ell, k \in \ZZ$, 
$$
\al_0(u+2\ell(2\omega'+\omega'')+6k \omega'') =\al_0(u),
\quad
\al_1(u+2\ell(\omega'+2\omega'')+6k \omega') =\al_1(u),
$$
$$
\al_2(u+2\ell(\omega'-\omega'')-6k (\omega'+\omega'')) =\al_2(u),
$$

\item and
$$
\varphi_0=\frac{2}{3}(2+\zeta_3^2)\eta', \quad
\varphi_1=-\frac{2}{3}(1+2\zeta_3^2)\eta', \quad
\varphi_2=-\frac{2}{3}(1-\zeta_3^2)\eta'. \quad
$$
\end{enumerate}
\end{prop}

It follows that 
the fundamental domain $\cJ_r$ of $\al_r(u)$ is given 
 in Figure \ref{fig:FD_altri_eh}.
\begin{figure}[ht]
\begin{center}
\includegraphics[width=0.80\textwidth]{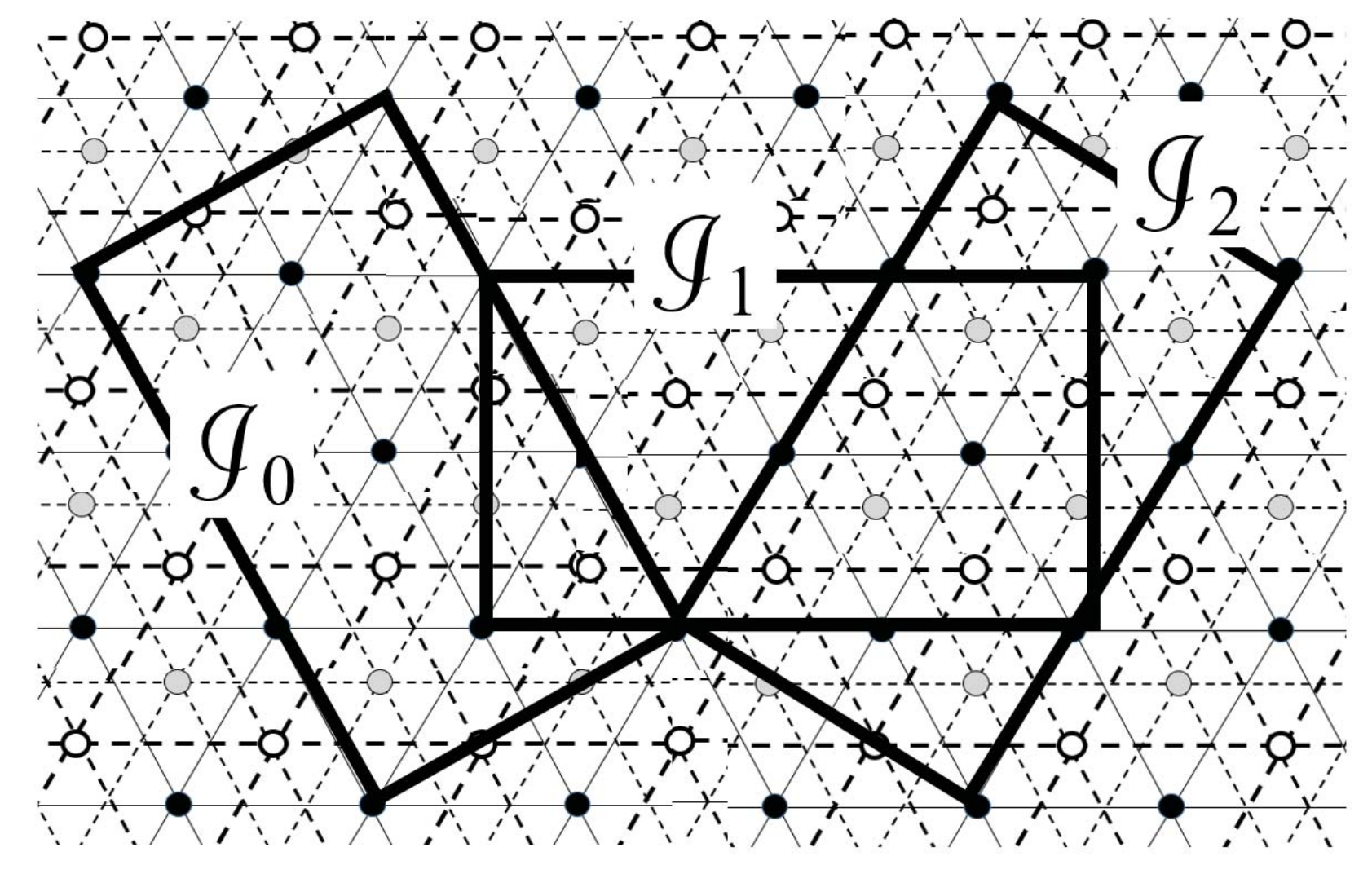}
\end{center}
\caption{
The fundamental domain $\cJ_r$ of $\al_r(u)$: 
$\cJ_r$ shows the fundamental domain of $\al_r(u)$
$(r=0,1,2)$}
\label{fig:FD_altri_eh}
\end{figure}

\begin{proof}
(1) and (3) are obvious from Lemma \ref{lmm:C.4} and properties of 
$\eta$'s in Subsection \ref{subsec:C.2} respectively. Thus we show (2).
Noting $\omega_s=\frac{1}{3}(2+\zeta_3)\omega'$ and 
$(\zeta_3\eta'-\eta'')\omega'=\frac{\pi}{2}\ii$
and letting $\ell_{m,n}:=m\omega' + n \omega''$, we have
$$
\frac{\sigma(u-\omega_s+\ell_{2m,2n})}{\sigma(u+\ell_{2m,2n})}
=\frac{\ee^{-2(m\eta'+n\eta'')\omega_s}\sigma(u-\omega_s)}
{\sigma(u)}.
$$
The factor in the right hand side is determined by
$$
2(m\eta'+n\eta'')\omega_s=
\frac{1}{3}(2+\zeta_3)\eta'(2m\omega')
+\frac{1}{3}(2+\zeta_3)\eta''(2n\omega')
$$
$$
=\frac{1}{3}(4m\eta'\omega' +m\pi\ii +2m\eta''\omega')
+\frac{1}{3}(-2n\pi\ii +4n\eta'\omega''+2n\zeta_3\eta''\omega').
$$
Then we obtain 
$$
\frac{\sigma(u-\omega_s+\ell_{4\ell,2\ell+6k})}
{\sigma(u+\ell_{4\ell,2\ell+6k})}
=\frac{\ee^{\frac{4}{3}\eta'(2\ell\omega'+(\ell+3k)\omega'')
+\frac{2}{3}\eta''(2\ell\omega'+(\ell+3k)\omega'')}\sigma(u-\omega_s)}
{\sigma(u)},
$$
which corresponds to the first relation in (2).
Similarly, since
$$
2(m\eta'+n\eta'')\zeta_3^2\omega_s=
-\frac{1}{3}(1+2\zeta_3)\eta'(2m\omega')
-\frac{1}{3}(1+2\zeta_3)\eta''(2n\omega')
$$
$$
=-\frac{1}{3}(2m\eta'\omega'+4m\eta''\omega'+2m\pi\ii)-
\frac{1}{3}(2n\zeta_3\omega'\eta'+4n\eta''\zeta_3\omega'-n\pi\ii),
$$
we have the second relation in (2) due to
$$
\frac{\sigma(u-\zeta_3^2\omega_s+\ell_{6k+2\ell,4\ell})}
{\sigma(u+\ell_{6k+2\ell,4\ell})}
=\frac{\ee^{-\frac{1}{3} \eta'((6k+2\ell)\omega'+4\ell\omega'')
-\frac{2}{3}\eta''((6k+2\ell)\omega'+4\ell\omega'')}
\sigma(u-\zeta_3^2\omega_s)}
{\sigma(u)}.
$$
The third relation in (2) is also obtained because 
the relation
$$
2(m\eta'+n\eta'')\zeta_3\omega_s=
\frac{1}{3}(\zeta_3-1)\eta'(2m\omega')
+\frac{1}{3}(\zeta_3-1)\eta''(2n\omega')
$$
$$
=\frac{1}{3}(2m\eta''\omega'-2m\eta'\omega'+m\pi\ii)+
\frac{1}{3}(2n\omega''\eta''-2n\eta'\zeta_3\omega'+n\pi\ii)
$$
shows the identity
\begin{gather*}
\begin{split}
&\frac{\sigma(u-\zeta_3\omega_s+\ell_{3k+\ell,3k-\ell})}
{\sigma(u+\ell_{6k+2\ell,6k-2\ell})}
\\
&=\frac{\ee^{-\frac{1}{3}\eta'((6k+2\ell)\omega'+(6k-2\ell)\omega'')
+\frac{1}{3}\eta''((6k+2\ell\omega'+(6k-2\ell)\omega'')}
\sigma(u-\zeta_3\omega_s)}
{\sigma(u)}.
\end{split}
\end{gather*}
\end{proof}

\begin{rem}\label{rmk:hal_r}
{\rm{
We can employ the alternative definition of 
{\lq\lq}al{\rq\rq}-function
$$
\hat\al_r(u) := \frac{\ee^{\varphi_r u}
\sigma(u+\zeta_3^{-r}\omega_s)}
{\sigma(u)\sigma(\omega_s)}
=\frac{\ee^{-\varphi_r u}\zeta_3^{-r}
\sigma(\zeta_3^{r}u+\omega_s)}
{\sigma(u)\sigma(\omega_s)}.
$$
Instead of $u$, we use $-u$ for the $\al_r$ function
and have the same relations of $\hat\al_r(u)$ as $\al_r$'s.
}}
\end{rem}

\begin{prop}\label{prop:alEy}
The $\al$ function is  expressed by
$$
\al_r(u)=\ee^{\psi_r (u)} \sqrt[3]{y(u)-s},
$$
where $\psi_r (u)$ determines the 
phase factor of the cubic root such that
 $\psi_0 + \psi_1 + \psi_2 =0$.

More rigorously, 
there is a function $z$ of $u \in \CC$ 
satisfying  $z^3 = y-s$ and 
$$
\al_r(u)=\zeta_3^r z(u).
$$
\end{prop}

\begin{proof}
Let $z^3 = y-s$. It gives the singular curve $z^3(z^3+s)=x^3$.
We normalize the curve
by $w=x/z$ and obtain the elliptic curve $\hE_s$ defined by
$$
  z^3 +s =w^3
$$
as the triple covering curve $\pi_{\hE_s}:\hE_s \to E_s$. 
This $w^3$ in the right 
hand side agrees with $y=z^3+s = w^3$.
The point $w=0$ corresponds to the point $\pi_{\hE_s}^{-1} (x=0,y=0)$
and thus there is a symmetry group $G_{\hE_s}$,
$$
(z, w) \mapsto (-w,-z)
$$
which comes from the hyperelliptic involution 
$(x, y -s/2) \to (x, -y+s/2)$. 
Further we have 
the two trigonal cyclic group actions 
$\hzeta_3$ and $\hzeta_3'$ on $\hE_s$
$$
\hzeta_3(z,w ) = (z, \zeta_3 w), \quad
\hzeta_3'(z,w ) = (\zeta_3z, \zeta_3^{-1} w), 
$$
where $\hzeta_3$ is induced from that of $E_s$.
Since $\infty$ is the fixed point of the action $\hzeta_3$,
we note that $\hE_s$ has three different infinite points 
\begin{equation}
(\zeta_3^p\infty, \zeta_3^{-p}\infty), \quad (p=0,1,2)
\label{eq:wz_inf}
\end{equation}
as $\pi_{\hE_s}^{-1}\infty$.
Further by noting $w^2 dw = z^2 dz$, we have the differential of the first kind (the holomorphic one-form)
of $\hE_s$ by the relation,
$$
\frac{dx}{2y-s}=\frac{d(wz)}{2z^3+s}
=\frac{dz}{w^2}.
$$

In order to use the results of Weierstrass elliptic function theory 
for $\hE_s$, we introduce
another curve $\hE_s'$ which is written by Weierstrass canonical form.
The $\hE_s$ is birational to the curve $\hE_s'$ defined by
$$
W^2-3\sqrt{-3}sW=Z^3, 
$$
$$
\left(W-\frac{3\sqrt{-3}}{2}s\right)^2=Z^3-\frac{27}{4}s^2
=\left(Z+3\sqrt[3]{\frac{s^2}{4}}\zeta_6\right)
\left(Z+3\sqrt[3]{\frac{s^2}{4}}\zeta_6^3\right)
\left(Z+3\sqrt[3]{\frac{s^2}{4}}\zeta_6^5\right),
$$
where $\zeta_6:=\ee^{2\pi\ii/6}=1+\zeta_3$,
$$
Z:= \frac{3s}{z-w}, \quad
W:= 3\frac{(z+w)Z+(1+2\zeta_3)}{2} 
= 9s\frac{(z+w)+(1+2\zeta_3)(z-w)}{2(z-w)},\quad
$$
or
$$
W-\frac{3\sqrt{-3}}{2}s= 9s\frac{(z+w)}{2(z-w)},\quad
   w= \frac{W-3(2+\zeta_3)}{3Z}, \quad
   z = \frac{W+3(2+\zeta_3^2)}{6Z}.
$$
$\hE_s'$ is also a trigonal covering
of $\pi_{\hE_s'}: \hE_s' \to E_s$ as the above sense.
Let $\displaystyle{\hat{e}_i=3\sqrt[3]{\frac{s^2}{4}}\zeta_6^{1-2i}}$. 
Here $Z=0$ means two infinity points in $\hE_s$,
$(\zeta_3^p\infty, \zeta_3^{-p}\infty)$, $(p=1,2)$,
whereas  $Z=W=\infty$ corresponds to 
$(w,z)=(\infty, \infty)$ of $\hE_s$.
The point $W=\dfrac{3\sqrt{-3}}{2}s$ in $\hE_s'$ corresponds to
 the point $z=-w$ in $\hE_s$, i.e., 
$w=-z=\displaystyle{\zeta_3^r
\sqrt[3]{\frac{s}{2}}}$ $(r=0,1,2)$.
Thus these points
$W=\dfrac{3\sqrt{-3}}{2}s$ and $w=-z=\displaystyle{\zeta_3^r
\sqrt[3]{\frac{s}{2}}}$ $(r=0,1,2)$
also correspond to the branch points $x=zw=e_i$ ($i=0, 1, 2$)
of $E_s$ and $\hat{e}_i$'s of $\hE_s'$.

Since  the differential of the first kind $d\hu$ of $\hE_s'$ is given by
$\displaystyle{
d\hu =\frac{dZ}{2W-3\sqrt{-3}s}}$
$\displaystyle{=-\frac{dz}{3w^2}
}$, we have 
the relation between the differentials of the first kind of
$E_s$ and $\hE_s'$, 
$$
du = -3 d\hu.
$$
 
Let us consider the half-period integrals of $\hE_s'$,
$$
\Omega_i:=\int^{\hat{e}_i}_\infty d\hu, \quad
$$
and then we obviously have the relation
$$
\Omega_i = 3 \zeta_6 \omega_i, \quad(i=0,1,2)
$$
because it can obtained by the variable change 
$Z=3x\zeta_6$ in the integral.
Hence the Jacobian $\cJ_{\hE_s'}$ of $\hE_s'$ is given by 
$$
\cJ_{\hE_s'}:=\CC/(6\zeta_6\omega'\ZZ \times 6\zeta_6 \omega''\ZZ).
$$
Then $z$ is a well-defined function of the 
Jacobian $\cJ_{\hE_s'}$,
which contains the nine points which correspond to the infinite points
 of $\hE_s$,
$$
(z,w)=(\zeta_3^p\infty, \zeta_3^{q-p}\infty), \quad (p,q=0,1,2)
$$
due to the  actions $\hzeta_3$ and $\hzeta_3'$.
However noting (\ref{eq:wz_inf}), 
$q$ should be fixed under the action of $\hzeta_3$.

We note that $\al_r$ and $z$ have the same 
poles in $\CC$.
The Jacobian $\cJ_{\hE_s'}$ contains nine 
$z=0$ points, which corresponds to 
$$
(\omega_s + 2m \omega'+ 2n \omega'')  \quad\mbox{modulo}\quad 
\cJ_{\hE_s'}.
$$
The involution $(z,w) \mapsto (-w, -z)$ 
in $E_s$ is related to 
$\sigma(-u) = -\sigma(u)$.
As in Lemma \ref{lm:y=0}, $z$ vanishes at 
$\omega_s$, $-\omega_0$, $\zeta_3\omega_s$,
 $-\zeta_3\omega_0$,  $\zeta_3^{2}\omega_s$,
and $-\zeta_3^{2}\omega_0$ 
modulo $6\zeta_6\omega'\ZZ \times 6\zeta_6 \omega''\ZZ$.

Therefore the fundamental domain $\cJ_{z}$
of $z(u)$ is smaller than $\cJ_{\hE_s}$ and 
$\cJ_{z}$ has six zeroes of $z$ and then
the cardinality of $\pi_{z}^{-1}\infty$ is six
for the covering $\pi_{z}:\cJ_z\to \cJ_{E_s}$.

Noting the relation (\ref{eq:trans_ehe2}), 
the difference among $\al_r$'s
are only the phase factor. 
At a point of $u\equiv 0$ 
modulo $\cJ_{z}$, we fix the phase factor $\psi_r$
using Lemma \ref{lm:exp_sigmaE} and $z$. 
There are six ways to fix it corresponding to 
the elements of $\pi_{z}^{-1}\infty$. 
From Proposition \ref{prop:al_rE1}, we have the result.
\end{proof}

\begin{rem}\label{rmk:hal_r2}
{\rm{
Due to Remark \ref{rmk:hal_r}, we have the similar relations,
$$
\hat\al_r(u) =\ee^{\psi_r (-u)} \sqrt[3]{y(-u)-s}=\al_r(-u).
$$
}}
\end{rem}

\subsection{Estimates on the degenerating family of curve $E_s$}

Let us consider its behavior of $\sigma$ on
$\pi_{\fE}^{-1}D_\varepsilon^*$ and its limit $s\to 0$.
More precisely, we also consider the $x$-constant section over
$D_\varepsilon$.

The section of the
line bundle on the Jacobian of each $s \in D_\varepsilon^*$
at the branch point $\omega_s$ can be evaluated by
the relation,
$$
\sigma(u+\omega_s) = \ee^{\frac{2}{3}((2+\zeta_3^2)\eta')u}
\al_0(-u) \sigma(\omega_s) \sigma(-u).
$$
In order to evaluate it,
we  compute $\sigma(\omega_s)$.
\begin{lem}
$$
\sigma(\omega_s)=
\frac{\ee^{2\sqrt{3}\pi/9}}{\sqrt[9]{-12} \sqrt[3]{s}}.
$$
\end{lem}

\begin{proof}
We note $\sigma'(0) = 1$ and due to the translational formula,
$\sigma(4\omega'+2\omega'')=0$ and
$\sigma'(4\omega'+2\omega'')
=-\ee^{(4\eta'+2\eta'')(4\omega'+2\omega'')}\sigma(0)$.
Using Kiepert's relation, we have the identity \cite{O1998},
$$
\frac{\sigma(3u)}{\sigma(u)^9}= 3\wp(u)(\wp(u)^3-12 s^2).
$$
When $u= \omega_s$, both sides vanish. 
Thus we consider
$$
\frac{\sigma(3u)}{3\wp(u)}=(\wp(u)^3-12 s^2)\sigma(u)^9
$$
and its limit $u \to \omega_s$ and then we have
$$
\frac{\sigma'(3\omega_s)}{\wp'(\omega_s)}=
\frac{-\ee^{(4\eta'+2\eta'')(4\omega'+2\omega'')}}{s}=
\sigma(\omega_s)^9(-12 s^2).
$$
Here we use $\sigma'(4\omega+2\omega'')=
-\ee^{(4\eta'+2\eta'')(4\omega'+2\omega'')}\sigma'(0)$
and $\sigma'(0)=1$, and thus we have,
$$
\frac{4(2\eta'+\eta'')(2\omega'+\omega'')}{9}
=\frac{4(2+\zeta_3^2)(2+\zeta_3)\eta'\omega'}{9}
=\frac{2\sqrt{3}\pi}{9}.
$$
\end{proof}

From Lemma \ref{lm:exp_sigmaE} and Proposition
\ref{prop:alEy}, the $\al$ function is given by 
$$
\al_0(u) =\frac{\zeta_3}{u}
(1+\frac{1}{3}su^3-\frac{5}{18}s^2 u^6+d_{\ge9}(u)).
$$
and from the Definition \ref{def:al_E}, we have the relation,
$$
\sigma(u+\omega_s)=
-\sigma(-u-\omega_s)=\sigma(-u)\sigma(\omega_s)\ee^{\varphi_r u}\al_r(-u).
$$
Using these relations, we evaluate its behavior at the branch point for 
$s\to0$.

\begin{prop} \label{prop:AppndixC}
The sigma function at the branch point is given for $s\to0$.
$$
\sigma(u+\omega_s) = -
\frac{\ee^{2\sqrt{3}\pi/9}}{\sqrt[9]{-12} \sqrt[3]{s}} 
\ee^{\frac{2}{3}(2+\zeta_3^2)\eta_0 s^{1/3} u}
(1- \frac{1}{3}su^3 - \frac{103}{360} s^2 u^6 +d_{\ge9}(u)).
$$
where $\eta_0:=\displaystyle{\frac{\pi\ii}{3\sqrt{3}}
\frac{\Gamma\left(\frac{2}{3}\right)}
{\Gamma\left(\frac{1}{3}\right)^2}}$, $(\eta'=\eta_0  s^{1/3}$).
\end{prop}
We note that this proposition corresponds to Theorem \ref{thm:5.12}.
In other words, we can compare Proposition \ref{prop:AppndixC}
and Theorem \ref{thm:5.12} in Remark \ref{rmk:4.16a}.

\begin{rem} \label{rmk:Esingfiber}
{\rm{
First we note that the Weierstrass sigma function could be defined
even for $s=0$ if we regard it as an expansion of $u$ at the point
in $\Gamma_s$ ($u=0$) as in Lemma \ref{lm:exp_sigmaE}.

The normalized curve of the singular fiber associated with
$\CC[x,y]/(x^3-y^2)$ is 
the rational curve $\PP^1$ whose affine ring
is $\CC[z]$. In this parametrization $z$, we have three 
choices $z=\zeta_3^a x/y$ of $a=0, 1, 2$, and three 
biholomorphic normalized 
rational curves
\begin{equation*}
\xymatrix{ \PP \ar[dr]\ar[r]^-{\zeta_3}&
           \PP\ar[d]\ar[r]^-{\zeta_3^2}&
           \PP \ar@/_20pt/[ll]_-{\hzeta_3^*} \ar[dl] \\
          & E_{s=0}& }.
\label{eq:EzetaP}
\end{equation*}
This corresponds to Kodaira's result in \cite{Ko}.
Instead of $u$, we introduce $t=\omega^{\prime-1} u$ or $u=\omega' t$.
The sigma function of $\PP$ could be regarded as 
$A_0\ee^{c_0 t}$ for certain 
constants $A$ and $c_0$. 
By letting
 $c_0= -\frac{\pi}{3\sqrt{3}}(2+\zeta_3^2)\pi$, 
Proposition \ref{prop:AppndixC} shows
$$
A_0\ee^{c_0 t}= \lim_{s\to 0}s^{1/3} \sigma(u+\omega_s),
$$
which corresponds to Theorem \ref{thm:5.12}.
It implies that we find
$$
A_r\ee^{c_r t}= \lim_{s\to 0}s^{1/3} \sigma(u+\zeta_3^{-r}\omega_s)
$$
for certain numbers $A_r$ and $c_r$ of $(r=1,2)$.
Corresponding to 
(\ref{eq:R4.18}), we obtain the functions on the rational $\PP$'s,
\begin{equation}
A_0 \ee^{c_0 t},  \quad A_1 \ee^{c_1 t},\quad
A_2 \ee^{c_2 t}.  
\label{eq:Clast-1}
\end{equation}
Since the action $\hzeta_3^*$ can be represented in $\PP$ 
$(\hzeta_3^*:z \mapsto \zeta_3 z)$, (\ref{eq:Clast-1}) is also
expressed by
\begin{equation}
A_0\ee^{c_0 t}, \quad 
\hzeta_3 z, \quad
\hzeta_3^2 z,
\label{eq:Clast}
\end{equation}
which  correspond to three rational curves and 
are related to  Remark \ref{rmk:4.16a}.
}}
\end{rem}


\bibliography{Reference}

\providecommand{\MR}{\relax\ifhmode\unskip\space\fi MR }
\providecommand{\MRhref}[2]{%
  \href{http://www.ams.org/mathscinet-getitem?mr=#1}{#2}
}
\providecommand{\href}[2]{#2}
\begin{thebibliography}{10}

\bibitem{atiyah}
M.F. Atiyah, \textsl{Riemann surfaces and spin structures}, Ann. Sci. \'Ecole
  Norm. Sup. \textbf{4} (1971), 47--62.

\bibitem{Baker98}
H.~F. Baker, \textsl{On the hyperelliptic sigma functions}, Amer. J. Math.
  \textbf{20} (1898), 301--384.

\bibitem{Baker97}
H.~F. Baker, \textsl{Abelian functions. abel's theorem and the allied theory of
  theta functions}, Cambridge University Press, 1995.

\bibitem{BEN}
J.~Bernatska, V.~Enolski and A.~Nakayashiki, \textsl{Sato grassmannian and
  degenerate sigma function}, arXiv:1810.01224, 2018.

\bibitem{BeL}
J.~Bernatska and D.~Leykin, \textsl{On degenerate sigma-function in genus 2},
  Glasgow Math. J \textbf{61} (2019), 169--193.

\bibitem{BB}
J.M. Borwein and P.B. Borwein, \textsl{A cubic counterpart of jacobi's identity
  and the agm}, Trans. Amer. Math. Soc. \textbf{323} (1991), 691--701.

\bibitem{BuL}
V.M. Buchstaber and D.~V. Leykin, \textsl{Solution of the problem of
  differentiation of abelian functions over parameters for families of $(n,
  s)$-curves}, Func. Anal. Appl. \textbf{42} (2008), 268--278.

\bibitem{BEL:1999}
V.M. Bukhshtaber, V.Z. \`Enol'ski\u{\i} and D.V. Le\u{\i}kin, \textsl{Rational
  analogues of abelian functions}, Funct. Anal. Appl. \textbf{33} (1999),
  83--94.

\bibitem{CG}
A.~Clebsch and P.~Gordan, \textsl{Theorie der abelschen funktionen}, Teubner,
  1866.

\bibitem{EEMOP07}
J.~C. Eilbeck, V.Z. Enol'skii, S.~Matsutani, Y.~\^Onishi, and E.~Previato,
  \textsl{Abelian functions for trigonal curves of genus three}, Int. Math.
  Res. Notices \textbf{2007} (2007), 1--38.

\bibitem{EMO}
J.~C. Eilbeck, S.~Matsutani and Y.~\^Onishi, \textsl{Addition formulae for
  abelian functions associated with specialized curves}, Phil. Trans. R. Soc. A
  \textbf{369} (2011), 1245--1263.

\bibitem{EEL}
J.C. Eilbeck, V.Z. Enolskii and D.V. Leykin, \textsl{On the kleinian
  construction of abelian functions of canonical algebraic curves}, SIDE
  III---symmetries and integrability of difference equations (Sabaudia, 1998)
  CRM Proc. Lecture Notes, (New York), Amer. Math. Soc., 2000.

\bibitem{EEMOP08}
J.C. Eilbeck, V.Z. Enol'skii, S.~Matsutani, Y.~\^Onishi, and E.~Previato,
  \textsl{Addition formulae over the jacobian pre-image of hyperelliptic
  wirtinger varieties}, J. reine angew. Math. \textbf{619} (2008), 37--48.

\bibitem{EGOY}
J.C. Eilbeck, J.~Gibbons, Y.~\^Onishi, and S.~Yasuda, \textsl{Theory of heat
  equations for sigma functions}, arXiv:1711.08395, 2017.

\bibitem{EG}
V.Z. Enolskii and T.~Grava, \textsl{Thomae type formulae for singular $z_n$
  curves}, Lett. Math. Phys. \textbf{76} (2006), 187--214.

\bibitem{FM}
Y.~Fedorov and S.~Matsutani, \textsl{in preparation}.

\bibitem{Igusa}
J.~Igusa, \textsl{Fibre systems of jacobian varieties}, Amer. J. Math.
  \textbf{78} (1956), 171--199.

\bibitem{Klein86}
F.~Klein, \textsl{Ueber hyperelliptische sigmafunctionen}, Math. Ann.
  \textbf{27} (1886), 431--464.

\bibitem{Ko}
K.~Kodaira, \textsl{On compact complex analytic structure ii}, Ann. Math.
  \textbf{77} (1963), 563--626.

\bibitem{KM}
J.~Komeda and S.~Matsutani, \textsl{Jacobi inversion formulae for a curve in
  weierstrass normal form}, Integrable Systems and Algebraic Geometry, Vol 2:
  Algebraic Geometry edited by R. Donagi, T. Shaska LMS Lecture Notes Series,,
  Cambridge Univ. Press, 2020.

\bibitem{KMP16}
J.~Komeda, S.~Matsutani and E.~Previato, \textsl{The riemann constant for a
  non-symmetric weierstrass semigroup}, Arch. Math. (Basel) \textbf{107}
  (2016), 499--509.

\bibitem{KMP18}
J.~Komeda, S.~Matsutani and E.~Previato, \textsl{The sigma function for
  trigonal cyclic curves}, Lett. Math. Phys. \textbf{109} (2019), 423--447.

\bibitem{Lew}
J.~Lewittes, \textsl{Riemann surfaces and the theta functions}, Acta Math.
  \textbf{111} (1964), 37--61.

\bibitem{Man}
Ju.~I. Manin, \textsl{Algebraic curves over fields with differentiation}, A. M.
  S. Transl., Twenty two papers on algebra, number theory and differential
  geometry \textbf{37} (1964), 59--78.

\bibitem{MK}
S.~Matsutani and J.~Komeda, \textsl{Sigma functions for a space curve of type
  (3,4,5)}, J. Geom. Symm. Phys. \textbf{30} (2013), 75--91.

\bibitem{MP08}
S.~Matsutani and E.~Previato, \textsl{Jacobi inversion on strata of the
  jacobian of the $c_{rs}$ curve $y^r = f(x)$}, J. Math. Soc. Jpn. \textbf{60}
  (2008), 1009--1044.

\bibitem{MP15}
S.~Matsutani and E.~Previato, \textsl{The al function of a cyclic trigonal
  curve of genus three}, Coll. Math. \textbf{66} (2015), 311--349.

\bibitem{mumford}
D.~Mumford, \textsl{Theta characteristics of an algebraic curve}, Ann. Sci.
  \'Ecole Norm. Sup \textbf{4} (1971), 181--192.

\bibitem{N}
A.~Nakayashiki, \textsl{On algebraic expansions of sigma functions for $(n, s)$
  curves}, Asian J. Math. \textbf{14} (2010), 175--212.

\bibitem{O1998}
Y.~\^Onishi, \textsl{Complex multiplication formulae for hyperelliptic curves
  of genus three}, Tokyo J. Math. \textbf{21} (1998), 381--431.

\bibitem{O2009}
Y.~\^Onishi, \textsl{Abelian functions for trigonal curves of degree four and
  determinantal formulae in purely trigonal case}, Int. J. Math. \textbf{20}
  (2009), 427--441.

\bibitem{O2011}
Y.~\^Onishi, \textsl{Determinant formulae in abelian functions for a general
  trigonal curve of degree five}, Comp. Methods Func. Theory \textbf{11}
  (2011), 547--574.

\bibitem{O2018}
Y.~\^Onishi, \textsl{Arithmetical power series expansion of the sigma function
  for a plane curve}, Proc. Edinburgh Math. Soc. \textbf{61} (2018), 995--1022.

\bibitem{vH}
M.~van Hoeij, \textsl{An algorithm for computing the weierstrass normal form},
  Proceeding ISSAC '95 Proceedings of the 1995 international symposium on
  Symbolic and algebraic computation, ACM New York, 1995, pp.~90--95.

\bibitem{Wei54}
K.~Weierstrass, \textsl{Zur theorie der abelschen functionen}, J. Reine Angew.
  Math. \textbf{47} (1854), 289--306.

\bibitem{WW}
E.T. Whittaker and G.N. Watson, \textsl{A course of modern analysis}, Cambridge
  University Press, 1927.

\end{thebibliography}
\bibliographystyle{Refart}

\bigskip
\bigskip
\noindent
Yuri N. Fedorov\\
Department of Mathematics,\\
Polytechnic University of Catalonia,\\
Barcelona, 08034 SPAIN\\
\\
\noindent
Jiryo Komeda:\\
Department of Mathematics,\\
Center for Basic Education and Integrated Learning,\\
Kanagawa Institute of Technology,\\
1030 Shimo-Ogino, Atsugi, Kanagawa 243-0292, JAPAN.\\
\\
\noindent
Shigeki Matsutani:\\
Faculty of Electrical, Information and Communication Engineering,\\
Kanazawa University\\
Kakuma Kanazawa, 920-1192, JAPAN\\
e-mail: s-matsutani@se.kanazawa-u.ac.jp\\
\\
\noindent
Emma Previato:\\
Department of Mathematics and Statistics,\\
Boston University,\\
Boston, MA 02215-2411, U.S.A.\\
\\
\noindent
Kazuhiko Aomoto:\\
Tenpaku-ku, Nagoya, 468-0015, JAPAN

\end{document}